\documentclass[a4paper,12pt]{article}

\usepackage{amsmath,amsthm,amssymb,enumerate}
\usepackage{xcolor}
\usepackage{float}
\usepackage{multirow}
\usepackage{longtable}
\usepackage{diagbox}
\usepackage{soul}
\usepackage{cancel}
\usepackage[affil-it]{authblk}
\usepackage{gensymb}
\usepackage{enumitem}
\usepackage[square,sort&compress,comma,numbers]{natbib}
\usepackage{bm}
\usepackage{hyperref}
\usepackage{subfig}
\usepackage{graphicx}
\usepackage[margin=1.7cm]{geometry}
\usepackage{tikz}
\usepackage{esint}
\usepackage{caption}
\usetikzlibrary{shapes,calc}
\usepackage{verbatim}
\usepackage{array}
\usepackage{bm}
\definecolor{maroon}{RGB}{144,0,32}
\newcolumntype{H}{>{\setbox0=\hbox\bgroup}c<{\egroup}@{}}

\usepackage{newtxtext,newtxmath}

\usepackage{multirow}
\usepackage{marginnote}
\chardef\bslash=`\\ 

\newtheorem{thm}{Theorem}[section]

\newtheorem{example}[thm]{Example}
\newtheorem{lem}[thm]{Lemma}

\newtheorem{rem}[thm]{Remark}
\newtheorem{defn}[thm]{Definition}
\theoremstyle{definition}
\theoremstyle{remark}

\numberwithin{equation}{section}

\newcommand{\bI}{\boldsymbol{I}}

\newcommand{\bT}{\mathbb T}

\newcommand{\cE}{\mathcal E}

\newcommand{\cT}{\mathcal{T}}

\newcommand{\cP}{{\rm P}}
\newcommand{\cM}{{\rm M}}

\newcommand{\hto}{H^2_0(\Omega)}

\newcommand{\vket}{von K\'{a}rm\'{a}n equations}
\newcommand{\integ}{\int_\Omega}

\newcommand{\fl}{\quad \text{for all}\:}

\newcommand{\half}{\frac{1}{2}}

\newcommand{\dx}{{\rm\,dx}}
\newcommand{\ds}{{\rm\,ds}}

\newcommand{\dg}{{\rm dG}}

\newcommand{\Poincare}{Poincar\'e }
\newcommand{\Holder}{H\"{o}lder~}

\newcommand{\smooth}{JI_{\rm M}}

{\algorithm}%
{\endalgorithm}
\newcommand{\lamw}{\Lambda_{\rm W}}

\theoremstyle{definition}

\numberwithin{equation}{section}

\newcommand{\bV}{\text{\bf V}}

\newcommand{\bz}{\boldsymbol{z}}

\newcommand{\bv}{\boldsymbol{v}}

\newcommand{\bL}{\boldsymbol{L}}

\newcommand{\St}{P_2(\cT)}

\newcommand{\Stb}{{\boldsymbol{P}_2}(\cT)}

\newcommand{\ip}{{\rm IP}}

\newcommand{\trinl}{\ensuremath{|\!|\!|}}
\newcommand{\trinr}{\ensuremath{|\!|\!|}}

\newcommand{\dy}{{\rm\,dy}}

\DeclareMathOperator{\E}{\mathcal{E}}
\newcommand{\M}{\mathrm{M}}

\newcommand{\T}{\mathcal{T}}

\newcommand{\nonlin}{b}

\newcommand{\dislin}{a_h}
\newcommand{\rop}{R}
\newcommand{\sop}{S}
\newcommand{\lamr}{\Lambda_{\rm R}}
\newcommand{\lams}{\Lambda_{\rm S}}

\def \R{{{\Bbb R}}}

\allowdisplaybreaks

\def\R{\mathbb{R}}

\def\O{\Omega}

\def\bv{{\mathbf V}}

\def\pw{\rm {pw}}

\def\cE{{\mathcal{E}}}

\def\jump#1{\left[\hskip -3.5pt\left[#1\right]\hskip -3.5pt\right]}
\def\bchi{\boldsymbol{\chi}}

\def\bxi{\boldsymbol{\xi}}

\newcommand{\be}{\begin{equation}}
\newcommand{\ee}{\end{equation}}

\usepackage{mathtools}  
\mathtoolsset{showonlyrefs} 
\usepackage[normalem]{ulem}
\definecolor{violet}{rgb}{0.580,0.,0.827}

\usepackage[normalem]{ulem}
\newcounter{corr}
\definecolor{violet}{rgb}{0.580,0.,0.827}
\newcommand{\corr}[3]{\typeout{Warning : a correction remains in page
		\thepage}
	\stepcounter{corr}        
	{\color{red}\ifmmode\text{\,{\ensuremath{#1}}\,}\else{#1}\fi}
	{#2}
	{\color{violet} #3}}

\newcounter{changeto}
\newcommand{\changeto}[2]{\typeout{Warning : a correction remains in page
		\thepage}
	\stepcounter{changeto}        
	{\ifmmode\text{\,\sout{\ensuremath{#1}}\,}\else\sout{#1}\fi}
	{\color{red}#2}}

\title{Unified a priori analysis of four second-order FEM \\{for fourth-order quadratic semilinear problems}}
\author{Carsten Carstensen,\footnote{Department of Mathematics, 
		Humboldt-Universit\"{a}t zu Berlin, 10099 Berlin, Germany,  Distinguished Visiting Professor, Department of Mathematics, Indian Institute of 
Technology Bombay, Powai, Mumbai 400076, India.  cc@math.hu-berlin.de}        
		\quad
	Neela Nataraj,\footnote{Department of Mathematics, Indian Institute of Technology Bombay, Powai, Mumbai 400076, India. neela@math.iitb.ac.in}
	\quad
		Gopikrishnan C. Remesan,\footnote{Department of Mathematics, Università degli Studi di Milano, Milan, Italy 21022. gopikrishnan.chirappurathu@unimi.it}\quad
		Devika Shylaja \footnote{Department of Mathematics, Indian Institute of Space Science and Technology, Thiruvanathapuram 695547, India. devikas.pdf@iist.ac.in}
	}

\date{}
\begin{document}
	\maketitle
\abstract A unified  framework  for fourth-order semilinear problems with trilinear nonlinearity and { general sources} allows for quasi-best approximation with lowest-order finite element methods. This paper establishes the stability and a priori error control in the {piecewise} energy and weaker Sobolev norms  under minimal hypotheses.   Applications include the stream function vorticity formulation of the incompressible 2D Navier-Stokes equations and the von K\'{a}rm\'{a}n equations with Morley, discontinuous Galerkin, $C^0$ interior penalty, and weakly over-penalized symmetric interior penalty schemes. The proposed new discretizations consider  quasi-optimal smoothers for the source term and smoother-type modifications inside the nonlinear terms. 

	


	\section{Introduction}\label{sec:introduction}
The abstract framework for fourth-order semilinear elliptic problems with trilinear nonlinearity in this paper allows a source term  { $F \in H^{-2}(\O)$}  in a bounded polygonal Lipschitz domain $\O$.  {It} simultaneously applies to {the} Morley finite element method (FEM) \cite{Ciarlet, CCGMNN_semilinear}, the discontinuous Galerkin (dG) FEM \cite{FengKarakashian:2007:CahnHilliard},  {the} $C^0$ interior penalty ($C^0$IP) method \cite{BS05}, and {the} weakly over-penalized symmetric interior penalty (WOPSIP) scheme \cite{BrenGudiSung10} for the  approximation of {a} regular  solution to a fourth-order semilinear problem with the biharmonic operator as the leading term.  
In comparison to \cite{CCGMNN_semilinear}, this article includes dG/$C^0$IP/WOPSIP schemes and more general source terms that allow single forces.  It thereby continues \cite{ccnnlower2022} for the linear biharmonic equation to semilinear problems and, for the {\em first} time,  establishes quasi-best approximation results for a discretisation by the Morley/\dg/$C^0$IP schemes with smoother-type modifications in the nonlinearities.

\medskip
A general source term {$F\in H^{-2}(\Omega)$} cannot  be immediately  evaluated at a possibly discontinuous test function $v_h \in V_h \not\subset H^2_0(\Omega)$  for
the nonconforming FEMs  of this paper.  The post-processing procedure in \cite{BS05} enables a new $C^0$IP method for right-hand sides in {$H^{-2}(\Omega)$}.  The articles~\cite{veeser1,veeser2,veeser3} employ a map $Q${,} referred to as a smoother{, that}  transforms a nonsmooth function $y_h$ to a smooth version $Qy_h${.}\:The  discrete schemes  are modified by replacing $F$ with $F\circ Q$ and the quasi-best approximation follows for Morley  and $C^0$IP schemes for linear problems in the energy norm. 
The quasi-optimal smoother $Q=JI_{\rm M}$ {in~\cite{ccnnlower2022}} for dG schemes{ is based on a } (generalised) Morley interpolation operator $I_\M$ and a companion operator $J$ from~\cite{aCCP,DG_Morley_Eigen}.  

\medskip \noindent 
 In addition to the smoother $Q$ in the right-hand side, this article introduces operators $R, S \in \{{\rm id}, I_\M, JI_\M \}$  in the trilinear form $\Gamma_{\pw}(Ru_h, Ru_h, Sv_h)$  that lead to {\it nine} new discretizations for each of the {four} discretization schemes (Morley/dG/$C^0$IP/WOPSIP) in two applications.   Here $R,S ={\rm id}$  means no smoother,  $I_\M$ is averaging in the Morley finite element space,  while $J I_\M$ is the quasi-optimal smoother.
 The simultaneous analysis applies to  the stream function vorticity formulation of the 2D Navier-Stokes equations \cite{BrezziRappazRaviart80, CN86,CN89} and \vket\,\cite{CiarletPlates,ng2} defined on a bounded polygonal Lipschitz domain $\Omega$ in the plane.  For $S=JI_\M$ and all $R \in \{{\rm id}, I_\M, JI_\M \}$, the Morley/dG/$C^0$IP schemes allow for the quasi-best approximation
\begin{equation}\label{eqn:best}
 \|u - u_h \|_{\widehat{X}} \le C_{\rm qo} \min_{x_h \in X_h} \|u - x_{h}\|_{\widehat{X}}.
\end{equation}
Duality arguments lead to optimal convergence rates in weaker Sobolev norm estimates for the discrete schemes with specific choices of $R$ in the trilinear form summarised in Table~\ref{tab:apriori}. The comparison results  suggest that,  amongst the lowest-order methods for fourth-order semilinear problems with trilinear nonlinearity,  the attractive Morley FEM is the {\it simplest}  discretization scheme  with optimal error estimates in (piecewise) energy and weaker Sobolev norms.

\medskip
\noindent For $F \in  H^{{-r}}(\Omega)$ with {{$2-\sigma \le r \le 2$
(with the index of elliptic regularity $\sigma _{\rm reg}>0$ and $\sigma:=\min\{\sigma_{\rm reg},1\}>0$ )}} and for the biharmonic, the 2D Navier-Stokes, and the \vket~with homogeneous Dirichlet boundary conditions, it is known   that the exact solution belongs to  $H^2_0(\Omega)\cap H^{4-r} (\Omega)$. 
	\vspace{0.05cm}

{
\begin{table}[h!]
{ 
\centering
	\begin{tabular}{|c|c|c|cc}
		\cline{1-3}
	\diagbox[innerleftsep=0.2cm,innerrightsep=1pt]{Method}{Results}	& quasi-best for $S=JI_{\rm M}$       &  $\|u-u_h\|_{{H^s(\mathcal{T})}}$        &  &  \\ \cline{1-3}
		Morley      & \multirow{2}{*}{ \eqref{eqn:best}} &       ${ O(h_{\rm max}^{\min\{4-2r,4-r-s\}})}$            &  &  \\ \cline{1-1} \cline{3-3}
		dG/$C^0$IP &                   & \multirow{2}{*}{
\begin{minipage}{6.2cm}
\centering
\vspace{0.1cm}
$O(h_{\rm max}^{2-r})  \text{ for }  R= {\rm id},$ \\
$O(h_{\rm max}^{\min\{4-2r,4-r-s\}}) \text{ for } {R\in \{I_\M, JI_\M\}}$ 
\end{minipage}
} &  &  \\ \cline{1-2}
		WOPSIP      &        \begin{minipage}{4cm} \centering perturbed \\
   Theorem~\ref{error control}.a \&  \ref{thm:aprio_weak_vke_WOPSIP}.a
  \end{minipage}&                   &  &  \\ \cline{1-3}
	\end{tabular}
 \captionsetup{width=0.85\textwidth}
{\small \caption{\small Summary  for Navier-Stokes and von K\'{a}rm\'{a}n eqn from Section~\ref{sec:nse} and~\ref{sec:vke} with 
 {{$F \in H^{-r}(\O)$ for $2-\sigma \le r,s \le 2$}} and $R,S \in \{{\rm id}, I_\M, JI_\M \}$ arbitrary unless otherwise specified.  \label{tab:apriori}
}}  }
\end{table}
}


\noindent
\medskip
\noindent
{\bf Organisation. }\:The remaining parts are organised as follows.\:Section~\ref{sec.stability} discusses an abstract discrete inf-sup condition for linearised problems.\:Section~\ref{sec:mainresults} {{introduces}} the main results {\bf(A)}-{\bf (C)} of this article. Section~\ref{sub:existence} discusses the quadratic convergence of Newton's scheme and the unique existence  of a local discrete solution $u_h$ that approximates  a regular root {{$u \in H^2_0(\Omega)$ for data $F \in H^{-2}(\O)$}}. Section \ref{sec:apriori} presents an abstract {a priori} error control in the piecewise energy norm with  a quasi-best approximation for $S = \smooth$ in \eqref{eqn:best}.\;Section~\ref{sec:goal-oriented} discusses the goal-oriented error control and derives an a priori error estimate in weaker Sobolev norms.\:There are at least two reasons for this abstract framework enfolded in Section~\ref{sec.stability}-\ref{sec:goal-oriented}.  First it minimizes the repetition of mathematical arguments in two important applications and four popular discrete schemes. Second, it provides a platform for further generalizations to more general smooth semilinear problems as it derives all the necessities for the leading terms in the Taylor expansion of a smooth semilinearity.\:Section \ref{sec:notations} presents preliminiaries, triangulations, discrete spaces,  the conforming companion, discrete norms and some auxiliary results on $I_\M$ and $J$. Section~\ref{sec:nse} and~\ref{sec:vke} apply  the abstract results to the stream function vorticity formulation of the 2D Navier-Stokes  and the \vket~for the Morley/dG/$C^0$IP/WOPSIP approximations. They contain
 comparison results and  convergence rates displayed in Table~\ref{tab:apriori}. 
\section{Stability}\label{sec.stability}
%

This section establishes an abstract discrete inf-sup condition  under the assumptions \eqref{cts_infsup}-\eqref{quasioptimalsmootherP},   \eqref{quasioptimalsmootherR}, \eqref{eqn:defn_of_beta0} 
	and~\ref{h1}-\ref{h4} stated below. This is a key step and has consequences for second-order elliptic problems (as in \cite[Section 2]{CCGMNN_semilinear}) and in this paper for the well-posedness of the discretization.  In comparison to \cite{CCGMNN_semilinear} that merely addresses nonconforming FEM,  the proof of the stability in this section applies to all the discrete schemes.
 Let $\widehat{X}$ (resp. ~$\widehat{Y}$) be a real Banach space with norm $\|\bullet\|_{\widehat{X}}$ (resp. $\|\bullet\|_{\widehat{Y}}$) and suppose $X$ and $X_h$ (resp. $Y$ and $Y_h$) are two complete linear subspaces of $\widehat{X}$ (resp. $\widehat{Y}$) with inherited norms $\|\bullet\|_{X}:=\big(\|\bullet\|_{\widehat{X}}\big)|_{X}$ and $\|\bullet\|_{X_h}:=\big(\|\bullet\|_{\widehat{X}}\big)|_{X_h}$ (resp. $\|\bullet\|_{Y}:=\big(\|\bullet\|_{\widehat{Y}}\big)|_{Y}$ and $\|\bullet\|_{Y_h}:=\big(\|\bullet\|_{\widehat{Y}}\big)|_{Y_h}$){; $X + X_h \subseteq \widehat{X}$ and $Y + Y_h \subseteq \widehat{Y}$.}

\begin{table}[h!]
{	\centering 
\begin{tabular}{|c|c||c|c|}
	\hline 
	bilinear form & domain & \begin{minipage}{4cm}
		associated operator
	\end{minipage} & operator norm \\ \hline \hline
$a_{\rm pw}$ & \begin{minipage}{2cm}
	\centering
	\vspace{5pt}
	$\widehat{X} \times \widehat{Y}$
	\vspace{5pt}
\end{minipage} & \begin{minipage}{4cm}
	\centering
--
\end{minipage} & -- \\ \hline 
$a := a_{\rm pw}\vert_{X \times Y}$ & $X \times Y$ & \begin{minipage}{4cm}
	\centering
	\vspace{5pt}
	$A \in L(X;Y^{\ast})$ \\
	$Ax = a(x,\bullet) \in Y^{\ast}$  \\	\vspace{5pt}
\end{minipage} & $\|A\| := \|A\|_{L(X;Y^\ast)}$ \\ \hline 
$a_h$ & $X_h \times Y_h$ & \begin{minipage}{4cm}
	\centering
	\vspace{5pt}
	$A_h \in L(X_h;Y_h^{\ast})$ \\
	$A_hx_h = a_h(x_h,\bullet) \in Y_h^{\ast}$  \\	\vspace{5pt}
\end{minipage} & -- \\ \hline 
$\widehat{\nonlin}$ &  \begin{minipage}{2cm}
	\centering
	\vspace{5pt}
	$\widehat{X} \times \widehat{Y}$
	\vspace{5pt}
\end{minipage} & \begin{minipage}{4cm}
	\centering
-
\end{minipage} & $\|\widehat{\nonlin} \| := \|\widehat{\nonlin} \|_{ \widehat{X} \times \widehat{Y}}$ \\ \hline 
$\nonlin := \widehat{\nonlin}\vert_{X \times Y}$ &  \begin{minipage}{2cm}
	\centering
	\vspace{5pt}
	$X \times Y$
	\vspace{5pt}
\end{minipage} & \begin{minipage}{4cm}
	\centering
	\vspace{5pt}
	$B \in L(X;Y^\ast)$ \\
	$Bx =\nonlin(x,\bullet) \in Y^{\ast}$  \\	\vspace{5pt}
\end{minipage} & $\|\nonlin \| := \|\nonlin\|_{ X \times Y }$  \\ \hline
\end{tabular}
\caption{Bilinear forms, operators, and norms}
\label{tab:bilinear_forms}}
\end{table}

Table~\ref{tab:bilinear_forms} summarizes the bounded bilinear forms and associated operators with norms. Let the linear operators
${A}\in L({X}; {Y}^*)$ and 
$A+B\in L(X;Y^*)$ be associated to the  
bilinear forms $a$ 
and $a+\nonlin$ and suppose $A$ and $A+B$ are invertible so that the inf-sup conditions 
\begin{gather} 
0 < \alpha :=\inf_{\substack{{x} \in {X} \\ \|{x}\|_{{X}}=1}} \sup_{\substack{{y}\in {Y}\\ \|{y}\|_{{Y}}=1}} a({x},{y})\;\;  \text{ and }\;\;  
0<\beta:=\inf_{\substack{x\in X\\ \|x\|_{X}=1}} \sup_{\substack{y\in Y\\ \|y\|_{Y}=1}}(a+\nonlin)(x,y)\label{cts_infsup}
\end{gather}
hold.  Assume that the linear operator $A_{h} : X_{h} \rightarrow Y_{h}^\ast$  is invertible and
\begin{align}
\label{dis_Ah_infsup}
0 < \alpha_0 \le \alpha_h  :=\inf_{\substack{{x}_h\in {X}_h \\ \|{x_h}\|_{{X_h}}=1}} \sup_{\substack{{y}_h\in {Y_h}\\ \|{y_h}\|_{{Y}_h}=1}}\dislin({x_h},{y_h})
\end{align}
holds for some universal constant $\alpha_0$. Let the linear operators $P \in L(X_h;X)$, $Q \in L(Y_h;Y)$,  $\rop \in L(X_h;\widehat{X})$, $S \in L(Y_h; \widehat{Y})$  and the constants $\Lambda_{\rm P},\Lambda_{\rm Q},  \lamr, \lams \ge 0$ satisfy
\begin{gather}
\| (1 - P)x_h \|_{\widehat{X}} \le \Lambda_{\rm P} \| x - x_h \|_{\widehat{X}} \fl x_h \in X_h \text{ and } x \in X,\label{quasioptimalsmootherP}\\
\| (1 - Q)y_h \|_{\widehat{Y}} \le \Lambda_{\rm Q} \| y - y_h \|_{\widehat{Y}} \fl y_h \in Y_h \text{ and } y \in Y,\label{quasioptimalsmootherQ}\\
\| (1 - \rop)x_h \|_{\widehat{X}} \le \lamr\| x - x_h \|_{\widehat{X}} \fl x_h \in X_h \text{ and } x \in X, \label{quasioptimalsmootherR}\\
\| (1 - \sop)y_h \|_{\widehat{Y}} \le \lams\| y - y_h \|_{\widehat{Y}} \fl y_h \in Y_h \text{ and } y \in Y\label{quasioptimalsmootherS}.
\end{gather}
Suppose the operator  $I_{{\rm X}_h} \in L(X ;X_h)$, 
the constants $\Lambda_1, \delta_2$, $\delta_3 \ge 0,$  the above bilinear forms $a,\,\dislin, \widehat{\nonlin}$, and the linear operator $A$ from Table~\ref{tab:bilinear_forms} satisfy, {for all $x_h \in X_h,\,y_h \in Y_h,\,x \in X,$ and $y \in Y$, } that

\medskip
\centerline{\begin{minipage}{12cm}
\centering
\begin{enumerate}[label= $\mathrm{\bf (H\arabic*)}$,ref=$\mathrm{\bf (H\arabic*)}$,leftmargin=\widthof{(C)}+3\labelsep]
	\item \label{h1} $	\dislin(x_h,y_h)-a(Px_h,Qy_h)\le \Lambda_{1}\| x_h - Px_h\|_{\widehat{X}} \| y_h \|_{Y_h},$
	\item \label{h3} 
	$\displaystyle \delta_2 := \sup_{\substack{x_h\in X_h\\ \|x_h\|_{X_h}=1}} \|(1 - I_{{\rm X}_h}) A^{-1}(\widehat{\nonlin}(\rop x_h,\bullet)|_{Y})\|_{\widehat{X}},$
	\item \label{h4}
	$\displaystyle\delta_3:=\sup_{\substack{x_h\in X_h\\ \|x_h\|_{X_h}=1}}\|\widehat{\nonlin}(\rop x_h,(Q-\sop)\bullet\big)\|_{ Y_h^*}$.
\end{enumerate}
\end{minipage}}
\noindent  {{In applications, we establish that $\delta_2$ and $\delta_3$ are sufficiently small.}}
 Given $\alpha$, $\beta$,  $\alpha_h$, {$\Lambda_{\rm P}$, $\Lambda_{1}$, $\lamr$, $\delta_{2}$, $\delta_{3}$ from above} and the norms $\| A\|$ and $\|\widehat{\nonlin}\|$ from Table \ref{tab:bilinear_forms}, define
\begin{align}\label{AsmpCondn}
\widehat{\beta} &:=\frac{{\beta}}{\Lambda_{\rm P}{\beta}+\|A\|\left(1+\Lambda_{\rm P}\left(1+{\alpha}^{-1}\|\widehat{\nonlin}\|(1 + \lamr)\right)   \right)},\\ \label{eqn:defn_of_beta0}
\beta_{0}& :=\alpha_h \widehat{\beta}-\delta_2(\|Q^*A\|(1+\Lambda_{\rm P})+\alpha_h+\Lambda_1 \Lambda_{\rm P})- \delta_3
\end{align}
{{with the adjoint $Q^*$ of $Q$.}}
 In all applications of this article, {$1/\alpha$, $1/\beta$, $1/\alpha_h$, $\Lambda_{\rm P}$, $\Lambda_{\rm Q}$, $\lamr$, $\lams$, $\Lambda_{1}$, {{ and $\|Q^*A\|$}} are bounded  from above by generic constants, while $\delta_2$ and $\delta_3$ are controlled in terms of the maximal mesh-size $h_{\rm max}$  of an  underlying triangulation and tend to zero as $h_{\rm max} \rightarrow 0$.  Hence, $\beta_0 > 0$ is positive for  sufficiently fine triangulations and even bounded away from zero,  $\beta_0 \gtrsim 1$.  (Here  $\beta_0 \gtrsim 1$ means $\beta_0  \geq C$ for some positive generic constant $C$.)~This enables the following discrete inf-sup condition.
\begin{thm}[discrete inf-sup condition]\label{dis_inf_sup_thm}
	Under the aforementioned notation, \eqref{cts_infsup}-\eqref{quasioptimalsmootherP},   \eqref{quasioptimalsmootherR}, \eqref{eqn:defn_of_beta0} 
	and~\ref{h1}-\ref{h4} imply the stability condition 
 \begin{align} \label{eqn:disinfsup}
	\beta_h:= \inf_{\substack{x_h\in X_h\\ \| x_h\|_{X_h}=1}}\sup_{\substack{y_h\in Y_h\\ \| y_h\|_{Y_h}=1}} 
	(\dislin(x_h,y_h) +\widehat{\nonlin}(\rop x_h,\sop y_h)) \ge \beta_{\rm 0}.
\end{align}
\end{thm}
\noindent Before the proof of Theorem \ref{dis_inf_sup_thm} completes this section, some remarks on the particular choices of $R$ and $S$ are in order to motivate the {general description.} 

\begin{example}[quasi-optimal smoother $JI_\M$] \label{ex1}
\noindent   This paper follows \cite{ccnnlower2022} in the definition of the quasi-optimal smoother $P=Q=JI_\M$ in the applications with $X=Y=V=:H^2_0(\Omega)$ for the biharmonic operator $A$ and the linearisation $B$ of the trilinear form. Then \eqref{quasioptimalsmootherP}-\eqref{quasioptimalsmootherQ} follow in Subsection~\ref{sec:inter_comp} below; cf.  Definition \ref{def:morleyii} (resp.  Lemma~\ref{hctenrich}) for the definition of the Morley interpolation $I_{\rm M}$ (resp.  the companion operator $J$).
\end{example}
\begin{example}[no smoother in nonlinearity] 
 The natural choice in the setting of Example \ref{ex1} reads $R={\rm id}=S$   \cite{CCGMNN_semilinear}. Then $\Lambda_{\rm R}=0 =\Lambda_{\rm S} $ in \eqref{quasioptimalsmootherR}-\eqref{quasioptimalsmootherS}  and a priori error estimates will be available for the respective discrete energy norms. However, only a few optimal convergence results shall follow for the error in the piecewise weaker Sobolev norms, e.g., for the Morley scheme for the Navier-Stokes (Theorem~\ref{thm:apost_ns}.c) and for the von K\'{a}rm\'{a}n equations~(Theorem~\ref{thm:aprio_weak_vke}.b). 
\end{example}
\begin{example}[smoother in nonlinearity]
The choices $R=P$ and $S=Q$ lead to $\Lambda_{\rm R} = \Lambda_{\rm P}$ and $\Lambda_{\rm S}= \Lambda_{\rm Q}$ in \eqref{quasioptimalsmootherR}-\eqref{quasioptimalsmootherS}, while $\delta_3=0$ in {\bf (H3)}.  This allows for optimal a priori error estimates in the piecewise energy and in  weaker Sobolev norms and this is more than an academic exercise for a richer picture on the respective convergence properties; cf. \cite{ccnn2021} for exact convergence rates for the Morley FEM.  This is important for the analysis of quasi-orthogonality in the proof of optimal convergence rates of adaptive mesh-refining algorithms in \cite{adaptivemorley}. 
\end{example}
\begin{example}[simpler smoother in nonlinearity]
 The realisation of ${R=S=P=JI_\M}$ in the setting of Example~\ref{ex1} may lead to cumbersome implementations in the nonlinear terms and so the much cheaper choice $R=S=I_\M$ shall also be discussed in the applications below.
\end{example}
\begin{rem}[on {\bf (H1)}] The paper \cite{ccnnlower2022}  adopts \cite{veeser1}-\cite{veeser2} and extends those results to the dG scheme as a preliminary work on linear problems for this paper. The resulting abstract condition ${\bf (H1)}$ therein is a key property to analyze the linear terms simultaneously. 
\end{rem} 
\begin{rem}[comparison with \cite{CCGMNN_semilinear}]
The set of hypotheses for the discrete inf-sup condition in this article differs from those in~\cite{CCGMNN_semilinear}. 
This paper allows smoothers in the nonlinear terms and also applies to dG/$C^0$IP/WOPSIP schemes.  
\end{rem}
\begin{rem}[consequences of \eqref{quasioptimalsmootherP}-\eqref{quasioptimalsmootherS}]\label{rem.consequences} The estimates in  \eqref{quasioptimalsmootherP}-\eqref{quasioptimalsmootherS} give rise to a typical estimate utilised throughout the analysis in this paper.  For instance,   \eqref{quasioptimalsmootherP}  (resp.  \eqref{quasioptimalsmootherR}) and a triangle inequality show, for all
$x\in X$ and $x_h \in X_h$,  that 
\begin{align} \label{eqn:PR}
\|x- Px_h\|_{{X}} \le (1+\Lambda_{\rm P}) \|x-x_h\|_{\widehat{X}} \; \;  (\text{resp.  } \|x- Rx_h\|_{\widehat{X}} {\le} (1+\Lambda_{\rm R}) \|x-x_h\|_{\widehat{X}}). 
\end{align}
The analog  \eqref{quasioptimalsmootherQ} (resp.  \eqref{quasioptimalsmootherS}) leads, for all $y \in Y$ and $y_h \in Y_h$, to
\begin{align} \label{eqn:QS}
\|y- Qy_h\|_{{Y}} \le (1+\Lambda_{\rm Q}) \|y-y_h\|_{\widehat{Y}}\; \;(\text{resp. } 
\|y- Sy_h\|_{\widehat{Y}} \le (1+\Lambda_{\rm S}) \|y-y_h\|_{\widehat{Y}}).
\end{align}
\end{rem}
\noindent \textit{Proof of Theorem~\ref{dis_inf_sup_thm}.} The proof of Theorem~\ref{dis_inf_sup_thm} departs {as in}~\cite[Theorem 2.1]{CCGMNN_semilinear} for nonconforming schemes for any given  $x_h\in X_h$ with \textcolor{blue}{$\| x_h\|_{X_h}=1$}. 
Define $$x:=P x_h, \; \eta := A^{-1}(Bx), \; 
\: \xi := A^{-1}(\widehat{\nonlin}(\rop x_h,\bullet)|_Y)\in X,\;\text{and}\; \xi_h:=I_{{\rm X}_h}\xi \in X_h .$$
The  definitions of~$\xi \in X$ and $\xi_h \in X_h$ lead in~\ref{h3} to
\begin{align}
\|\xi - \xi_h \|_{\widehat{X}} \le \delta_2.
\label{eqn:psi_delta2}
\end{align}
\noindent {{The second inf-sup condition in \eqref{cts_infsup}}} and $A\eta=Bx \in Y^*$ result in
	\begin{equation*}
	\beta\|x\|_X\leq \|Ax+Bx\|_{Y^*}=\|A(x+\eta)\|_{Y^*}\leq \|A\|\|x+\eta\|_{X}
	\end{equation*} 
with the operator norm of $A$ in the last step. 
	This and  triangle inequalities imply
	\begin{align}
	(\beta/{\|A\|})\,\|x\|_X\leq\|x+\eta\|_{X}\leq{} \|x-x_h\|_{\widehat{X}}+\|x_h+\xi\|_{\widehat{X}}+\|\xi - \eta\|_X. \label{xbound}
	\end{align}
{The above definitions} of $\xi$ and $\eta$ {guarantee} $a(\xi -\eta,\bullet) = \widehat{\nonlin}(\rop x_h - x,\bullet)\vert_{Y} \in Y^{\ast}$. This, {\eqref{cts_infsup}}, and the norm 
	$\|\widehat{\nonlin}\|$ of the bilinear form $\widehat{\nonlin}$ show
	\begin{equation}\label{xi_eta_bdd}
\alpha	\|\xi-\eta\|_X \le \|\widehat{b}(x - \rop x_h, \bullet)\|_{Y^\ast}\leq \|\widehat{\nonlin}\|\|x-\rop x_h\|_{\widehat{X}} \le \| \widehat{\nonlin} \|(1 + \lamr) \|x - x_{h} \|_{\widehat{X}}
	\end{equation}
	with  \eqref{eqn:PR} 
 in the last step. Note that the definition $x =P x_h$ and~\eqref{quasioptimalsmootherP} imply 
	\begin{align}\label{A4app}
	\| x-x_h\|_{\widehat{X}}\leq{}& \Lambda_{\rm P}\| x_h+\xi\|_{\widehat{X}}.
	\end{align}
	The combination of \eqref{xbound}-\eqref{A4app} results in 
	\begin{align}\label{xfirst_bdd}
\|x\|_{X}\leq \|x_h+\xi\|_{\widehat{X}}(1  +\Lambda_{\rm P} (1 +{\alpha}^{-1}\|\widehat{\nonlin}\| (1 + \lamr)) ) \|A\|/\beta.
	\end{align}		
	A triangle inequality, \eqref{A4app}-\eqref{xfirst_bdd}, and the definition of $\widehat{\beta}$ in \eqref{AsmpCondn} lead to  
	\begin{align*}
	1=\|x_h\|_{X_h}&\leq\|x-x_h\|_{\widehat{X}}+\|x\|_X \le {\widehat{\beta}}^{-1}\|x_h + \xi \|_{\widehat{X}}.
	\end{align*}
This in  the first inequality below and a triangle inequality  plus \eqref{eqn:psi_delta2} show
	\begin{equation}\label{beta_bdd}
	\widehat{\beta} \leq \|x_h+\xi\|_{\widehat{X}} \le \|x_h + \xi_h \|_{X_h} +  \|\xi - \xi_h\|_{\widehat{X}} \le  \|x_h + \xi_h \|_{X_h}+\delta_2.
	\end{equation}
	{{The { condition} \eqref{dis_Ah_infsup}}} implies for $x_h +  \xi_h \in X_h$ and for any $\epsilon >0$, the existence of some $\phi_h \in Y_h$ such that $ \|\phi_h\|_{Y_h} \le 1+  {\epsilon}$ and $\alpha_h \|x_h + \xi_h\|_{X_h}  = \dislin(x_h + \xi_h,\phi_h).$  Elementary algebra shows 
	\begin{align} \label{eqn:split_axh}
	\hspace{-0.5cm} \alpha_h  \|x_h + \xi_h \|_{X_h}={} \dislin(x_h,\phi_h){+}\dislin( \xi_h,\phi_h){-}a(P\xi_h,Q\phi_h){+ }a(P\xi_h - \xi,Q\phi_h) {+}  a(\xi,Q\phi_h)
	\end{align}
and motivates the control of {the} terms below. 
	\noindent Hypothesis~\ref{h1} and~\eqref{quasioptimalsmootherP} imply
	\begin{align} \label{eqn:split_est1}
	a_h(\xi_h,\phi_h) - a(P\xi_h,Q\phi_h) \le \Lambda_{1} \Lambda_{\rm P} \|\xi  - \xi_h \|_{\widehat{X}} \|\phi_h\|_{Y_h}  \le  \Lambda_{1} \Lambda_{\rm P} \delta_{2} (1 +\epsilon)
	\end{align}
	with~\eqref{eqn:psi_delta2} and $\|\phi_h\|_{Y_h} \le 1 + \epsilon$ in the last step { above}.
The boundedness of $Q^{\ast}A \in L(X; Y_h^*)$,  $\|\phi_h\|_{Y_h} \le 1 + \epsilon$, ~\eqref{eqn:PR}, and \eqref{eqn:psi_delta2}} for $\| \xi - P \xi_h \|_{X} \le (1 + \Lambda_{\rm P}) \| \xi - \xi_h \|_{\widehat{X}} \le (1 + \Lambda_{\rm P}) \delta_2$ reveal 
	\begin{align}  \label{infsup.t2}
\hspace{-0.8cm}	a(P\xi_h-\xi,Q\phi_h)
 \le  \|Q^{\ast}A\|(1+\Lambda_{\rm P})\delta_2(1+\epsilon).
	\end{align}~\noeqref{eqn:split_est1,infsup.t2}
 The definition of $\xi$ shows that $a(\xi,Q\phi_h) =\widehat{\nonlin}(\rop x_h,Q\phi_h)$. 
 This,  $\|\phi_h\|_{Y_h} \le 1 + \epsilon$, and ~\ref{h4} imply
\begin{equation}\label{infsup.t3}
a(\xi,Q\phi_h) \le \widehat{\nonlin}(Rx_h,S\phi_h)+\delta_3(1+\epsilon).
\end{equation}
The combination of {{~\eqref{eqn:split_axh}- \eqref{infsup.t3}}} reads
\begin{align} \label{eqn:pstep}
\hspace{-0.1in} \alpha_h \|x_h + \xi_h\|_{X_h} &  \le \dislin(x_h,\phi_h)+ \widehat{\nonlin}(\rop x_h,\sop \phi_h) + \big((\|Q^{\ast}A\|(1+\Lambda_{\rm P})+\Lambda_{1}\Lambda_{\rm P})\delta_2+\delta_3\big)(1+\epsilon).
\end{align}
{{This, \eqref{beta_bdd},  and $\|\phi_h\|_{Y_h} \le 1+\epsilon$  imply
$\alpha_h  \widehat{\beta} \le{} (\big\|a_h(x_h,\bullet) +  \widehat{b}(R x_h,S \bullet) \|_{Y_h^*}
+(\|Q^{\ast}A\|(1+\Lambda_{\rm P})+\Lambda_{1}\Lambda_{\rm P})\delta_2+\delta_3\big)(1 + \epsilon)+\alpha_h \delta_2. $ This and }}	\eqref{eqn:defn_of_beta0} 
{demonstrate} 
	$\alpha_h \widehat{\beta} \le{} 
	{{(\|\dislin(x_h,\bullet) +  \widehat{\nonlin}(\rop x_h,\sop \bullet) \|_{Y_h^*}}}
	+ \alpha_h \widehat{\beta} - \beta_{0})(1 + \epsilon)	{{-\epsilon\alpha_h \delta}}.$
	At this point, we may choose  $\epsilon \searrow 0$ and obtain 
	\begin{align*}
\beta_{0}\leq  \|\dislin(x_h,\bullet)+  \widehat{\nonlin}(\rop x_h,\sop \bullet)\|_{Y_h^*}.
	\end{align*}
	Since $x_h \in X_h$ is arbitrary with $\|x_h \|_{X_h} = 1$, this {proves} the discrete inf-sup condition  { \eqref{eqn:disinfsup}}.  (In this section $Y_h$ is a closed subspace of the Banach space $\widehat{Y}$ and not necessarily reflexive.  In the sections below,  $Y_h$ is finite-dimensional and the above arguments apply immediately to {$\epsilon=0$.})
\qed


\section{Main results}\label{sec:mainresults}
This section introduces the continuous and discrete nonlinear problems, associated notations, and  states the main results of this article in {\bf (A)}-{\bf (C)} below.   The paper has two parts written  in abstract results of  Section \ref{sec.stability}, \ref{sub:existence}-\ref{sec:goal-oriented} and their applications in Section \ref{sec:nse}-\ref{sec:vke}.  In the first part,  the hypotheses~\ref{h1}-~\ref{h4}   in the setting of Section \ref{sec.stability} and the hypothesis \ref{h5} stated below  guarantee the existence and uniqueness of an approximate solution for the discrete problem,  feasibility of an iterated Newton scheme, and { an} a priori energy norm estimate in {\bf (A)}-{\bf(B)}.  An additional hypothesis \ref{h1hat}  enables a priori error estimates in weaker Sobolev norms stated in {\bf (C)}. The second part in Section \ref{sec:nse}-\ref{sec:vke} verifies the  abstract results for the 2D Navier-Stokes equations in the stream function vorticity formulation and for the \vket. 

\medskip
Adopt the notation on the Banach spaces $X$ and $Y$ (with $X_h, \widehat{X}$ and $Y_h, \widehat{Y}$) of the previous section and suppose 
that the quadratic function $N:X \rightarrow Y^*$ is 
\begin{equation}\label{eqccdefN}
N(x):= Ax + \Gamma(x,x,\bullet) - F(\bullet) \quad \text{for all } x \in X
\end{equation}
with a bounded linear operator $A \in L(X;Y^*)$,  a bounded trilinear form $\Gamma: X\times X\times Y\to \R$, and a linear form $F\in Y^*$.
Suppose there exists a bounded trilinear form $ \widehat{\Gamma}:\widehat{X}\times \widehat{X}\times \widehat{Y}\to \R$
with $\Gamma=\widehat{\Gamma}|_{X\times X\times Y}$, 
$\Gamma_h=\widehat{\Gamma}|_{X_h\times X_h\times Y_h}$, and  let
\begin{align}
\| \widehat{\Gamma}\|:=\| \widehat{\Gamma}\|_{\widehat{X}\times \widehat{X}\times \widehat{Y}}:=
\sup_{\substack{\widehat{x}\in \widehat{X}\\ \|\widehat{x}\|_{\widehat{X}}=1}}
\sup_{\substack{\widehat{\xi}\in \widehat{X}\\ \|\widehat{\xi}\|_{\widehat{X}}=1}}
\sup_{\substack{\widehat{y}\in \widehat{Y}\\ \|\widehat{y}\|_{\widehat{Y}}=1}}
\widehat{\Gamma}(\widehat{x},\widehat{\xi},\widehat{y})<\infty.
\label{eqn:norm_gamma}
\end{align}

\noindent The linearisation of $\widehat{\Gamma}$  at  $u \in X$  defines the
bilinear form $\widehat{\nonlin}:\widehat{X}\times \widehat{Y}\to \mathbb{R}$,
\begin{align}\label{star}
\widehat{\nonlin}(\bullet,\bullet)& := \widehat{\Gamma}(u,\bullet, \bullet)+\widehat{\Gamma}(\bullet,u,\bullet).
\end{align}
The boundedness of $ \widehat{\Gamma}(\bullet,\bullet,\bullet)$  applies to \eqref{star}  and provides 
$\|\widehat{\nonlin}\| \le 
2\| \widehat{\Gamma}\|\|u\|_X $. 
\begin{defn}[regular root]  A function $u \in X$ is a regular root to~\eqref{eqccdefN},  if  $u$ solves
	\begin{align}
	\label{eqn:p}
	N(u;y) = a(u,y) + \Gamma(u,u,y) -  F(y)= 0 \fl y \in Y 
	\end{align} and the Frech\'et derivative $DN(u)=:(a+\nonlin)(\bullet,\bullet)$ defines an isomorphism $A+B$ and in particular satisfies the 
	inf-sup condition \eqref{cts_infsup} for $b:=\widehat{b}|_{X \times Y}$ and $\widehat{b}$ from \eqref{star}.  \qed
\end{defn}  
\noindent 

\noindent Abbreviate $(a+b)(x,y):=a(x,y)+b(x,y)$ etc.  Several \emph{discrete problems} in this article are defined  for  different choices of $\rop$ and $\sop$  with  \eqref{quasioptimalsmootherR}-\eqref{quasioptimalsmootherS} to approximate the regular root $u$ to $N$. In the applications of Section \ref{sec:nse}-\ref{sec:vke},  $ \rop, \sop \in \{\mathrm{id},I_{\M},  JI_{\M}\}$ lead to {\it eight} new discrete nonlinearities.
Let $X_h$ and $Y_h$ be finite-dimensional spaces 
and let  
\begin{align} \label{eqn:defn_of_nh}
N_h(x_h) := \dislin(x_h,\bullet) + \widehat{\Gamma}(\rop x_h,\rop x_h,\sop\bullet)- F(Q\bullet) \in Y_h^*.
\end{align}
The discrete problem seeks a root $u_h \in X_h$ to $N_h$; in other words it seeks $u_h \in X_h$ that satisfies
\begin{align}
\label{eqn:dp}
N_h(u_h;y_h): ={}& \dislin(u_h,y_h)  + \widehat{\Gamma}(\rop u_h,\rop u_h,\sop y_h) - F(Qy_h) = 0 \text{ for all }y_h \in Y_h. 
\end{align}
The local discrete solution $u_h \in X_h$ depends on $R$ and $S$ (suppressed in the notation). 
{Suppose
\medskip
\centerline{\begin{minipage}{13cm}
		\begin{enumerate}[label= $\mathrm{\bf (H\arabic*)}$,ref=$\mathrm{\bf (H\arabic*)}$,leftmargin=\widthof{(C)}+3\labelsep]
			\setcounter{enumi}{3}
			\item \label{h5} $\exists x_h \in X_h$ such that  $\delta_4 := \|u- x_h\|_{\widehat{X}}<\beta_0/2 (1 + \lamr)  \|\widehat{\Gamma}\| \| R\| \|S\|$
		\end{enumerate}
\end{minipage}
}
so that, in particular,
\begin{equation} \label{eqn:beta1}
\beta_1:= \beta_0 -  2 (1 + \lamr)  \|\widehat{\Gamma}\| \| R\| \|S\|\delta_4 > 0. 
\end{equation}}
\noindent The non-negative  parameters $\Lambda_1, \delta_2,\delta_3,\,\delta_4$, $\beta$, and $\|\widehat{\nonlin}\|$ 
depend on the regular root $u$ to $N$
(suppressed in the notation).

\medskip
\noindent
{{The hypotheses \ref{h1}-\ref{h5} with sufficiently small $\delta_2,$ $\delta_3$, $\delta_4$ imply the results stated in {\bf (A)}-{\bf (B)} below for parameters}} $\epsilon_1,  \epsilon_2, \delta, \rho,$ $C_{\rm qo}>0$  and $0 <  \kappa <1$,  such that {\bf (A)}-{\bf (B)} hold for any underlying triangulation $\cT$ with maximum mesh-size $h_{\rm max} \le \delta$ {{in the applications of this article}}.

\begin{enumerate}[label= $\mathrm{\bf (\Alph*)}$,ref=$\mathrm{\bf (\Alph*)}$,leftmargin=\widthof{(C)}+3\labelsep]
	\item[{\bf (A)}] \label{main-a} {\em local existence of a discrete solution.} There exists a unique discrete solution $u_h\in 
		X_h $ to $N_h(u_h)=0$ in~\eqref{eqn:dp}  with $\| u-u_h\|_{\widehat{X}} \le \epsilon_1$.  For any  initial iterate  
		$v_h \in X_h$ with $\|u_h-v_h\|_{{X_h}}  \le \rho$,
		the Newton scheme converges
		quadratically to  $u_h$.
	\item[{\bf (B)}] \label{main-b} 
 {\em a priori error control in energy norm}.   The continuous (resp. discrete) solution $u \in X$ (resp. $u_h \in X_h$) with 
$\| u-u_h\|_{\widehat{X}} \le \epsilon_2:=\min \left\{ \epsilon_1, \frac{ \kappa  \beta_1}{ (1+ \Lambda_{\rm R})^2\|\sop\|  \|\widehat{\Gamma}\|} \right\}$ satisfies
$$\displaystyle 
			\|u - u_h \|_{\widehat{X}} \le C_{\rm qo} \min_{x_h \in X_h} \|u - x_{h}\|_{\widehat{X}} + {\beta^{-1}_1(1-\kappa)^{-1}}{\| \widehat{\Gamma}(u,u,(\sop-Q)\bullet) \|_{Y_h^\ast} }
	$$ with a lower bound $\beta_1$ of $\beta_h$ defined  in~\eqref{eqn:beta1}.  The quasi-best approximation result \eqref{eqn:best} 
 holds for $S=Q$.
	\item[{\bf (C)}] \label{main-c}  {\em a priori error control in weaker Sobolev norms}.  
\noindent In addition to~\ref{h1}--\ref{h5}, suppose the existence of $\Lambda_5>0$ such that, for all $x_h \in X_h$, $y_h \in Y_h$, $x \in X$, and $y \in Y$,

\centerline{\begin{minipage}{10cm}
	\begin{enumerate}[label= $\mathrm{\bf \widehat{(H\arabic*)}}$,ref=$\mathrm{\bf \widehat{(H\arabic*)}}$,leftmargin=\widthof{(C)}+3\labelsep]
	\setcounter{enumi}{0}
	\item \label{h1hat} $\dislin(x_h,y_h)-a(Px_h,Qy_h)\le \Lambda_{5}\| x - x_h\|_{\widehat{X}} \|y - y_h \|_{\widehat{Y}}$.
\end{enumerate}
\end{minipage}}
For any $G \in X^{\ast}$, if  $z \in Y$ solves the dual linearised problem $a(\bullet,z) + \nonlin(\bullet,z) = G(\bullet)$ in $X^*$,
then any $z_h \in Y_h$ satisfies
	\begin{align}
	& \| u - u_h \|_{X_{\rm s}}  \le \omega_1(||u||_{X},||u_h||_{X_h})  \| z - z_h \|_{\widehat{Y}}  \|u - u_h \|_{\widehat{X}}    + \omega_2(\|z_h\|_{Y_h})  \|u - u_{h} \|^2_{\widehat{X}} \nonumber \\
	& \quad + \|u_h - Pu_h\|_{X_{\rm s}}   + \widehat{\Gamma}(u,u,(S-Q)z_h) +
\widehat{\Gamma}(Ru_h,Ru_h,Q z_h) -{\Gamma}(Pu_h,Pu_h,Q z_h)
	\end{align}
with appropriate weights defined in \eqref{weights} below.  Here $X_{\rm s}$ is a Hilbert space with $X \subset X_{\rm s}$. 
\end{enumerate}
The abstract results  {\bf(A)}-{\bf (C)} are established in Theorems \ref{thm:existence}, \ref{thm:apriori}, and \ref{cor:lower}.  A summary of their consequences in the applications in Section \ref{sec:nse}-\ref{sec:vke} for a triangulation with sufficiently small maximal mesh-size $h_{\rm max}$ is displayed in Table \ref{tab:apriori}. 
\section{Existence and uniqueness of discrete solution} \label{sub:existence}
This section applies the Newton-Kantorovich convergence theorem 
 to establish {\bf (A)}.  Let  $u\in X$ be a regular root to $N$. 
 Let \eqref{quasioptimalsmootherP}, \eqref{quasioptimalsmootherR}, and~\ref{h1}-\ref{h5} hold with parameters $\Lambda_{\rm P},\,\lamr,\,\Lambda_{1},\,\delta_2,$ $\,\delta_3,$ $\,\delta_4 \ge 0$. Define  $L := 2   \| \widehat{\Gamma}\| \|R\|^2 \|S\|$, $m :=L/ \beta_1$, and  
		\begin{align}
		{\epsilon_0}& := \beta_1^{-1}\big{(}(\Lambda_{1}\Lambda_{\rm P} +  \|Q^*A\| (1 + \Lambda_{\rm P}) +(1 + \lamr) (\| \rop \| \|\sop\| \|x_h\|_{X_h}+\|Q\| \: \|u\|_{X})\|\widehat{\Gamma}\| \big) \delta_4\, \nonumber \\ 
&\qquad \qquad  \quad +\|x_h\|_{X_h}\delta_3/2\big{)}.
		 	\label{defn_delta}
		\end{align}
In this section (and in Section \ref{sec:apriori} below), $Q \in L(Y_h;Y)$ (resp. $S \in L(Y_h; \widehat{Y})$) is bounded, but \eqref{quasioptimalsmootherQ} (resp.  \eqref{quasioptimalsmootherS}) is not employed. 
\begin{thm}[existence and uniqueness of a discrete solution]\label{thm:existence}
				(i)  If $\epsilon_0 m \le 1/2$, then there exists a root $u_h \in X_h$ of $N_h$ with $\| u-u_h \|_{\widehat{X}} \le \epsilon_1 := \delta_4 +  (1-\sqrt{1-2 \epsilon_0 m })/m.  $

				\noindent (ii)  If $\epsilon_0 m < 1/2$,  then given any $v_h\in X_h$
				with $\|u_h - v_h\|_{{X_h}}\le \rho := (1+\sqrt{1-2 \epsilon_0 m})/ m>0$, the Newton scheme 
				with initial iterate  $v_h$ converges quadratically
				to the root  $u_h$ to $N_h$ in (i).

				\noindent (iii) If $ \epsilon_1 m \le  1/2$, then there exists at most one root $u_h$ to $N_h$ 
				with $\| u-u_h \|_{\widehat{X}}\le \epsilon_1$.
			
	\end{thm}
	The proof of Theorem~\ref{thm:existence} applies the well-known Newton-Kantorovich convergence theorem found, e.g., in \cite[Subsection~5.5]{MR1344684} for \( X= Y=\R^n\) and in \cite[Subsection~5.2]{MR816732} for Banach spaces. The notation is adapted to the present situation.
\begin{thm}[Kantorovich (1948)] \label{kantorovich}
	Assume the Frech\'et derivative $DN_h(x_h)$ of $N_h$ at some \(x_h\in X_h\) satisfies
	\begin{equation}\label{Kanto_Condn}
	\|D N_h(x_h)^{-1}\|_{L( Y_h^*; X_h)} \leq 1/\beta_1
	\quad\text{and}\quad
	\|D N_h(x_h)^{-1}N_h(x_h)\|_{X_h} \leq {\epsilon_0}.
	\end{equation}
	Suppose that \(D N_h\) is Lipschitz continuous with Lipschitz constant $L$
	and that $ 2 \epsilon_0 L \le \beta_1$.
	Then there exists a root \(u_h\in
	\overline{ B(x_1,r_-)} \) of  \(N_h\) in the closed ball around the first iterate \(x_1 := x_h - D N_h(x_h)^{-1}N_h(x_h)\) of radius $r_-:= (1-\sqrt{1-2 \epsilon_0 m })/m  - {\epsilon_0}$
	and this is the only root of $N_h$ in \(\overline{B(x_h,\rho)}\) with $\rho := (1+\sqrt{1-2 \epsilon_0 m})/ m$. If $2 \epsilon_0 L < \beta_1$,
	then 
	the Newton scheme with   initial iterate \(x_h\) leads to a sequence in   {\(B(x_h,\rho)\)}  
	that  converges R-quadratically to \(u_h\). \qed
\end{thm}


\medskip

		\noindent{\em Proof of Theorem \ref{thm:existence}}. {\em Step 1 establishes \eqref{Kanto_Condn}.} 
		The bounded trilinear form $\widehat{\Gamma}$
		leads to the Frech\'et derivative $DN_h( x_h)\in L(X_h;Y_h^*)$ of $N_h$  from \eqref{eqn:defn_of_nh} evaluated at any $x_h \in X_h$ for all  $\xi_h\in X_h$, $\eta_h\in Y_h$ with
	\begin{align}
	 DN_h( x_h;\xi_h,\eta_h)= \dislin(\xi_h,\eta_h)+\widehat{\Gamma}(\rop x_h, \rop \xi_h,\sop \eta_h)
	+\widehat{\Gamma}( \rop\xi_h,\rop x_h,\sop \eta_h).
	\label{eqn:dhn}
	\end{align}
{For any $x_h^1, x_h^2, \xi_h \in X_h$  and $\eta_h \in Y_h$,~\eqref{eqn:dhn} implies the global Lipschitz continuity {of $DN_h$}  with Lipschitz constant 
$L:=  2 \| \widehat{\Gamma}\|  \| \rop \|^{2} \|\sop\|$,  and so
\[|DN_h(x_h^1;\xi_h,\eta_h) - DN_h(x_h^2 ;\xi_h,\eta_h)| \le  L \|x_h^1 - x_h^2\|_{X_h} \| \xi_h\|_{X_h} \|\eta_h\|_{Y_h}.\]

\medskip \noindent 
Recall  $x_h$ from {\bf (H4)} with $\delta_4 = \|u- x_h\|_{\widehat{X}}$. For this $x_h \in X_h$, \eqref{eqn:PR} leads to $\|u-Rx_h\|_{\widehat{X}} \le (1+\Lambda_{\rm R}) \delta_4$. {{This and the boundedness of $\widehat{\Gamma}(\bullet,\bullet,\bullet)$ show \[\widehat{\Gamma}(u-\rop x_h, \rop \xi_h,\sop \eta_h)+\widehat{\Gamma}( \rop \xi_h,u-\rop x_h,\sop \eta_h)\le 2\delta_4(1+\Lambda_{\rm R}) \|\widehat{\Gamma}\| \| \rop\| \|\sop\| \|\xi_h\|_{X_h}\|\eta_h\|_{Y_h}.\] The discrete inf-sup condition in 
		Theorem~\ref{dis_inf_sup_thm}, elementary algebra, and the above displayed estimate establish a positive inf-sup constant  }}
		\begin{align}
		\label{eqn:lower_betah}
		0< \beta_1 = \beta_0- 2 (1 + \lamr)  \|\widehat{\Gamma}\| \| R\| \|S\|\delta_4
\le  \inf_{\substack{\xi_h\in X_h\\ \|\xi_h\|_{X_h}=1}}
		\sup_{\substack{\eta_h\in Y_h\\ \|\eta_h\|_{Y_h}=1}}DN_h(x_h;\xi_h,\eta_h)
		\end{align}
		for the {discrete} bilinear form \eqref{eqn:dhn}. The inf-sup constant $\beta_1 > 0$ in~\eqref{eqn:lower_betah} is known to be (an upper bound of the) reciprocal of the operator norm of $DN_h(x_h)$ and that provides the first estimate in \eqref{Kanto_Condn}.  It also leads to 
		\begin{equation}\label{Kant_cond1}
		\|DN_h(x_h)^{-1}N_h(x_h)\|_{X_h}\leq \beta_1^{-1}\|N_h(x_h)\|_{Y_h^*}.
		\end{equation}
\noindent To establish the second inequality in  \eqref{Kanto_Condn}, for any $y_h\in Y_h$ with $\|y_h\|_{Y_h}=1$, set $y:=Q y_h\in Y.$ Since $N(u;y)=0$,~\eqref{eqn:p}-\eqref{eqn:defn_of_nh}  
reveal
\begin{align}
	\hspace{-1cm}	N_h(x_h;y_h)={}& N_h(x_h;y_h)-N(u;y)= \dislin( x_h,y_h)-a(u,y)
		+\widehat{\Gamma}( \rop x_h, \rop x_h, \sop y_h)-\Gamma(u,u,y).
		\label{nonlin_exp}
		\end{align}
The combination of~\ref{h1} and~\eqref{quasioptimalsmootherP} results in
		\begin{align}
		\dislin( x_h,y_h)- a(u,Qy_h) &= \dislin( x_h,y_h)- a(Px_h,Qy_h) - a(u -Px_h,Qy_h) \nonumber  \\
		&\le \Lambda_{1} \Lambda_{\rm P} \|u - x_h \|_{\widehat{X}} + \|Q^{\ast}A\|\|u - Px_{h}\|_{X} \nonumber 
		\end{align}
	with the operator norm  $\|Q^{\ast}A\| $ of $Q^\ast A$ in $L(X;Y_h^\ast)$ in the last step.   Utilize \eqref{eqn:PR} and~\ref{h5} to establish  {$\|u - Px_{h}\|_{X} \le (1 + \Lambda_{\rm P}) \delta_4 $}. This and the previous estimates imply
		\begin{align}
	\dislin( x_h,y_h)- a(u,Qy_h) \le   (\Lambda_{1}\Lambda_{\rm P} +  \|Q^*A\| (1 + \Lambda_{\rm P}) ) \delta_4.
		\end{align}
Elementary algebra and the boundedness of $\widehat{\Gamma}(\bullet,\bullet,\bullet)$, \eqref{quasioptimalsmootherR}, {and} {\ref{h4}-\ref{h5}} show 
		\begin{align*}
		2(\widehat{\Gamma}( \rop x_h, \rop x_h, \sop y_h) - \widehat{\Gamma}(u,u,y))
		& = \widehat{\Gamma}( \rop x_h-u,\rop x_h, \sop y_h) + \widehat{\Gamma}( \rop x_h,\rop x_h-u, \sop y_h) \\
		& \quad + \widehat{\Gamma}(u,\rop x_h-u,y)+\widehat{\Gamma}(\rop x_h-u,u,y) - \widehat{b}(Rx_h,(Q - S)y_h)
		\\
		& \leq 2\delta_4 (1 + \lamr)\left(\|\rop \| \|\sop\|  \|x_h\|_{X_h}+\|Q\| \: \|u\|_{X}\right)\|\widehat{\Gamma}\|\,+\delta_3 \|x_h \|_{X_h}.
		\label{eqn:gamma_delta3}
		\end{align*}
		A combination of the two above displayed estimates {in}~\eqref{nonlin_exp} reveals
		\begin{align}
		|N_h(x_h;y_h)|{ \le}&{}  (\Lambda_{1} \Lambda_{\rm P}+{}  \|Q^*A\| (1 {}+ {}\Lambda_{\rm P}) {+(1 {}+{}\lamr){}(\| \rop \|  \|\sop \|  \|x_h\|_{X_h}+\|Q\|  \|u\|_{X})\|\widehat{\Gamma}\|) \delta_4
+\frac{1}{2}\|x_h\|_{X_h}\delta_3}. 
		\end{align}
This implies $\|N_h(x_h)\|_{Y_h^*}\leq \beta_1{\epsilon_0} $
		with ${\epsilon_0}\ge 0$ from  \eqref{defn_delta}. The latter bound leads in \eqref{Kant_cond1} to the second condition 
in \eqref{Kanto_Condn}. 
		
		\noindent {\em Step 2 establishes the assertion (i) and (ii).}	
Since  $\epsilon_0 m \le 1/2$, $ r_-, \rho\ge 0$ is well-defined, $2 \epsilon_0 L \le \beta_1$, and hence Theorem \ref{kantorovich}  applies. 

We digress to discuss the degenerate case $\epsilon_0 = 0$ where~\eqref{defn_delta} implies $\delta_4 = 0$. An immediate consequence is that~\ref{h5} results in $u = x_h \in X_h$.  The proof of Step 1 remains valid and $N_h(x_h) = 0$ (since $\epsilon_0 = 0$) provides that $x_h = u$ is the discrete solution $u_h$. Observe that in this particular case,  the Newton iterates form the constant sequence $u=  x_h = x_1 = x_2 = \cdots$ and Theorem~\ref{Kanto_Condn} holds for the trivial choice $r_{-} = 0$.

Suppose $\epsilon_0 > 0$. For $\epsilon_0 m \le 1/2$,  Theorem~\ref{kantorovich} shows the existence of a root  $u_h$ to $N_h$ in $\overline{B(x_1,r_-)}$ that is the only root in $\overline{B(x_h,\rho)}$. This,  $ \|x_1-x_h\|_{X_h} \le \epsilon_0$, with $\epsilon_0$ from \eqref{defn_delta},  for  the Newton correction $x_1-x_h$ in the second inequality of 
		\eqref{Kanto_Condn}, and 
		triangle inequalities result in
		\begin{align}
		\|u-u_h\|_{\widehat{X}}\leq \|u-x_h\|_{\widehat{X}}+
		\|x_1-x_h\|_{X_h}
		+\|x_1-u_h\|_{X_h}\leq \delta_4+ (1-\sqrt{1-2 \epsilon_0 m })/m= \epsilon_1 \label{Newton_conv}.
		\end{align}
		This proves the existence of a discrete solution 
		$u_h$ in  $ X_h\cap \overline{B(u,\epsilon_1)}$  as asserted in $(i)$. Theorem~\ref{kantorovich} implies $(ii)$. \medskip
		
		\noindent {\em Step 3 establishes the assertion (iii).}  
		Recall from Theorem~\ref{kantorovich} that the limit $u_h\in \overline{B(x_1,r_-)}$  in $(i)$-$(ii)$ is the only discrete
		solution in $\overline{B(x_h,\rho)}$.
		Suppose there exists a second solution $\widetilde{u}_h\in X_h\cap \overline{B(u,\epsilon_1)}$ to  $N_h({{\widetilde{u}_h}})=0$.
		Since $u_h$ is unique in $\overline{B(x_h,\rho)}$, $\widetilde{u}_h$ lies outside $\overline{B(x_h,\rho)}$.  This and a triangle inequality show 
		\[
\dfrac{1}{m} \le (1 + \sqrt{1 - 2\epsilon_0 m })/m		=  \rho< \|  x_h- \widetilde{u}_h\|_{\widehat{X}}\le  \|  u- \widetilde{u}_h\|_{\widehat{X}}
		+\|  u- x_h\|_{\widehat{X}}\le \epsilon_1+\delta_4\le 2\epsilon_1\le \dfrac{1}{m}
		\]
		with  {$2m\epsilon_1  \le 1$}   in the last step. This contradiction concludes the proof of $(iii)$.
		\qed
{
\begin{rem}[error estimate]\label{error_estimate}
Recall $\delta_4$ from {\bf (H4)} and $\epsilon_0$ from \eqref{defn_delta}.  An algebraic manipulation in  \eqref{Newton_conv} reveals, for $\epsilon_0 m \le 1/2$, that 
\[ \|u-u_h\|_{\widehat{X}} \le \delta_4 + \frac{2 \epsilon_0}{1+\sqrt{1-2 \epsilon_0 m} } \le \delta_4 +2 \epsilon_0. \] 
In the  applications of Section \ref{sec:nse}-\ref{sec:vke},  this leads to the energy norm estimate. 
\end{rem}}
{{\begin{rem}[estimate on $\epsilon_1$]\label{rem_epsilons}
	In the applications, \eqref{defn_delta} leads to $\epsilon_0 \lesssim \delta_3+\delta_4$. This, the definition of $\epsilon_1$ in Theorem \ref{thm:existence}, \eqref{Newton_conv}, and Remark \ref{error_estimate} provide $\epsilon_1 \lesssim \delta_3+\delta_4$.
	
\end{rem}}}

	\section{A priori  error control} \label{sec:apriori}
	This section is devoted to a quasi-best approximation up to perturbations {\bf(B)}.
	Recall that the bounded bilinear form $a: X \times Y \rightarrow  \mathbb{R}$ satisfies~\eqref{cts_infsup},  the  trilinear form $\Gamma : X \times X \times Y \rightarrow \mathbb{R}$ is bounded,  and $F \in Y^{\ast}$. 
The assumptions on the discretization with $\dislin : X_h \times Y_h \rightarrow \mathbb{R}$ with  non-trivial finite-dimensional spaces $X_h$ and $Y_h$ of the same dimension $\mathrm{dim}(X_h) = \mathrm{dim}(Y_h) \in \mathbb{N}$ are encoded in the stability and quasi-optimality. The stability of $\dislin$ and \eqref{dis_Ah_infsup} {{mean}} $\alpha_{0} > 0$  and the quasi-optimality assumes $P \in L(X_h; X)$ with \eqref{quasioptimalsmootherP},  $R \in L(X_h; \widehat{X})$ with \eqref{quasioptimalsmootherR},  $S \in L(Y_h; \widehat{Y})$
,  and $Q \in L(Y_h; Y)$ (in this section, \eqref{quasioptimalsmootherQ} and~\eqref{quasioptimalsmootherS} are not employed).
Recall $\beta_1$ and $\epsilon_1$  from \eqref{eqn:beta1} and Theorem \ref{thm:existence}.
\begin{thm}[a priori error control] \label{thm:apriori}
	Let $u \in X$ be a regular root to~\eqref{eqn:p}, {let} $u_h \in X_h$ solve \eqref{eqn:dp},  {and suppose}\ref{h1},  \eqref{dis_Ah_infsup}{{-\eqref{quasioptimalsmootherP}}}, \eqref{quasioptimalsmootherR}, 
$\|u - u_h\|_{\widehat{X}} \le \epsilon_2:=\min \left\{ \epsilon_1, \frac{ \kappa  \beta_1}{ (1+ \Lambda_{\rm R})^2\|\sop\|  \|\widehat{\Gamma}\|} \right\}$, 
and $0 < \kappa <1$.  Then
			\begin{align} 
			\|u - u_h \|_{\widehat{X}} \le C_{\rm qo} \min_{x_h \in X_h} \|u - x_{h}\|_{\widehat{X}} + {\beta^{-1}_1(1-\kappa)^{-1}}{\| \widehat{\Gamma}(u,u,(\sop-Q)\bullet) \|_{Y_h^\ast} }
			\end{align}
{holds for $C_{\rm qo} = C_{\rm qo}'{\beta^{-1}_1(1-\kappa)^{-1}}(\beta_1 + 2 (1+\Lambda_{\rm R}) \|S\|\|\widehat{\Gamma}\|\|u\|_X )$} with $C_{\rm qo}':= 1 + \alpha_0^{-1}(\Lambda_{1} \Lambda_{\rm P} + \|Q^\ast A\|(1 + \Lambda_{\rm P}))$. 
		\end{thm}
\noindent The theorem {{establishes}} a quasi-best approximation result \eqref{eqn:best} for $S=Q$. The proof utilizes a quasi-best approximation result from \cite{ccnnlower2022} for linear problems.
 \begin{lem}[quasi-best approximation for linear problem \cite{ccnnlower2022}]\label{lemma:quas_eq}
If $u^\ast \in X$ and $G(\bullet)=a(u^*,\bullet) \in Y^*$, $u_h^\ast \in X_h$ and $a_h(u_h^*,\bullet) =G(Q \bullet)\in Y_h^*$,  then  {{\eqref{dis_Ah_infsup}-}}\eqref{quasioptimalsmootherP} and {\bf (H1)} imply
\begin{align} 
	\mathbf{(QO)} \quad	\|u^\ast - u_h^\ast\|_{\widehat{X}} \le C_{\rm qo}' \inf_{x_h \in X_h} \|u^\ast - x_h \|_{\widehat{X}}. \label{eqn.quasiopt}
	\end{align}
\end{lem}	
\begin{proof} This is indicated in \cite[Theorem 5.4.a]{ccnnlower2022} for Hilbert spaces and we give the proof for completeness. For any $x_h \in X_h$, the inf-sup condition \eqref{dis_Ah_infsup} leads for $e_h:=x_h-u_h^* \in X_h$ to some $\|y_h\|_{Y_h} \le 1$ such that 
\[ \alpha_0 \|e_h\|_{X_h}  \le a_h(x_h,y_h)-a_h(u_h^*,y_h). \] Since $a_h(u_h^*, y_h)= G(Qy_h)=a(u^*, Qy_h)$, this implies 
\[\alpha_0 \|e_h\|_{X_h} \le a_h(x_h,y_h) -a(Px_h, Qy_h) + a(Px_h -u^*, Qy_h) \le \Lambda_1 \|x_h-Px_h\|_{\widehat{X}} +\|Q^*A\| \|u^*-Px_h\|_{X}
\]
with {\bf (H1)},  the operator norm $\|Q^*A\|$ of $Q^*A=a(\bullet, Q\bullet)$, 
and $\|y_h\|_{Y_h} \le 1$ in the last step.  Recall \eqref{quasioptimalsmootherP} and
$\|u^*-Px_h\|_X  
 \le (1+\Lambda_{\rm P}) \|u^*-x_h\|_{\widehat{X}} $ from  \eqref{eqn:PR} to deduce
\[   \alpha_0 \|e_h\|_{X_h} \le (\Lambda_1 \Lambda_{\rm P} +(1+\Lambda_{\rm P}) \|Q^* A\|)\|u^*-x_h\|_{\widehat{X}}. \]
This and a triangle inequality $\|u^*-u_h^*\|_{\widehat{X}} \le \|e_h\|_{X_h} +\|u^*-x_h\|_{\widehat{X}}$ conclude the proof.
\end{proof}
\noindent{\it Proof of Theorem~\ref{thm:apriori}}.
Given a regular root  $u \in X$ to~\eqref{eqn:p}, $G(\bullet) := F(\bullet) - \Gamma (u,u,\bullet) \in Y^\ast$ is an appropriate right-hand side in the problem $a(u,\bullet)  = G(\bullet)$ with a discrete solution $u_{h}^\ast \in X_h$ to $\dislin(u_{h}^\ast,\bullet) = G(Q\bullet)$ in $Y_h$. {Lemma~\ref{lemma:quas_eq} implies~\eqref{eqn.quasiopt} with $u^\ast$ substituted by $u$, namely}
\begin{align} \label{eqn:quasi}
 \|u - u_{h}^\ast\|_{\widehat{X}} \le C_{\rm qo}'  \inf_{x_h \in X_h} \|u - x_{h}\|_{\widehat{X}} . 
\end{align}

\smallskip
\noindent Given the discrete solution $u_h \in X_h$ to~\eqref{eqn:dp} and the approximation $u_{h}^\ast \in X_h$ from above,  let $e_h := u_{h}^\ast - u_h \in X_h$. The stability of the discrete {problem from Theorem~\ref{dis_inf_sup_thm}} leads to the existence of some $y_h \in Y_h$ with norm $\|y_h\|_{Y_h} \le 1/\beta_h$ for $\beta_h \ge \beta_{0}$ from~\eqref{eqn:disinfsup} and 
\begin{align}
\|e_{h}\|_{X_h} = \dislin(e_h,y_h) + \widehat{\nonlin}(\rop e_h,\sop y_h) = \dislin(e_h,y_h) + \widehat{\Gamma}(u,Re_h,Sy_h) +  \widehat{\Gamma}(Re_h,u,Sy_h)  \nonumber
\end{align} 
with~\eqref{star} in the last step. 
 The definition of $u_h^*$, $G$, and \eqref{eqn:dp} show
$$\dislin(u_h^\ast,y_h) = F(Qy_h) -{\Gamma}(u,u,Qy_h) = \dislin(u_h,y_h) + \widehat{\Gamma}(\rop u_h,\rop u_h,\sop y_h) - {\Gamma}(u,u,Qy_h).$$
 The combination of the two previous displayed identities and elementary algebra show that 
\begin{align}
\|e_h \|_{X_h}  & = \widehat{\Gamma}(\rop u_h ,\rop u_h,\sop y_h) - \widehat{\Gamma}(u,u,Sy_h) + \widehat{\Gamma}(u,\rop e_h,\sop y_h) + \widehat{\Gamma}(\rop e_h,u,\sop y_h) + \widehat{\Gamma}(u,u,(S - Q)y_h)  \nonumber \\
&= \widehat{\Gamma}(u - Ru_h,u - Ru_h, Sy_h)+ \widehat{\Gamma}(u,Ru_h^\ast - u,Sy_h) + \widehat{\Gamma}(Ru_h^\ast - u,u,Sy_h) +\widehat{\Gamma}(u,u,(S - Q)y_h) \nonumber\\
&\le (\|S\| \| \widehat{\Gamma} \| \| u - Ru_h \|_{\widehat{X}}^2 + 2 \|u \|_{X} \| S\| \| \widehat{\Gamma} \| \| u - R u_h^*\|_{\widehat{X}} + \| \widehat{\Gamma}(u,u,(S-Q)\bullet) \|_{Y_{h^\ast}} )/\beta_h
\label{eqn:qh_in_gammah}
\end{align}
with {the} boundedness of $\widehat{\Gamma}(\bullet,\bullet,\bullet)$ and $\|y_h\|_{Y_h} \le 1/\beta_h$ in the last step.  This, $\|u - Ru_h \|_{\widehat{X}} \le (1 + \lamr)\| u - u_{h} \|_{\widehat{X}}$ (resp. $\|u - Ru_h^* \|_{\widehat{X}} \le (1 + \lamr)\| u - u_{h}^* \|_{\widehat{X}}$) from~\eqref{eqn:PR}, $\beta_1 \le \beta_h$,  and a triangle inequality show
\begin{align}
\beta _1 \|u - u_{h}\|_{\widehat{X}} \le{}& \left(\beta_1 +2 (1 + \lamr)\|\sop\|\| \widehat{\Gamma}\|\|u\|_{\widehat{X}} \right) \|u - u_{h}^\ast\|_{\widehat{X}} + \| \widehat{\Gamma}(u,u,(\sop-Q)\bullet) \|_{Y_h^\ast} \nonumber \\
&+ (1 + \lamr)^2 \|\sop \|\| \widehat{\Gamma}\| \|u - u_h \|_{\widehat{X}}^2.
\label{eqn:penul_est}
\end{align}
Recall the assumption on $\|u-u_h\|_{\widehat{X}} \le \epsilon_2$ to absorb the last term and obtain

\begin{align}
\|u - u_{h}\|_{\widehat{X}} \le \dfrac{(\beta_1 + 2 (1+ \Lambda_{\rm R}) \|S\| \| \widehat{\Gamma} \| \|u\|_{{X}}) \|u - u_{h}^\ast\|_{\widehat{X}} + \| \widehat{\Gamma}(u,u,(S-Q)\bullet) \|_{Y_h^\ast}}{\beta_1 - \epsilon_2 (1+\Lambda_{\rm R})^2 \|S\| \|\widehat{\Gamma}\|}.
\end{align}
This,  the definition of $\epsilon_2$, and \eqref{eqn:quasi} conclude the proof. 
\qed
{{
		\begin{rem}[estimate on $\epsilon_2$]
	The assumption of Theorem \ref{thm:apriori} and Remark \ref{rem_epsilons} reveal $\epsilon_2 \le \epsilon_1\lesssim \delta_3+\delta_4$ for the  applications of Section \ref{sec:nse}-\ref{sec:vke}.	\end{rem}}}


\section{Goal-oriented error control} \label{sec:goal-oriented}
This section proves an a priori error estimate in weaker Sobolev norms based on a  duality argument. 
Suppose $Y $ is reflexive throughout this section so that,  given any $G \in X^{\ast}$,  there exists a unique solution $z \in Y$ to the  dual linearised problem 
\begin{align}\label{dual_linear}
a(\bullet,z)+b(\bullet,z)=G(\bullet) \mbox{ in } X^*.
\end{align}
{Recall $N$ from~\eqref{eqccdefN}, $A$ and $B$ from Table \ref{tab:bilinear_forms} with~\eqref{star}, $P$, $Q$, $R$, and $S$   with~\eqref{quasioptimalsmootherP}--\eqref{quasioptimalsmootherS},} and  \ref{h1hat}  from Section~\ref{sec:mainresults}.  Since $u \in X$ is a regular root, the derivative $A+B \in {L(X;Y^\ast)}$  of $N$ evaluated at $u$ is a bijection and so is its dual operator 
$A^*+B^* \in L(Y; X^*)$. 

\medskip
%
%
\begin{thm}[goal-oriented error control] 
	\label{thm:lower}
Let $u \in X$ be a regular root to \eqref{eqn:p}  and let $u_{h} \in X_h$ (resp. $z \in Y$) solve~\eqref{eqn:dp} (resp.~\eqref{dual_linear}). Suppose~\ref{h1hat} and ~\eqref{quasioptimalsmootherP}-\eqref{quasioptimalsmootherS}.  
Then, any $G \in X^{\ast}$ and any $z_h \in Y_h$ satisfy
	\begin{align*}
	G(u-Pu_h) \le{}& 
  \omega_1(||u||_{X},||u_h||_{X_h})   \|u - u_h \|_{\widehat{X}}   \| z - z_h \|_{\widehat{Y}}  + \omega_2(\|z_h\|_{Y_h})  \|u - u_{h} \|^2_{\widehat{X}}  \nonumber \\
& \quad + \widehat{\Gamma}(u,u,(S-Q)z_h) +
\widehat{\Gamma}(Ru_h,Ru_h,Q z_h) -{\Gamma}(Pu_h,Pu_h,Qz_h)
	\end{align*}
with the weights
\begin{align} 
 \omega_1(||u||_{X},||u_h||_{X_h}) & := (1 + \Lambda_{\rm P})(1 + \Lambda_{\rm Q}) (\|A\| +2\|\Gamma\|\|u \|_{X}) 
 + \Lambda_{5} + (1+ \Lambda_{\rm R})(\Lambda_{\rm S} +\Lambda_{\rm Q})  \nonumber \\
& \times\| \widehat{\Gamma}\| (\|Ru_{h}\|_{\widehat{X}} + \|u\|_{X}), \quad
 \omega_2(\|z_h\|_{Y_h}) := \| {\Gamma} \| (1+\Lambda_{\rm P})^2 \|Q z_h\|_{Y}.  \qquad \label{weights}
\end{align}

\end{thm}

	\begin{proof} 
Since $z \in Y$ solves  \eqref{dual_linear},  elementary algebra with \eqref{eqn:p}, \eqref{eqn:dp}, and  any $z_h \in Y_h$ lead to 
			\begin{align}
G(u-Pu_h)&=(a+\nonlin)(u-Pu_h,z) 
		=(a+\nonlin)(u-Pu_h,z-Qz_h)  +\nonlin(u-Pu_h,Qz_h) 
		\nonumber \\
& \qquad +\big(\dislin(u_h,z_h)  -a(Pu_h,Qz_h)\big)+ \widehat{\Gamma}(Ru_h,Ru_h,Sz_h)-\Gamma(u,u,Qz_h).  \label{est}
\end{align}
	The first term  $
(a+\nonlin)(u-Pu_h,z-Qz_h)$ on the right-hand side of \eqref{est}  is bounded by
		\begin{equation}
		\hspace{-0.03cm} (\|A\|+2\|\Gamma\|\|u \|_{X}) \| u-Pu_h \|_{X} \| z-Qz_h \|_{Y} 
		{\le}(\|A\|+2\|\Gamma\|\|u \|_{X}) (1+\Lambda_{\rm P})(1+\Lambda_{\rm Q}) \| u-u_h \|_{\widehat{X}} \| z-z_h \|_{\widehat{Y}}\label{first}
		\end{equation}
with~\eqref{eqn:PR}-\eqref{eqn:QS}  in the last step. 
The hypothesis~\ref{h1hat} controls the third  term on the right-hand side of \eqref{est},  namely
		\begin{align}
		\dislin(u_h,z_h) - a(Pu_h,Qz_h) &\le \Lambda_{5}\| u-u_h \|_{\widehat{X}} \| z-z_h \|_{\widehat{Y}}. 
\label{second}
		\end{align}
Elementary algebra with~\eqref{star} shows that the remaining terms $\widehat{\Gamma}(Ru_h,Ru_h,Sz_h)-\Gamma(u,u,Qz_h)+\nonlin(u-Pu_h,Qz_h) $  on the right-hand side of \eqref{est} can be re-written as 
		\begin{align}
		\hspace{-0.25cm}  \widehat{\Gamma}(Ru_h,Ru_h,(S-Q)z_h)
		+\widehat{\Gamma}(Ru_h,Ru_h,Qz_h) -{\Gamma}(Pu_h,Pu_h,Qz_h)
		+{\Gamma}(u-Pu_h,u-Pu_h,Qz_h).  \; \; \quad
		\label{eqn:non-linear_splitting}
		\end{align}
	
		\noindent  Elementary algebra with the first term on the right-hand side of~\eqref{eqn:non-linear_splitting} reveals
$$\widehat{\Gamma}(Ru_h, Ru_h,(S-Q)z_h) = \widehat{\Gamma}(Ru_h - u,Ru_h,(S-Q)z_h) + \widehat{\Gamma}(u,Ru_h - u,(S-Q)z_h) + \widehat{\Gamma}(u,u,(S-Q)z_h).$$ 
The boundedness of $\widehat{\Gamma}(\bullet,\bullet,\bullet)$, \eqref{quasioptimalsmootherQ},  \eqref{quasioptimalsmootherS}, and \eqref{eqn:PR}  show 
\begin{align*}
\widehat{\Gamma}(Ru_h - u,Ru_h,(S-Q)z_h)& = 
\widehat{\Gamma}(Ru_h - u,Ru_h,(S-I)z_h) +\widehat{\Gamma}(Ru_h - u,Ru_h,(I-Q)z_h) \\
& \le (\Lambda_{\rm S} +\Lambda_{\rm Q} ) 
 \| \widehat{\Gamma} \| (1+ \Lambda_{\rm R})\|Ru_{h}\|_{\widehat{X}} \|u - u_{h}\|_{\widehat{X}} \| z-z_h \|_{\widehat{Y}}  . \\
\widehat{\Gamma}(u,Ru_h - u,(S-Q)z_h) &\le  (\Lambda_{\rm S} +\Lambda_{\rm Q} ) 
 \| \widehat{\Gamma} \| (1+ \Lambda_{\rm R})\|u\|_{{X}} \|u - u_{h}\|_{\widehat{X}} \| z-z_h \|_{\widehat{Y}}  . 
\end{align*}

\noindent The boundedness of  ${\Gamma}(\bullet,\bullet,\bullet)$  and \eqref{eqn:PR} lead to 
$$		{\Gamma}(u-Pu_h,u-Pu_h,Qz_h) \le \| {\Gamma} \|(1+\Lambda_{\rm P})^2 \|u - u_{h}\|^2_{\widehat{X}} \|Qz_{h}\|_{Y}.$$
  A combination  of  the above estimates of the terms in \eqref{est} concludes the proof.
	\end{proof}

\noindent An abstract a priori estimate for error control in weaker Sobolev norms concludes this section.
\begin{thm}[a priori error estimate in weaker Sobolev norms] 
	\label{cor:lower}
	Let $X_{\rm s}$ be a Hilbert space with $X \subset X_{\rm s}$. Under the assumptions of Theorem~\ref{thm:lower},  any $z_h \in Y_h$ satisfies
	\begin{align}
	\| u - u_h \|_{X_{\rm s}} \le{}& \omega_1(||u||_{X},||u_h||_{X_h})  \|u - u_h \|_{\widehat{X}}   \| z - z_h \|_{\widehat{Y}}   + \omega_2(\|z_h\|_{Y_h})  \|u - u_{h} \|^2_{\widehat{X}} + \|u_h - Pu_h\|_{X_{\rm s}}  \nonumber \\
	& + \widehat{\Gamma}(u,u,(S-Q)z_h) +
\widehat{\Gamma}(Ru_h,Ru_h,Q z_h) -{\Gamma}(Pu_h,Pu_h,Q z_h).
	\end{align}
\end{thm}

\begin{proof}
	Given $u - Pu_h \in X \subset X_{\rm s}$, a corollary of the Hahn-Banach extension theorem leads to some $G \in X_{\rm s}^\ast \subset X^\ast$ with norm $\|G\|_{X_{\rm s}^\ast} \le 1$ in $X_{\rm s}^\ast$ and 
	$
	G(u - Pu_h) = \|u - Pu_h \|_{X_{\rm s}}
	$
	{{\cite{brezis}}}. This, a triangle inequality, and Theorem~\ref{thm:lower} conclude the proof.
\end{proof}

\section[Auxiliary Results]{Auxiliary results for applications}\label{sec:notations}
\subsection{General notation}Standard notation of  Lebesgue and Sobolev spaces, 
their norms, and $L^2$ scalar products  applies throughout the paper
such as the abbreviation $\|\bullet\|$ for $\|\bullet\|_{L^2(\Omega)}$.
For real $s$, $H^s(\Omega)$ denotes the Sobolev space endowed with the  Sobolev-Slobodeckii semi-norm (resp. norm) $|\bullet|_{H^{\boldmath{s}}(\Omega)}$  (resp.  $\|\bullet\|_{H^{\boldmath{s}}(\Omega)}$ )
\cite{Grisvard}; $H^s(K):= H^s({\rm int}(K))$ abbreviates the Sobolev space with respect to the interior ${\rm int}(K)\ne \emptyset$ of a triangle $K$.  The closure of $D(\Omega)$ in $H^s(\Omega)$ is denoted  by $H^s_0(\Omega)$   and $H^{-s}(\Omega)$ is the dual of $H^s_0(\Omega)$. The semi-norm and norm in $W^{s,p} (\Omega)$,  $1 \le p \le \infty$, are denoted by $|\bullet|_{W^{s,p}(\Omega)}$ and $\|\bullet\|_{W^{s,p}(\Omega)}$. The Hilbert space $V := H_0^2(\Omega)$ is endowed with the energy norm 
$\displaystyle\trinl\bullet\trinr := |\bullet|_{H^2(\Omega)}$.  The product space $H^s(\O) \times H^s(\O)$ (resp. ~$L^p(\O) \times L^p(\O)$) is denoted by ${\bf H}^s(\Omega)$ (resp. ~${\bf L}^p(\Omega)$)   and  $\bv =: V \times V$.   The energy norm in the product space ${\bf H}^2(\Omega)$ is also denoted by  $\trinl\bullet\trinr$ and is $(\trinl \varphi_1 \trinr^2+\trinl \varphi_2\trinr^2)^{1/2}$   for all $\Phi=(\varphi_1,\varphi_2)\in {\bf H}^2(\O)$. The norm on ${\bf W}^{s,p}(\Omega)$ is denoted by $\|\bullet\|_{{\bf W}^{s,p}(\Omega)}$. Given any function $v \in L^2(\omega)$,  define the integral mean $ \fint_\omega v \dx:= {1/ |\omega| }\int_\omega v \dx$; {where $|\omega|$ denotes the area of $\omega$}. 
The notation $A \lesssim B$ (resp.  $A \gtrsim B$) abbreviates $A \leq CB$ (resp. $A \geq CB$)  for some positive generic constant $C$, 
which depends exclusively on $\Omega$ and the shape regularity of a triangulation $\T$;
 $A\approx B$ abbreviates $A\lesssim B \lesssim A$. 

\medskip

\noindent
{\bf{Triangulation. }}Let ${\cT}$ denote a shape regular triangulation of the {polygonal Lipschitz domain} $\Omega$ with boundary $\partial \Omega$ {into compact triangles} and $\mathbb{T}(\delta)$ be a set of uniformly shape-regular triangulations  $\cT$ with maximal mesh-size smaller than or equal to $\delta>0$.  {Given  $\cT \in \mathbb{T}$}, define the piecewise constant mesh function $h_{\cT}(x)=h_K={\rm diam}  (K)$ 
for all $x \in K \in \cT$, and set $h_{\rm max} :=\max_{K\in \cT}h_K$. The set of all interior vertices (resp. boundary vertices) of the triangulation $\cT$ is denoted by $\mathcal{V}(\Omega)$ (resp.  $\mathcal{V}(\partial\Omega)$) {and $\mathcal{V} := \mathcal{V}(\O) \cup \mathcal{V}(\partial \O)$}.  Let $\cE(\Omega)$ (resp. $\cE(\partial\Omega)$) denote the set of all interior edges (resp. boundary edges) in $\cT$. 
Define a piecewise constant edge-function on $\cE:=\cE(\Omega)\cup \cE(\partial\Omega)$ by $h_{\cE}|_E=h_E={\rm diam}(E)$ for any $E\in \cE$.  For a positive integer $m$,  define the Hilbert (resp.  Banach) space  $H^m(\cT) \equiv  \prod\limits_{K \in \cT} H^m(K)$ (resp.  $W^{m,p}(\cT)  \equiv  \prod\limits_{K \in \cT}  W^{m,p}(K)) $.  
The triple norm  $\trinl \bullet \trinr:=|\bullet|_{H^{m}(\Omega)}$ is the energy norm
 and  $\trinl \bullet \trinr_{\text{pw}}:=|\bullet|_{H^{m}(\cT)}:=\| D^m_\text{pw}\bullet\|$ 
 is its piecewise version with the piecewise partial derivatives  $D_\text{pw}^m$ of order $m\in  {\mathbb N}$. 
{For $1<s<2$, the piecewise Sobolev space $H^s(\T)$ is the product space  $\prod_{T\in\T} H^s(T)$  defined as $ \{ v_{\pw} \in L^2(\Omega): \forall T\in\T, \;   v_{\pw}|_T\in H^s(T)\}$ and is equipped with the Euclid norm of those contributions $ \|\bullet  \|_{H^s(T)}$ for all $T\in\T$. 
For  ${{s}}=1 +\nu$ with $0<\nu<1$,  the 2D Sobolev-Slobodeckii norm \cite{Grisvard}
of $f \in H^s(\Omega)$ reads $\|f\|_{H^{\boldmath{s}}(\Omega)}^2:=
\|f\|_{H^{1}(\Omega)}^2+
|f|_{H^{\nu}(\Omega)}^2$ and
\begin{align} \label{eq:sobslobo}
|f|_{H^{s}(\Omega)} &:= \left( 
\sum_{|\beta|=1} \int_\Omega  \int_\Omega 
\frac{|\partial^\beta f(x) - \partial^\beta f(y)|^2 }{|x-y|^{2 +2 \nu} } \dx \dy   \right)^{1/2}. 
\end{align}}
The piecewise version of the energy norm in $H^2(\cT)$ reads $\trinl \bullet \trinr_{\text{pw}}:=|\bullet|_{H^{2}(\cT)}:=\| D^2_\text{pw}\bullet\|$ with the  piecewise Hessian $D_\text{pw}^2$. { 
The  curl  of a scalar function $v$ is defined by   ${\rm Curl} \; v =\big(-\partial v /\partial y ,-\partial v /\partial x \big)^T$ and its piecewise version is denoted by  ${\rm Curl }_{\rm pw}$.} The seminorm (resp. norm) in $W^{m,p}(\cT)$
 is denoted by $|\bullet|_{W^{m,p}(\cT)}$ (resp.  $\|\bullet\|_{W^{m,p}(\cT)}$).  Define the jump $\jump{\varphi}_E:=\varphi|_{K_+}-\varphi|_{K_-}$ and the average $\langle \varphi \rangle _E:=\half\left(\varphi|_{K_+}+\varphi|_{K_-}\right)$ across the interior edge $E$ of $\varphi\in H^1(\cT)$ of the adjacent triangles  $K_+$ and $K_-$. Extend the definition of the jump and the average to an edge on boundary by $\jump{\varphi}_E:=\varphi|_E$ and $\langle \varphi\rangle_E:=\varphi|_E$ for $E\in \cE(\partial\Omega)$.  For any vector function, the jump and the average are understood component-wise. {Let $\Pi_k$ denote the $L^2(\Omega)$ orthogonal projection onto the piecewise polynomials $\displaystyle P_k(\cT):=\left\{v \in L^2(\Omega):\;\forall\,K  \in \cT,\;\;v|_{K}\in P_k(K) \right\}$ of degree at most $k \in \mathbb{N}_0$.} (The notation $\trinl \bullet \trinr_{\rm pw}$,  $\Pi_K$, and $V_h$ below hides the dependence on $\cT \in \mathbb{T}$.)
\vspace{-0.1in}
\subsection{Finite element function spaces and discrete norms} \label{sec:norms}
\noindent This section introduces the discrete spaces and norms for the Morley/dG/$C^0$IP/WOPSIP schemes. The Morley finite element space \cite{Ciarlet} reads
\begin{eqnarray*}
{\M}(\T):=\left\{ v_\M\in P_2(\T){{\Bigg |}}
{ \begin{aligned}
&\; v_\M \text{ is continuous at the vertices and its normal derivatives } 
 \nu_E\cdot D_{\rm pw}{ v_\M}  \text{ are } \\
& \text{  continuous at the midpoints of interior edges},\; v_\M \text{  vanishes at the vertices  }\\
& \text { of } \partial \Omega \text{ and }
\; \nu_E\cdot D_{\rm pw}{ v_\M} \text{  vanishes at the midpoints of boundary edges}
\end{aligned}}\right\}.
\end{eqnarray*}
The semi-scalar product $a_{\rm pw}$ is defined by the piecewise Hessian $D^2_{\rm pw}$, for all  $v_{\rm pw}, w_{\rm pw} \in H^2(\cT)$ as
\begin{align}
\hspace{-0.7cm} a_{\text{pw}}(v_{\rm pw},w_{\rm pw})
:={}& \int_\Omega D_{\rm pw}^2 v_{\rm pw}:D_{\rm pw}^2 w_{\rm pw}\dx. \label{eqccnerwandlast1234a} 
\end{align}
The bilinear form $a_{\rm pw}(\bullet,\bullet)$ induces a piecewise $H^2$ seminorm $\trinl \bullet \trinr_{\rm pw} = a_{\rm pw}(\bullet,\bullet)^{1/2}$ that is  a norm on $V+\M(\cT)$ \cite{ccnn2021}. The piecewise Hilbert space $ H^2(\cT)$ is endowed with a norm $\|\bullet \|_h$~\cite{ccnngal} defined by 
\begin{align} 
\|v_{\pw}\|_{h}^2 &:= \trinl v_{\pw} \trinr_{\pw}^2 + j_h(v_{\pw})^2 \text{ for all } v_{\pw} \in H^2(\cT), \\ \label{eqn:jh_defn}
j_h(v_{\pw})^2 & := \sum_{E \in \mathcal{E}} \sum_{z \in \mathcal{V}(E)} h_{E}^{-2} |\jump{v_{\rm pw}}_E(z)|^2 + \sum_{E \in \mathcal{E}} \left| \fint_{E} \jump{\partial v_{\pw}/\partial \nu_{E}}_E\,\mathrm{d}s \right|^2   
\end{align}
with the jumps $\jump{v_{\pw}}_{{E}}(z) = v_{\pw}\vert_{\omega(E)} (z)$ for $z \in \mathcal{V}(\partial \Omega)$; the edge-patch $\omega(E):=\text{\rm int}(K_+\cup K_-)$ of the interior edge 
$E=\partial K_+\cap\partial K_-\in\E(\Omega)$  is the interior of the union 
$K_+\cup K_-$ of the neighboring triangles $K_+$ and $K_-$, and $\jump{\frac{\partial v_{\pw}}{\partial \nu_{E}}}_{E} = \frac{\partial v_{\pw}}{\partial \nu_{E}} \vert_{E}$ for $E \in \mathcal{E}(\partial \Omega)$ at the boundary with jump partner zero owing to the homogeneous boundary conditions.


 \noindent For all  $v_{\rm pw}, w_{\rm pw} \in H^2(\cT)$ and  parameters $\sigma_1,\sigma_2>0$ (that will be chosen sufficiently large but fixed in applications),  define $c_\dg(\bullet, \bullet)$ and the mesh dependent dG norm 
$\|\bullet\|_{\dg}$   by 
\begin{align}
c_\dg(v_{\rm pw},w_{\rm pw}) & := {\sum_{E\in\cE}\frac{\sigma_1}{h_E^3}\int_E \jump{v_{\rm pw}}_E \jump{w_{\rm pw}}_E\ds}{}+{}
\sum_{E\in\cE} \frac{\sigma_2}{h_E}\int_E\jump{ \partial v_{\rm pw}/\partial  \nu_E}_E  \jump{ \partial w_{\rm pw}/\partial \nu_E}_E \ds, \quad \label{eqn:cdg} \\
\|v_{\rm pw}\|_{\dg}^2&:= \trinl v_{\rm pw} \trinr_{\rm pw}^2+c_\dg(v_{\rm pw},v_{\rm pw}). \label{dgnorm}
\end{align} 
The discrete space for the $C^0$IP scheme is $S^2_0(\cT) := P_2(\cT) \cap H^1_0(\Omega)$.  The restriction of $\|\bullet \|_{\rm dG}$ to $H^1_0(\Omega)$ with a stabilisation parameter  $\sigma_{\rm IP}>0$ defines the norm for the $C^0$IP scheme  below,
\begin{align}
c_{\rm IP}(v_{\rm pw},w_{\rm pw}):=\sum_{E \in \E}  \frac{\sigma_\ip}{h_E} \int_E \jump{{\partial v_{\rm pw}}/{\partial \nu_E}}  
\jump{{\partial w_{\rm pw}}/{\partial \nu_E}}{}{\rm ds}, 
\; \|v_{\rm pw}\|_{\rm IP}^2:=\trinl v_{\rm pw} \trinr_{\rm pw}^2 + c_{\rm IP}(v_{\rm pw},v_{\rm pw}).  \qquad \label{eqn:cip} 
\end{align}
\noindent  For all $v_{\rm pw}, w_{\rm pw}  \in H^2(\cT)$ the WOPSIP norm  $\|\bullet\|_{\rm P}$ is defined by 
\begin{align} 
c_{\rm P}(v_{\rm pw},w_{\rm pw}) & := \sum_{E \in \mathcal{E}} 
\sum_{z \in \mathcal{V}(E)} h_{E}^{-4} \; ( \;  \jump{v_{\rm pw}}_E(z))   
{\displaystyle (\jump{w_{\rm pw}}_E(z)) } \nonumber \\
& \qquad 
{+}\displaystyle \sum_{E \in \mathcal{E}} h_{E}^{-2}
\fint_{E} \jump{ {\partial v_{\rm pw}}/{\partial \nu_{\rm E}}}\ds 
\fint_{E} \jump{ {\partial w_{\rm pw}}/{\partial \nu_{\rm E}}} \ds, \label{eqn:cp} \\
 \|v_{\rm pw}\|_{\rm P}^2& :=\trinl v_{\rm pw} \trinr_{\rm pw}^2 +c_{\rm P}(v_{\rm pw} , v_{\rm pw} ).
\label{eqn:norm_wopsip}
\end{align}
The discrete space for dG/WOPSIP schemes is $P_2(\T)$. The discrete norms $\trinl \bullet \trinr_{\rm pw}$,   $\|\bullet \|_{\rm dG}$  and $ \|\bullet \|_{\rm IP}$ are all equivalent to $\|\bullet \|_{h}$  on $V+V_h$ for $V_h \in \{ \mathrm{M}(\cT), P_2(\cT),S^2_0(\cT) \}$.  In comparison to $j_h(\bullet)$,  the jump contribution in $\|\bullet \|_{\rm P}$ {{involves}} smaller negative powers of the mesh-size and so {$j_h(v_{\pw})^2 \lesssim c_{\rm P}(v_{\pw},v_{\pw})$ (with $h_E \le {\rm diam}(\O) \lesssim 1$); but} there is no equivalence of $\|\bullet\|_h$ with $\|\bullet\|_{\rm P}$ in $V+P_2(\cT)$. 
\begin{lem}[{Equivalence of norms \cite[Remark 9.2]{ccnnlower2022}}] \label{lem:equivalence}
	It holds $\|\bullet \|_{h} = \trinl \bullet \trinr_{\rm pw}$ on $V + \mathrm{M}(\cT)$, $\| \bullet \|_{h} \approx \|\bullet \|_{\rm dG} { \lesssim \|\bullet \|_{\rm P}}$ on $V + P_2(\cT)$, and $\| \bullet \|_{h} \approx \|\bullet \|_{\rm IP}$ on $V + S^2_0(\cT)$.
\end{lem} 
%
\subsection{Interpolation and Companion operators} \label{sec:inter_comp} 
The classical Morley interpolation operator $I_\M$ is generalized from $H^2_0(\Omega)$ to the piecewise $H^2$ functions by averaging in \cite{ccnnlower2022}.
\begin{defn}[{Morley interpolation~\cite[Definition 3.5]{ccnnlower2022}}]   \label{def:morleyii}
	Given any $v_{\pw} \in H^2(\cT)$, define $I_{\rm M}v_{\pw} := v_{\rm M} \in {\rm M}(\cT)$ by the degrees of freedom as follows.  For any interior vertex $z \in \mathcal{V}(\cT)$ with the set of attached triangles $\cT(z)$ of cardinality $|\cT(z)| \in \mathbb{N}$ and for any interior edge {{$E \in \mathcal{E}(\Omega)$}} with a mean value operator $\langle \bullet \rangle_{E}$
set 
	\begin{align}
	v_{\rm M}(z) := |\cT(z)|^{-1} \sum_{K \in \cT(z)} (v_{\pw}\vert_{K})(z)\;\;\text{ and } \fint_{E} \dfrac{\partial v_{\rm M}}{\partial \nu _{\rm E}}\,\mathrm{d}s := \fint_{E} \left\langle \dfrac{\partial v_{\rm pw}}{\partial \nu_{E}} \right\rangle\,\mathrm{d}s.
	\end{align} 
	The remaining degrees of freedom at vertices and edges on the boundary are set zero owing to the homogeneous boundary conditions.
\end{defn}



\begin{lem}[{interpolation {{error}} estimates \cite[Lemma 3.2, Theorem 4.3]{ccnnlower2022}}] \label{lemma:interpoltion_IM}
	Any $v_{\pw} \in H^2(\cT)$ and its Morley interpolation $I_{\rm M} v_{\pw} \in \mathrm{M}(\cT)$ satisfy 
		\begin{enumerate}[label= ${ (\alph*)}$,ref=$\mathrm{(\alph*)}$,leftmargin=\widthof{(C)}+3\labelsep]
			\item $\displaystyle 
			\sum_{m=0}^2     | h_{\cT}^{m-2}  (v_{\rm pw}- I_{\rm M} v_{\rm pw}) |_{H^m(\cT)}
			\lesssim \| (1-\Pi_0)D^2_{\rm pw} v_{\rm pw}\|
			+ j_h(v_{\rm pw}) { \lesssim} \| v_{\rm pw}\|_h;$ 
		\item $\displaystyle \sum_{m=0}^2 |h_{\cT}^{m-2} (v_{\pw} - I_{\rm M} v_{\pw} )|_{H^m(\cT)} \approx \min_{w_{\rm M} \in \mathrm{M}(\cT)} \|v_{\pw} - w_{\rm M} \|_h \approx \min_{w_{\rm M} \in {\rm M}(\cT)} \sum_{m=0}^2 |h_{\cT}^{m-2} (v_{\pw} - w_{\mathrm{M}})|_{H^m(\cT)};$\item  the integral mean property of the Hessian, $D^2_{\text{\rm pw}} I_\cM =\Pi_0 D^2 \text{ in } V;$
  \item  $\trinl v- I_{\rm M}v \trinr_{\pw} \lesssim h_{\rm max}^{t-2} \|v \|_{H^{t}(\Omega)}\;\text{ for all } v \in H^{t}(\Omega) \text{ with } 2 \le t \le 3.$ 
		\end{enumerate}
\end{lem}


\noindent Let $HCT(\cT)$ denote the Hsieh-Clough-Tocher finite element space{{~\cite[Chapter 6]{Ciarlet}}}. 
\begin{lem}[right-inverse 
\cite{DG_Morley_Eigen,ccnnlower2022,ccnn2021}] \label{hctenrich} 
	There {{exists}} a linear map $J:\cM(\cT) \rightarrow (HCT(\cT) + P_8(\cT))\cap H^2_0(\Omega)$ such that any $v_{\rm M} \in {\rm M}(\cT)$ and any $v_2 \in P_2(\cT) $ satisfy {$(a)$}-{{$(h)$}}.

\medskip
	\noindent (a) $Jv_{\cM}(z) {=}v_{\cM}(z)$ for any $z \in \mathcal{V}$; \; \\ \noindent{(b)} $\nabla (Jv_{\cM})(z) = |\cT(z)|^{-1} \sum_{K \in \cT(z)}(\nabla v_{\cM} \vert_{K})(z)$ for {$z \in \mathcal{V}(\Omega)$}; \\
	\noindent	(c) \: $\fint_{E} \partial Jv_{\cM}/\partial \nu_{E} \mathrm{d}s = \fint_{E} \partial v_{\cM}/\partial \nu_{E} \mathrm{d}s$ for any $E \in \mathcal{E}$; \:\\
	\noindent {(d)} \; $v_{\cM} - J v_{\cM} \perp P_2(\cT)$ in $L^2(\Omega)$; \\
	\noindent	(e)  \: $\displaystyle{\sum_{m=0}^2 \| h_{\cT}^{m-2} D_{\rm pw}^m (v_{\rm M} - J v_{\rm M}) \| {\lesssim}\min_{v \in V} {|\!|\!|v_{\cM} - v |\!|\!|_{\rm pw}}}$; 
\noindent

\noindent (f) \;
	$ \displaystyle \|v_2 - JI_{\rm M} v_2 \|_{H^t(\cT)} \lesssim h_{\mathrm{max}}^{2-t} \min_{v \in V} \| v_2 - v \|_h \text{ holds for } \; 0\le t \le 2$;
 
 \noindent {(g)\;$\displaystyle{\sum_{m=0}^{2}  \| h_\cT^{m-3} D^m_{\rm pw}((1-I_\M)  v_2) \| 
+ \sum_{m=0}^2  \|h_{\cT}^{m-2}   D^m_{\rm pw}((1- J )I_{\rm M}v_2)\|    \lesssim  \min_{v \in V} \|v-v_2\|_{\rm P}}$;}

\noindent {(h)\;$ \displaystyle |v_2-JI_{\rm M}v_2|_{W^{1,2/(1-t)}(\cT)}  \lesssim h_{\rm max}^{1-t} \min_{v \in V}\|v-v_2\|_{h}$ holds for $0 <t<1$.}
\end{lem}
 
     \noindent \emph{Proof of $(a)$-$(f)$.} This is included in \cite{DG_Morley_Eigen,ccnn2021}, \cite[Lemma 3.7, Theorem 4.5]{ccnnlower2022}. \qed
     
\noindent \emph{Proof of $(g)$.} The  inequality  $\sum_{m=0}^{2}  \| h_\cT^{m-3} D^m_{\rm pw}((1-I_\M)  v_2) \|  \lesssim  \| v - v_2\|_{\rm P}$  follows as in the proof of Lemma 10.2 in ~\cite{ccnnlower2022}.  Lemma~\ref{hctenrich}.e and a triangle inequality show 
\begin{align*}
\sum_{m=0}^2  \|h_{\cT}^{m-2}   D^m_{\rm pw}(1- J )I_{\rm M}v_2\| 
\lesssim  \trinl I_{\rm M} v_2  -  v \trinr_{\rm{pw}}  \le  \trinl I_{\rm M} v_2  -  v_2 \trinr_{\rm{pw}} + \trinl v_2 - v \trinr_{\rm pw}.
\end{align*}
Since $  \trinl I_{\rm M} v_2  -  v_2 \trinr_{\rm{pw}}   \le h_{\rm max} \trinl h_{\cT}^{-1}
 (I_{\rm M} v_2 - v_2) \trinr_{\rm{pw}} \lesssim h_{\rm max} \| v - v_2 \|_{\rm P}$ {{from the first part of $(g)$ with $m=2$}},  the above displayed estimate, and $\trinl \bullet \trinr_{\rm pw} \le \| \bullet \|_{\rm P}$ conclude the proof of $(g)$.  \qed

{ \noindent \emph{Proof of $(h).$} An inverse estimate \cite[Lemma 12.1]{ErnGuermond_FE1}, \cite[Lemma 4.5.3]{Brenner}, \cite[Theorem 3.2.6]{Ciarlet} on each triangle $\widehat{T}$ in the HCT subtriangulation $\widehat{\cT}$ of $\cT$ in each component of $g:=\nabla_{\pw}(v_2-JI_{\rm M}v_2)$ reads $\|g\|_{L^{2/(1-t)}(\widehat{T})} \le C_{\rm inv}h_{\widehat{T}}^{-t}\|g\|_{L^2(\widehat{T})}.$ Consequently,
\[C_{\rm inv}^{-1}\|g\|_{L^{2/(1-t)}(\O)}\le \left(\sum_{\widehat{T} \in \widehat{\cT}}\|h_{\widehat{T}}^{-t}g\|_{L^2(\widehat{T})}^{2/(1-t)}\right)^{(1-t)/2}\le \left(\sum_{\widehat{T} \in \widehat{\cT}}\|h_{\widehat{T}}^{-t}g\|_{L^2(\widehat{T})}^{2}\right)^{1/2} \]
with $\|\bullet \|_{\ell^{2/(1-t)}} \le \|\bullet \|_{\ell^2}$ in the sequence space $\R^\mathbb{N}$ ($\ell^p$ is decreasing in $p \ge 1$) in the last step. With the shape regularity $h_{\widehat{\cT}} \approx h_{\cT}$, this reads
\begin{equation}\label{eqn:suminv}
|v_2-JI_{\rm M}v_2|_{W^{1,2/(1-t)}(\cT)} \lesssim |h_\T^{-t}(v_2-JI_{\rm M}v_2)|_{H^1(\T)}.
\end{equation}
Since $I_{\rm M}(v_2-JI_{\rm M}v_2)=0$ 
by Lemma~\ref{hctenrich}, Lemma~\ref{lemma:interpoltion_IM}.a provides
\begin{equation}\label{eqn:ortho}
|h_\T^{-t}(v_2-JI_{\rm M}v_2)|_{H^1(\T)}\le h_{\max}^{1-t}|h_\T^{-1}(v_2-JI_{\rm M}v_2)|_{H^1(\T)}\lesssim h_{\max}^{1-t}\|
v_2-JI_{\rm M}v_2\|_h.
\end{equation}
Since $j_h(JI_{\rm M}v_2)=0=j_h(v)$,
the definition of $j_h(\bullet)$ shows  %
 $j_h(v_2-JI_{\rm M}v_2)=j_h(v_2-v)$. This, the definition of $\|\bullet \|_h$ in \eqref{eqn:jh_defn},  and Lemma~\ref{hctenrich}.f imply
\begin{equation}\label{eqn:err}
\|v_2-JI_{\rm M}v_2\|_h \lesssim \|v-v_2\|_h.
\end{equation}
The combination of \eqref{eqn:suminv}-\eqref{eqn:err} implies the assertion. \qed}
\begin{rem}[orthogonality of $J$]\label{ortho} Since $J$ is {{a}} right-inverse of $I_\M$,  i.e., $I_\M J={\rm id}$ in $\M(\T)$ \cite[(3.9)]{ccnnlower2022}, the integral mean property of the Hessian from Lemma~\ref{lemma:interpoltion_IM}.c reveals
$a_{\pw}(v_2, (1-J)v_\M) = a_{\rm pw} (v_2, (1-I_\M)Jv_\M)=0$ for any $v_2 \in P_2(\cT)$ and $v_\M \in \M(\T)$.
\end{rem}
\begin{lem}[an intermediate bound] \label{lem.new}For $1<p<\infty$, any $(v_2 ,v)\in P_2(\cT) \times V$ satisfies $|v+v_2|_{W^{1,p}(\cT)}$ $ \lesssim \|v+v_2\|_h.$
\end{lem}
\begin{proof} The triangle inequality $|v+v_2|_{W^{1,p}(\cT)} \le |v+J I_\M v_2|_{W^{1,p}(\O)} + |v_2-JI_\M v_2|_{W^{1,p}(\cT)} $ and the Sobolev embedding $H^2_0(\Omega) \hookrightarrow W^{1,p}_0(\Omega)$ in 2D lead to 
\begin{align*}
|v+J I_\M v_2|_{W^{1,p}(\O)} &\lesssim  \trinl  v+J I_\M v_2 \trinr \le \trinl  v+ v_2 \trinr_{\pw} + \trinl  v_2-J I_\M v_2 \trinr_{\pw} \lesssim \|v+v_2\|_h 
\end{align*}
with  $\trinl \bullet \trinr_{\pw}  \le \|\bullet\|_h$ and Lemma~\ref{hctenrich}.f  in the last step.  The inequality  $|v_2-JI_\M v_2|_{W^{1,p}(\cT)}  \le |\Omega|^{1/p} |v_2 - J I_\M v_2|_{W^{1,\infty}(\cT)}$ leads to some $K \in \cT$  with  $ |v_2 - J I_\M v_2|_{W^{1,\infty}(\cT)} = |v_2 - JI_\M v_2|_{W^{1,\infty}(K)}$.  The inverse estimate $ |v_2 - JI_\M v_2|_{W^{1,\infty}(K)} \lesssim h_K^{-1}  |v_2 - JI_\M v_2|_{H^1(K)}$ and  Lemma~\ref{hctenrich}.f  reveal $|v_2 - J I_\M v_2|_{W^{1,\infty}(\cT)}  \lesssim \|v+v_2\|_h$.  The combination of the above inequalities concludes the proof.
\end{proof}

\begin{lem}[{quasi-optimal smoother $R$}]\label{lem:R}
Any $R \in \{\mathrm{id}, I_{\rm M}, JI_{\rm M}\}$ and $\widehat{V} =V+V_h$ with 
$$V_h \;  (\text{resp.  } \|\bullet\|_{\widehat{V}})   :=
\begin{cases} 
 \M(\T) \text{ for the Morley scheme }(\text{resp. }    \trinl \bullet \trinr_{\pw}), \\
 P_2(\T)  \text{ for the dG scheme } (\text{resp.  }  \|\bullet\|_{\dg}),   \\
S^2_0(\T)  \text{ for the } C^0 \text{IP scheme } (\text{resp.  }  \|\bullet\|_{\ip} ),\\
 P_2(\T)  \text{ for the WOPSIP scheme } (\text{resp.  }  \|\bullet\|_{\rm P})  \\
\end{cases}$$
satisfy 
\[\| (1 - \rop)v_h \|_{\widehat{V}} \le \lamr\| v- v_h \|_{\widehat{V}}\; \mbox{ for all }(v_h,v) \in V_h \times V.\]
 The constant $\Lambda_{\rm R}$ exclusively depends on the shape regularity  of $\cT$.
\end{lem}
\noindent {\it Proof for $R = \mathrm{id}$}. This holds with $\Lambda_{\rm R}=0$. 
\qed \\
\smallskip \noindent
{\it Proof for $R =  I_{\rm M}$}. 
 {  Since $\| (1-\Pi_0)D^2_{\rm pw} v_{h}\|=0 $ for $v_h \in V_h \subseteq P_2(\cT)$, Lemma~\ref{lemma:interpoltion_IM}.a leads to $\trinl (1-I_{\rm M})v_h \trinr_{\pw}\lesssim j_h(v_h)$. This, the 
  definition of $\|\bullet\|_h$, and $j_h(I_{\M}v_h) =0=j_h(v)$ show $$\trinl (1-I_{\rm M})v_h \trinr_{\rm pw}\le  \|(1- I_{\rm M})v_h \|_h\lesssim j_h(v_h)=j_h(v-v_h)\le \|v - v_h \|_h\lesssim \|v - v_h \|_{\widehat{V}}$$ with Lemma~\ref{lem:equivalence} in the last step. Theorem 4.1 of \cite{ccnnlower2022} provides $\|(1- I_{\rm M})v_h \|_{\widehat{V}}  \lesssim  \| (1-I_{\rm M})v_h \|_h$ for the dG/$C^0$IP norm $\|\bullet\|_{\widehat{V}}$. The combination proves the assertion for Morley/dG/$C^0$IP. 

  \medskip \noindent
  For WOPSIP,  the definition  of $\|\bullet\|_{\rm P}$ in \eqref{eqn:norm_wopsip}, $\trinl (1-I_{\rm M})v_h \trinr_{\rm pw}\lesssim  \|v - v_h \|_{\rm P}$ from the displayed inequality above, and $c_{\rm P}(v,v)=c_{\rm P}(v,v_h)=0$  reveal
 $$\|(1  - I_{\rm M})v_h \|_{\rm P}\le \trinl (1- I_{\rm M})v_h \trinr_{\pw}+c_{\rm P}(v_h,v_h)^{1/2} \lesssim \|v-v_h\|_{\rm P}. \qed$$ }
\noindent {\it Proof for $R = JI_{\rm M}$}.  Triangle inequalities and $\|\bullet \|_{\widehat{V}}=\trinl \bullet \trinr_{\rm pw} $ in $V$ show
\begin{align*}
\|(1 - JI_{\rm M})v_h \|_{\widehat{V}}  \le \|v - v_h  \|_{\widehat{V}} + \trinl v - JI_\M v_h  \trinr_{\rm pw}
\le 2\|v - v_h  \|_{\widehat{V}} + \trinl (1- JI_{\rm M} )v_h \trinr_{\rm pw}. 
\end{align*}
Lemma~\ref{hctenrich}.{f}  and Lemma~\ref{lem:equivalence}
 conclude the proof for $R=JI_\M$. 
\qed

\medskip
\noindent The transfer from  $\M(\T)$ into $V_h$ \cite{ccnnlower2022} is modeled by some linear map  $I_h: \M(\T) \rightarrow V_h$  
that  is bounded in the sense that there exists some constant $\Lambda_h \ge 0$ such that
$ \|v_{\rm M} - I_h v_{\rm M}\|_h \le \Lambda_h \trinl v_{\rm M} - v \trinr_{\rm pw}$ holds for all $v_\M \in \M(\T)$ and all $v \in V$. A precise definition of $I_h= I_{\rm C}  I_\M$ concludes this {{section}}.
\begin{defn}[transfer operator {\cite[(8.4)]{ccnnlower2022}}]  \label{defn:ic}
For $v_{\rm M} \in {\rm M}(\cT)$, let $I_{\rm C} : \rm{M}(\cT) \rightarrow S^2_0(\cT)$ be defined by 
	\begin{equation}
	(I_{\rm C}v_{\rm M})(z) = \left\{
	\begin{array}{l}
	v_{\rm M}(z)\;\;\text{ at } z \in \mathcal{V}, \\
	\langle v_{\rm M}\rangle_{E} (z)\;\;\text{ at } z = \text{mid}(E)  \text{ for }  E\in \mathcal{E}(\O),\\
	0\;\;\text{ at } z = \text{mid}(E) \text{ for }  \, E \in \mathcal{E}(\partial \O)
	\end{array}
	\right.
	\end{equation}
followed by Lagrange interpolation in $P_2(K)$ for all $K \in \T$.
\end{defn}
\begin{rem}[approximation]\label{icimz} A triangle inequality with $I_\M v$, Lemma~\ref{lem:equivalence}, and $\|v_{\rm M} - I_{\rm C} v_{\rm M} \|_h \lesssim 
\trinl v- v_{\rm M} \trinr_{\pw}$ for any $v \in V$ and $v_{\rm M} \in {\rm M}(\cT)$  from \cite[(5.11)]{ccnnlower2022} show
	$\| v - I_{\rm C}I_{\rm M}v  \|_h   \lesssim \trinl v- I_\M v \trinr_{\pw}$. {{In particular, given any $v \in V$ and given any positive $\epsilon>0$,  there exists $\delta>0$ such that for any triangulation $\T \in \mathbb{T}(\delta)$ with discrete space $V_h$, we have $\| v-v_h\|_{\widehat{V}} <\epsilon$ for some $v_h \in V_h$. (The proof utilizes the density of smooth functions in $V$, the preceding estimates,  and Lemma~\ref{lemma:interpoltion_IM}.)}}
\end{rem}

\section{Application to Navier-Stokes equations} \label{sec:nse}
This section verifies the hypotheses~\ref{h1}-\ref{h5} and~\ref{h1hat} and establishes  {\bf (A)}-{\bf(C)} for the 2D Navier-Stokes equations in the stream function vorticity formulation.  Subsection  \ref{sec:prob} and \ref{sec:discrete_schemes} describe the problem and four quadratic  discretizations.  The a priori  error control for the Morley/dG/$C^0$IP (resp. WOPSIP) schemes  follows in Subsection \ref{sec:apriori_NS}-\ref{sec:lower} (resp. Subsection \ref{sec:wopsip}) .
\subsection{Stream function vorticity formulation of Navier-Stokes equations} \label{sec:prob}
The stream function vorticity formulation of the incompressible 2D Navier--Stokes equations in a bounded polygonal Lipschitz domain $\O \subset \R^2$ seeks $u \in H^2_0(\O)=:V=X=Y$  such that
\begin{equation} \label{NS_eqa}
 \Delta^2u + \frac{\partial}{\partial x}\bigg((-\Delta u)\frac{\partial u}{\partial y}\bigg)- \frac{\partial}{\partial y}\bigg((-\Delta u)\frac{\partial u}{\partial x}\bigg)= F 
\end{equation}
\noindent {{ for a given right-hand side $F \in V^*$.}} The biharmonic operator $\Delta^2$ is defined by $\Delta^2\phi:=\phi_{xxxx}+\phi_{yyyy}+2\phi_{xxyy}$.  The analysis of extreme viscosities lies beyond the scope of this article, and the viscosity in \eqref{NS_eqa} is set  one.
\noindent  For all $\phi,$ $\chi,\psi \in V$,  define the bilinear and trilinear  forms $a(\bullet, \bullet)$ and $\Gamma(\bullet,\bullet,\bullet)$ by 
\begin{align} 
a(\phi,\chi) := \int_{\O}^{}D^2\phi: D^2\chi \dx \;\text{ and }\;\Gamma(\phi, \chi,\psi) :=\int_{\O}^{}\Delta \phi\bigg(\frac{\partial \chi}{\partial y}\frac{\partial \psi}{\partial x}-\frac{\partial \chi}{\partial x}\frac{\partial \psi}{\partial y}\bigg)\dx. \label{abform}
\end{align}
The weak formulation that corresponds to \eqref{NS_eqa} seeks $u \in V$ such that
\begin{align} \label{NS_weak}
a(u,v) + \Gamma(u,u,v) = F(v)\fl v \in V.
\end{align}
\vspace{-0.3in}
\subsection{Four quadratic discretizations}\label{sec:discrete_schemes}
This subsection presents four  lowest-order discretizations,  namely,  the Morley/dG/$C^0$IP/WOPSIP schemes for \eqref{NS_weak}.
Define the discrete bilinear forms 
\[a_h:= a_{\rm pw} + {\mathsf b}_h+ 
{\mathsf c}_h:\left(V_h + \M(\T)\right)\times \left(V_h + \M(\T)\right)\to \mathbb{R}, \] with $a_{\pw}$ from \eqref{eqccnerwandlast1234a} and  $ {\mathsf b}_h,
{\mathsf c}_h $ in  Table \ref{tab:spaces} for the four discretizations. 
Let $\widehat{\Gamma}(\bullet,\bullet,\bullet):=\Gamma_{\rm pw}(\bullet,\bullet,\bullet)$ be the piecewise trilinear form  defined for all  $\phi, \chi, \psi \in H^2(\cT)$ by 
\begin{align}\label{defn:trilinear}
\displaystyle \Gamma_{\rm pw}(\phi, \chi,\psi)
:= \sum_{K\in\cT}\int_K \Delta \phi\left(\frac{\partial \chi}{\partial y}\frac{\partial \psi}{\partial x}-\frac{\partial \chi}{\partial x}\frac{\partial \psi}{\partial y}\right)\dx.
\end{align}
For all the four discretizations of Table \ref{tab:spaces},  recall $\widehat{b}(\bullet,\bullet) := \Gamma_{\rm pw}(u,\bullet,\bullet) + \Gamma_{\rm pw}(\bullet,u,\bullet): (V+P_2(\cT)) \times (V+P_2(\cT)) \rightarrow \mathbb{R}$ from \eqref{star}. 
Given $R,S \in \{{\rm id}, I_\M, JI_\M\}   $,  the discrete schemes for \eqref{NS_weak} seek a solution $u_h\in V_h$ to
\begin{align}\label{eqn:DWP_Vh}
	N_h(u_h; v_h):= a_h(u_h, v_h) +\Gamma_{\text{pw}}( Ru_h, Ru_h , Sv_h) - F(JI_\M v_h)=0&\text{ for all }v_h\in V_h.
\end{align}
\begin{table}[h!]
\centering
\small
\begin{tabular}{|c|c|cc|c|}
\hline
Scheme & Morley & \multicolumn{1}{c|} {dG} &  $C^0$IP & WOPSIP \\ \hline
\begin{minipage}{2cm}
\vspace{3pt}
\centering
${\widehat{X} = \widehat{Y}:= }\widehat{V} = V+V_h$
\end{minipage}
&  $V+ {\rm M}(\cT) $     & \multicolumn{1}{c|}{ $V+P_2(\cT) $ }         &  $V+S^2_0(\cT)$  & $V+P_2(\cT) $      \\ \hline


$\| \bullet \|_{\widehat{V}}$ & $\trinl \bullet \trinr_{\rm pw}$   & \multicolumn{1}{c|}{$\| \bullet \|_{\rm dG}$}       &      $\| \bullet \|_{\rm IP}$   &

$\| \bullet \|_{\rm P}$      \\ \hline

$P=Q$             & $J$     & \multicolumn{1}{c|}{$\smooth$}  & $\smooth$ & $\smooth$        \\ \hline

$I_{h}$ & ${\rm id}$ & \multicolumn{1}{c|}{${\rm id}$} & $I_{\rm C}$  from Definition \ref{defn:ic} & ${\rm id}$ \\ \hline  

{$I_{{\rm X}_h{}}=I_{{\rm V}_h}{=}I_h I_\M$}            &  $I_{\rm M}$ & \multicolumn{1}{c|}{$I_{\rm M}$} & $I_{\rm C}I_{\rm M}$ & $I_{\rm M}$     \\ \hline

$\mathcal{J}(\bullet ,\bullet)  $ & -- & \multicolumn{2}{c|}{{${{{}\displaystyle \sum_{E\in\cE}\int_E \langle D^2v_2\;\nu_E\rangle_E} \cdot \jump{\nabla w_2}_E\ds}$}}  & -- \\\hline

${\mathsf b}_h(\bullet,\bullet)$ & 0 &
\multicolumn{2}{c|}{
$-\theta \mathcal{J}(v_2,w_2) {-\mathcal{J}(w_2,v_2),\: }{{-1 \le \theta \le 1}}$
}
& 0\\ \hline
${\mathsf c}_h(\bullet,\bullet)$ & 0 & \multicolumn{1}{c|}{$c_{\rm dG}$ from \eqref{eqn:cdg}} & $c_{\rm IP}$ from \eqref{eqn:cip} & $c_{\rm P}$ from  \eqref{eqn:cp} \\ \hline
%
\end{tabular}
\caption{Spaces, operators, bilinear forms,  and norms in Section \ref{sec:nse}. }
\label{tab:spaces}
\end{table}
\vspace{-0.4in}
\subsection{ Main results}\label{sec:apriori_NS}
This subsection states  the results on the a priori control  for the   discrete schemes of Subsection \ref{sec:discrete_schemes}. {{Lemma~\ref{lem:equivalence} shows that $\|\bullet\|_{\widehat{V}} \approx \|\bullet\|_h$ for the Morley/dG/$C^0$IP schemes. The WOPSIP scheme is discussed in Subsection~\ref{sec:wopsip}. Unless stated otherwise,  $R \in \{{\rm id}, I_\M, JI_\M\}$ is arbitrary.}}
\begin{thm}[a priori energy norm error control] \label{thm:error_apriori_energy} {Given a regular root  $u \in V=H^2_0(\O)$ to \eqref{NS_weak} with $F \in H^{-2}(\Omega)$ and $0 <t<1$},  there exist $\epsilon, \delta > 0$ such that, for any $\displaystyle\cT\in\bT(\delta)$, the unique discrete solution $u_h \in V_h$ to \eqref{eqn:DWP_Vh} with 
$\| u-u_h\|_{{h}} \le \epsilon$ for the Morley/dG/$C^0$IP schemes 
satisfies
		\begin{align}  \label{eqn:apriori_a}
		\| u - u_h \|_{{h}} \lesssim{}&  \min_{v_h \in V_h} \|u - v_h \|_{{h}}+ 
\begin{cases}
0 \; \text{ for  }S=JI_\M,\\
{{ h_{\rm max}^{1-t} }}\; \text{ for }S={\rm id} \text{ or } I_\M.
\end{cases}
		\end{align}
{If $F \in H^{-r}(\O)$ for some $r<2$, then \eqref{eqn:apriori_a} holds with $t=0$.}
\end{thm}
{ 
\begin{rem}[quasi best-approximation] The best approximation result \eqref{eqn:best} holds for ${S=Q=J I_\M}$.
\end{rem}
\noindent 
{A comparison result 
follows as in \cite[Theorem 9.1]{ccnnlower2022} and the proof is therefore omitted.}
\begin{thm}[comparison for $R \in \{ {\rm id}, I_\M, JI_\M \}$ and {$S=Q=JI_\M$}]\label{thm:MainResult} 
{The regular root  $u \in V$} to \eqref{NS_weak} and {for $h_{\rm max}$ sufficiently small}, the respective local discrete solution $u_\M, u_{\dg}, u_{\rm IP} \in V_h$ to \eqref{eqn:DWP_Vh} 
	for the Morley/dG/$C^0$IP  schemes with $S=J I_\M$   satisfy
	\begin{align*}
		\|u-u_\M\|_h
		\approx
		\|u-u_\dg\|_h
		\approx
		\|u-u_\ip\|_h
		\approx \|(1-\Pi_0) D^2 u\|_{L^2(\Omega)}.  
	\end{align*}
\end{thm}
\noindent{A summary of the a priori error control in Theorem~\ref{thm:apost_ns} below is 
\begin{equation} \label{eqn:lower_gen}
\|u-u_h\|_{H^s(\T)} \lesssim \| u - u_{h} \|_{h} \left( h_{\rm max}^{a}+\| u - u_{h} \|_h \right) +C_b h_{\rm max}^b
\end{equation}
with  $a,b, C_b$ as described in Table~\ref{tab4}.} 
\begin{rem}[Table \ref{tab:apriori} vs \ref{tab4}] Note that the parameter $t>0$ appears in Table \ref{tab4} and not in Table~\ref{tab:apriori}. For $r=2$, \eqref{eqn:lower_gen} solely asserts
 $\|u-u_h\|_{H^s(\T)}\lesssim \|u-u_h\|_h^2 \lesssim 1$ even though $a$ and $b$ depend on $t$. 
\end{rem}
 \noindent {Recall the index of elliptic regularity  $\sigma_{\rm reg}$
 and $\sigma:=\min\{\sigma_{\rm reg},1\}>0$ from Section~\ref{sec:introduction}.} 
\begin{thm}[{a priori error control in weaker Sobolev norms}] \label{thm:apost_ns} Given a regular root  {$u \in V$} to \eqref{NS_weak} with {{$F \in H^{-2}(\Omega)$, $2-\sigma \le  s <2$,  and $0 <t<1$}},  there exist $\epsilon, \delta > 0$ such that, for any $\displaystyle\cT\in\bT(\delta)$,  the unique discrete solution $u_h \in V_h$ to \eqref{eqn:DWP_Vh} with 
	$\| u-u_h\|_{\widehat{V}} \le \epsilon$  satisfies $(a)$-$(e)$.
	
	\noindent
 {{(a) {For the Morley/dG/$C^0$IP  schemes  with } $R:=J I_\M,$
\begin{align*}	
  \| u -  u_h \|_{H^s(\cT)}  & \lesssim  \| u - u_{h} \|_{h} \left( h_{\rm max}^{2-s}+\| u - u_{h} \|_h \right) +\begin{cases}
		0 \; \textrm{ for }\; S = \smooth,\\
		{{h_{\rm max}^{3-t-s}}} \; \textrm{ for } \; S = {\rm id} \text{ or } I_{\rm M}.
	\end{cases}
	\end{align*}
	(b) {For the Morley/dG/$C^0$IP  schemes  with } $R:= I_\M$ and (c) for the Morley scheme with $R={\rm id}$, 
\begin{align*}	
  \| u -  u_h \|_{H^s(\cT)}  & \lesssim  \| u - u_{h} \|_{h} \left( {{h_{\rm max}^{\min\{2-s,1-t\}}}}+\| u - u_{h} \|_h \right) +\begin{cases}
		0 \; \textrm{ for }\; S = \smooth,\\
		{{h_{\rm max}^{3-t-s}}} \; \textrm{ for }\; S = {\rm id} \text{ or } I_{\rm M}.
	\end{cases} 
	\end{align*}
  (d) For $\sigma < 1 $, whence $1<s<2$, for the Morley/dG/$C^0$IP  schemes  with } $R \in \{ I_\M, JI_\M \}$ and for the Morley scheme with $R= {\rm id}$,
  \[
  \| u -  u_h \|_{H^s(\cT)}   \lesssim  \| u - u_{h} \|_{h} \left( h_{\rm max}^{2-s}+\| u - u_{h} \|_h \right) +\begin{cases}
		0 \; \textrm{ for } \; S = \smooth,\\
		{{h_{\rm max}^{4-2s}}} \; \textrm{ for } \; S = {\rm id} \text{ or } I_{\rm M}.
  \end{cases}
  \]
  (e)  If $F \in H^{-r}(\O)$ for some $ r<2$, then $(a)$-$(c)$ hold with $t=0$.}
\end{thm}


{
\begin{table}[] 
		\centering
		\begin{tabular}{|c|c|ccl|cll|c|c|c|}			
   \hline
			\multirow{2}{*}{$r$}             & \multirow{2}{*}{$s$}                         & \multicolumn{3}{c|}{$R$}                                                                            & \multicolumn{3}{c|}{$S$}                    & \multirow{2}{*}{$a$}   & \multirow{2}{*}{$b$} & \multirow{2}{*}{$C_b$} \\ \cline{3-8}
			&                                            & \multicolumn{1}{c|}{Morley}                       & \multicolumn{2}{c|}{dG/$C^0$IP}              & \multicolumn{3}{c|}{\begin{minipage}{2cm}
					\vspace{0.1cm}
					\centering
				Morley/dG/ $C^0$IP
					\vspace{0.1cm}
			\end{minipage}} &                      &                    &                     \\ \hline
			\multirow{2}{*}{$r < 2$} & \multirow{2}{*}{$2 - \sigma \le s < 2$}             & \multicolumn{1}{c|}{\multirow{2}{*}{$\mathrm{id}, I_{\rm M}, JI_{\rm M}$}} & \multicolumn{2}{c|}{\multirow{2}{*}{$I_{\rm M}, JI_{\rm M}$}}  & \multicolumn{3}{c|}{$JI_{\rm M}$}                  & \multirow{2}{*}{$2-s$} & $-$                  & $0$                   \\ \cline{6-8} \cline{10-11} 
			&                                            & \multicolumn{1}{c|}{}                             & \multicolumn{2}{c|}{}                         & \multicolumn{3}{c|}{id, $I_{\rm M}$}               &                      & $3-s$               & $1$                   \\ \hline \hline
			\multirow{6}{*}{$r = 2$}         & \multirow{2}{*}{$1 < s < 2$} & \multicolumn{1}{c|}{\multirow{2}{*}{$\mathrm{id}, I_{\rm M}, JI_{\rm M}$}} & \multicolumn{2}{c|}{\multirow{2}{*}{$I_{\rm M}, JI_{\rm M}$}} & \multicolumn{3}{c|}{$JI_{\rm M}$}                  & \multirow{2}{*}{$2-s$} &     $-$               &       $0$              
			\\ \cline{6-8} \cline{10-11} 
			&                                          & \multicolumn{1}{c|}{}                             & \multicolumn{2}{c|}{}                         & \multicolumn{3}{c|}{$\mathrm{id}, I_{\rm M}$}               &                     &      $4 - 2s$               &   $1$                  \\ \cline{2-11} 
			& \multirow{4}{*}{$s = \sigma= 1$}                     & \multicolumn{3}{c|}{\multirow{2}{*}{$JI_{\rm M}$}}                                                         & \multicolumn{3}{c|}{$JI_{\rm M}$}                  & \multirow{2}{*}{$1$}   &        $-$            &      $0$               \\ \cline{6-8} \cline{10-11} 
			&                                            & \multicolumn{3}{c|}{}                                                                             & \multicolumn{3}{c|}{$\mathrm{id}, I_{\rm M}$}               &                      &      $2-t$              &     $1$                \\ \cline{3-11} 
			&                                            & \multicolumn{1}{c|}{\multirow{2}{*}{$\mathrm{id}, I_{\rm M}$}}      & \multicolumn{2}{c|}{\multirow{2}{*}{$I_{\rm M}$}}      & \multicolumn{3}{c|}{$JI_{\rm M}$}                  & \multirow{2}{*}{$1-t$} &           $-$         &   $0$                  \\ \cline{6-8} \cline{10-11} 
			&                                            & \multicolumn{1}{c|}{}                             & \multicolumn{2}{c|}{}                         & \multicolumn{3}{c|}{$\mathrm{id}, I_{\rm M}$}               &                      &       $2-t$             &         $1$            \\ \hline
		\end{tabular}
  \caption{Summary of error control in \eqref{eqn:lower_gen} from Theorem~\ref{thm:apost_ns}.\label{tab4}}
	\end{table}
 }
 \begin{rem}[constant dependency]The constants hidden in the notation $\lesssim$ of Theorem~\ref{thm:error_apriori_energy} (resp.  \ref{thm:apost_ns}) exclusively depend on the exact solution $u$  (resp.  $u$ and $z$)  to \eqref{NS_weak} (resp.  \eqref{NS_weak} and \eqref{dual_linear}),  shape regularity of $\,\,{\cT},$ {$t$ (resp. $s$, $t$)}, and on respective stabilisation parameters ${\sigma_1,  \sigma_2, \sigma_{\rm IP}\approx 1}$.
\end{rem}
%
%
\begin{rem}[scaling for WOPSIP]
	The {{semi-scalar product}} ${\mathsf c}_h(\bullet,\bullet)$  in the WOPSIP scheme is an analog to the one in $j_h$ from \eqref{eqn:jh_defn} with different powers of the mesh-size.   It is a consequence of  the  different scaling of the norms that  \ref{h1} and ~\ref{h1hat} do not hold for the WOPSIP scheme.
\end{rem}

\subsection{ Preliminaries}
\quad \quad \;\;
{{
	
		\noindent This section investigates the piecewise trilinear form $\Gamma_{\rm pw}(\bullet,\bullet,\bullet)$ from \eqref{defn:trilinear} and its boundedness with a global parameter $0 <t<1 $ that may be small. Recall the energy norm $\trinl \bullet \trinr$, and the discrete norms $\trinl \bullet \trinr_{\pw}$, $\|\bullet\|_h$, and $\|\bullet\|_{\rm P}$ from Section \ref{sec:norms}. The constants hidden in the notation $\lesssim $ in Lemma~\ref{lemma:global_bounds_nse_new} below exclusively depend on the shape regularity of $\cT$ and on $t$.		
\begin{lem}[{boundedness of the trilinear form}]\label{lemma:global_bounds_nse_new} Any $\psi \in V$ and any
	$\widehat{\phi}, \widehat{\chi}, \widehat{\psi} \in V+P_2(\T)$, satisfy
	\begin{align}
 &(a) \; \Gamma_{\rm pw}(\widehat{\phi}, \widehat{\chi}, \widehat{\psi}) \lesssim  \trinl \widehat{\phi} \trinr_{\pw}
  \|  \widehat{\chi} \|_{h}  \| \widehat{\psi} \|_{h} \text{ and }
	(b) \; \Gamma_{\rm pw}( \widehat{\phi},\widehat{\chi}, {\psi}) \lesssim \trinl \widehat{\phi}\trinr_{\pw}
	\| \widehat{\chi} \|_{h}   \|\psi \|_{H^{1+t}(\O)}.
 \end{align}
\end{lem}
\begin{proof}
A general \Holder inequality reveals
\begin{align}\label{eq.bound.new}
 \Gamma_{\rm pw}( \widehat{\phi},\widehat{\chi}, \widehat{\psi}) &\le \sqrt{2}  \trinl \widehat{\phi} \trinr_{\pw}|\widehat{\chi}|_{W^{1,2/t}(\cT)}| \widehat{\psi} |_{W^{1,2/(1-t)}(\T)}
 \end{align}
 (owing to $t/2 +(1-t)/2=1/2$ and 
 $|\Delta_{\rm pw} \widehat{\phi}| \le \sqrt{2} |D^2_{\rm pw} \widehat{\phi}| $ a.e.). Lemma~\ref{lem.new} provides $| \widehat{\chi} |_{W^{1,2/t}(\cT)} \lesssim \|\widehat{\chi}\|_h$ and
 $| \widehat{\psi} |_{W^{1,2/(1-t)}(\cT)} \lesssim
 \|\widehat{\psi}\|_h$. The combination with \eqref{eq.bound.new} concludes the proof of $(a)$. 
 For $\psi \in V$ (replacing $\widehat{\psi}$), the Sobolev embedding $H^{t}(\Omega)\hookrightarrow L^{2/(1-t)}(\O)$ \cite[Corollary 9.15]{brezis} provides 
 $$|{\psi} |_{W^{1,2/(1-t)}(\cT)} = | {\psi} |_{W^{1,2/(1-t)}(\Omega)}  \lesssim \|\psi\|_{H^{1+t}(\Omega)}.$$
 The combination with \eqref{eq.bound.new} concludes the proof of $(b)$.
    \end{proof}

{\begin{lem}[{approximation properties}] \label{lemma:gamma_approx_new}  For all $t>0$, there exists a constant $C(t)>0$ such that any $\phi, \chi \in V\cap H^{2+t}(\O)$,
	$\widehat{\phi}, \widehat{\chi} \in V+P_2(\T)$, and $(v,v_2,v_{\rm M})\in V \times P_2(\T)\times \M{(\cT)}$ satisfy
	\begin{itemize}
\item[$(a)$]
$	\; \Gamma_{\rm pw}( \widehat{\phi},\widehat{\chi}, (1-JI_{\rm M})v_2) \le C(t) h_{\max}^{1-t}\trinl\widehat{\phi} \trinr_{\rm pw}
	\| \widehat{\chi} \|_{h}   \|v-v_2\|_{h}, $
\item[$(b)$] $	\; \Gamma_{\rm pw}( \widehat{\phi},{\chi}, (1-JI_{\rm M})v_2) \le C(t) h_{\max}\trinl \widehat{\phi} \trinr_{\pw}
	\| {\chi} \|_{H^{2+t}(\O)}  \|v-v_2\|_{h}, $
\item[(c)] $\Gamma_{\rm pw}(( 1- J)v_{\rm M},\widehat{\phi},\widehat{\chi}) \le{} C(t) h_{\max}^{1-t}  \trinl v-v_\M\trinr_{\pw}\|\widehat{\phi}\|_h  \|\widehat{\chi}\|_h .$
\item[(d)] $\Gamma_{\rm pw}(( 1- J)v_{\rm M},{\phi},{\chi}) \le{} C(t) h_{\max}\trinl v-v_\M\trinr_{\pw}\|{\phi}\|_{H^{2+t}(\O)} \|{\chi}\|_{H^{2+t}(\O)} .$
\end{itemize}
\end{lem}
\noindent


\noindent {\it Proof of  $(a)$}. 
 Lemma~\ref{lem.new} and~\ref{hctenrich}.h establish $|\widehat{\chi} |_{W^{1,2/t}(\cT)} \lesssim \|\widehat{\chi}\|_h$ and $| (1-JI_{\rm M})v_2 |_{W^{1,2/(1-t)}(\cT)}\lesssim h_{\max}^{1-t}\|v-v_2\|_{h}.$ The combination with \eqref{eq.bound.new} concludes the proof of $(a)$.
\qed





\noindent {\it Proof of $(b)$}. 
A generalised \Holder inequality and the embedding $H^{2+t}(\Omega)\hookrightarrow W^{1,\infty}(\Omega)$ \cite[Corollary 9.15]{brezis} provide
\begin{align*}
\Gamma_{\rm pw}(\widehat{\phi}, {\chi},(1 - J I_{\rm M})v_2) 
& \le \sqrt{2} \trinl \widehat{\phi}\trinr_{\pw}|{\chi}|_{W^{1,\infty}(\cT)}  |  (1 - J I_{\rm M})v_2 |_{H^1(\cT)}\\
&\lesssim  \trinl \widehat{\phi}\trinr_{\pw}\|{\chi}\|_{H^{2+t}(\cT)}  |  (1 - J I_{\rm M})v_2 |_{H^1(\cT)}.
\end{align*}
Lemma~\ref{hctenrich}.f 
controls the last factor and concludes the proof of  $(b)$. 
\qed			

\noindent {\it Proof of  $(c)$}.  Lemma~\ref{lemma:interpoltion_IM}.c implies
  { $\int_\O\Delta_{\rm pw}(v_{\rm M}-J v_{\rm M}) \: \Pi_0 D_{\rm pw}\widehat{\phi}\cdot \Pi_0 {\rm Curl}_{\rm pw}\widehat{\chi} \dx=0$} and so 
	\begin{align}\label{p1_new}
	\Gamma_{\rm pw}((1-J) v_{\rm M},\widehat{\phi},\widehat{\chi})& =\int_\O\Delta_{\rm pw}((1-J) v_{\rm M})((1-\Pi_0) D_{\rm pw}\widehat{\phi})\cdot {\rm Curl}_{\rm pw}\widehat{\chi}\dx \nonumber\\
	&\qquad +\int_\O\Delta_{\rm pw}((1-J) v_{\rm M}) \,\Pi_0 D_{\rm pw}\widehat{\phi}\cdot((1-\Pi_0) {\rm Curl}_{\rm pw}\widehat{\chi})\dx.
	\end{align}
A generalised \Holder inequality shows
	\begin{align}\label{eq.curlD}
&\int_\O\Delta_{\rm pw}((1-J) v_{\rm M})
((1-\Pi_0) D_{\rm pw}\widehat{\phi})\cdot {\rm Curl}_{\rm pw}\widehat{\chi}\dx\nonumber\\
&\qquad\le  \|h_{\cT}\Delta_{\pw}(1-J) v_{\rm M} \|_{L^{2/(1-t)}(\O)} \| h_{\cT}^{-1}(1-\Pi_0 )D_{\rm pw}\widehat{\phi}\|_{L^2(\O)} |\widehat{\chi} |_{W^{1,2/t}(\T)}.
	\end{align}
Abbreviate $a_T:=h_T^{2-t}\|\Delta(v_{\rm M}-Jv_{\rm M})\|_{L^\infty(T)}$ for a triangle $T \in \cT$ with area $|T| \le h_T^2$ to establish
\[ \|h_{\cT}\Delta_{\pw}(1-J) v_{\rm M} \|_{L^{2/(1-t)}(\O)} \le \big(\sum_{T \in \cT}a_T^{2/(1-t)}\big)^{(1-t)/2}\le \big(\sum_{T \in \cT}a_T^{2}\big)^{1/2}\]
with the monotone decreasing $\ell^p$ norm for $2\le 2/(1-t)$ in the last step. An inverse estimate (with respect to the HCT refinement $\widehat{\cT}$ of $\cT$) as in the proof of Lemma~\ref{hctenrich}.h provides
$\|\Delta((1-J)v_{\rm M})\|_{L^\infty(T)} \le \sqrt{2} \|v_{\rm M}-Jv_{\rm M} \|_{W^{2,\infty}(\O)}\lesssim h_{T}^{-1}\|v_{\rm M}-Jv_{\rm M} \|_{H^2(T)}.$ Hence $a_T \lesssim h_{T}^{1-t}\|v_{\rm M}-Jv_{\rm M} \|_{H^2(T)}$ and
\[ \|h_{\cT}\Delta_{\pw}(1-J) v_{\rm M} \|_{L^{2/(1-t)}(\O)} \lesssim \trinl h_{\cT}^{1-t}(v_{\rm M}-Jv_{\rm M} )\trinr_{\pw}\le h_{\max}^{1-t}\trinl v_{\rm M}-Jv_{\rm M} \trinr_{\pw}.\]
A piecewise \Poincare inequality with Payne-Weinberger constant $h_{T}/\pi$ \cite{PayneWeinbeger} reads $$\pi\| h_{\cT}^{-1}(1-\Pi_0) D_{\rm pw}\widehat{\phi}\|_{L^2(\O)} \le \trinl \widehat{\phi} \trinr_{\pw}.$$ Recall $|\widehat{\chi} |_{W^{1,2/t}(\T)}\lesssim \|\widehat{\chi}\|_h$ from the proof of $(a)$. The combination of the previous estimates of the three terms in \eqref{eq.curlD} proves the asserted estimate for the first integral in the right-hand side of \eqref{p1_new}. The analysis for the second term is rather analogue (interchange the role of $\widehat{\phi}$ and $\widehat{\chi}$). Notice that $(c)$ follows even in the form  $\Gamma_{\rm pw}(( 1- J)v_{\rm M},\widehat{\phi},\widehat{\chi}) \le{} C(t) h_{\max}^{1-t}  \trinl v-v_\M\trinr_{\pw}(\trinl \widehat{\phi}\trinr_{\pw} \|\widehat{\chi}\|_h+\|\widehat{\phi}\|_h  \trinl \widehat{\chi}\trinr_{\pw}).$
\qed

   \medskip
   \noindent {\it Proof of $(d)$}. Substitute $\phi \equiv \widehat{\phi}$, $\chi\equiv \widehat{\chi}$ in \eqref{p1_new} (with $\phi,\chi \in V\cap H^{2+t}(\O)$) and employ a different generalised \Holder inequality for the first term to infer
   	\begin{align}\label{eq.curlD.1}
&\int_\O\Delta_{\rm pw}((1-J) v_{\rm M})((1-\Pi_0) D{\phi})\cdot {\rm Curl}{\chi}\dx&\nonumber\\
&\qquad\le \|\Delta_{\pw}(1-J) v_{\rm M} \|_{L^{2}(\O)} \| (1-\Pi_0 )D{\phi}\|_{L^2(\O)} |{\chi} |_{W^{1,\infty}(\O)}.
	\end{align}
 The remaining arguments of the proof of $(c)$ simplify to $\|\Delta_{\pw}(1-J) v_{\rm M} \|_{L^{2}(\O)} \le \sqrt{2} \trinl (1-J) v_{\rm M} \trinr_{\pw}$, $\pi \| (1-\Pi_0 )D{\phi}\|_{L^2(\O)}  \le h_{\max}\trinl \phi \trinr$, and $|{\chi} |_{W^{1,\infty}(\O)} \lesssim \|\chi \|_{H^{2+t}(\O)}$ (by embedding $H^{2+t}(\Omega)\hookrightarrow W^{1,\infty}(\Omega)$ for $t>0$). The resulting estimate 
 \begin{align}\label{eq.curlD.1}
\int_\O\Delta_{\rm pw}((1-J) v_{\rm M})((1-\Pi_0) D{\phi})\cdot {\rm Curl}{\chi}\dx&\lesssim h_{\max}\trinl (1-J) v_{\rm M} \trinr_{\pw} \trinl \phi \trinr  \|\chi \|_{H^{2+t}(\O)}
	\end{align}
and Lemma~\ref{hctenrich}.e lead to the assertion for one term in the right-hand side of \eqref{p1_new}. The analysis of the other term is analog. Notice that $(d)$ follows even in the form $\Gamma_{\rm pw}(( 1- J)v_{\rm M},{\phi},{\chi}) \le{} C(t) h_{\max}\trinl v-v_\M\trinr_{\pw}(\trinl {\phi}\trinr \|{\chi}\|_{H^{2+t}(\O)}+\|{\phi}\|_{H^{2+t}(\O)} \trinl {\chi}\trinr).$
		\qed}
\subsection{ Proof of Theorem \ref{thm:error_apriori_energy}}\label{proof_of_8.1}
The conditions in Theorem~\ref{thm:apriori} are verified to establish the energy norm estimates.   The  hypotheses \eqref{quasioptimalsmootherP}-\eqref{quasioptimalsmootherS} follow from Lemma~\ref{lem:R}. Hypothesis~\ref{h1} is verified  for Morley/dG/$C^0$IP in the norm $\|\bullet\|_h$  in~\cite[Lemma 6.6]{ccnnlower2022} and this norm is equivalent to $\trinl \bullet \trinl_{\rm pw} $,$\|\bullet\|_{\dg}$, and $\|\bullet\|_{\ip}$ by Lemma~\ref{lem:equivalence}. 

\medskip
{{ \noindent Recall $a(\bullet,\bullet)$ and $\Gamma(\bullet,\bullet,\bullet)$ from \eqref{abform}, $\widehat{\Gamma}(\bullet,\bullet,\bullet)\equiv {\Gamma}_{\pw}(\bullet,\bullet,\bullet)$ from \eqref{defn:trilinear}, and $\widehat{b}(\bullet,\bullet)$ from \eqref{star} for the regular root $u \in H^2_0(\O)$. For $\theta_h \in V_h$ with $\|\theta_{h}\|_{h} = 1$,  
{Lemma~\ref{lemma:global_bounds_nse_new}}.b, and $\trinl \bullet \trinr_{\rm pw} \le \| \bullet \|_h
$  provide
 $\widehat{b}(R\theta_h,\bullet) \in H^{-1-t}(\Omega)$ for $R \in \{ {\rm id},  I_\M, JI_\M\}$.
There exists a unique $\xi\equiv\xi(\theta_h) \in V\cap H^{3-t}(\Omega)$ such that $a(\xi,\phi) = \widehat{b}(R\theta_h,\phi)$ for all $\phi \in V$ and $\|\xi\|_{H^{3-t}(\O)} \lesssim \|\widehat{b}(R\theta_h,\bullet)\|_{H^{-1-t}(\O)}\lesssim 1$. The last inequality follows from Lemma~\ref{lemma:global_bounds_nse_new}.b and the boundedness of $R \in \{{\rm id},  I_\M, JI_\M\}$ from Lemma~\ref{lem:R}. 
Since $I_h=\rm{id}$ (resp. $I_h=I_{\rm C}$) for Morley/dG (resp. $C^0$IP), Lemma~\ref{lem:equivalence} (resp. Remark~\ref{icimz}) and Lemma~\ref{lemma:interpoltion_IM}.d establish~\ref{h3} with $\delta_{2} =  \sup \{  \|\xi - I_h I_{\M} \xi \|_{h}: \theta_h\in V_h, \|\theta_h\|_{h}=1 \}\lesssim h_{\rm max}^{1-t}$.

\medskip
\noindent  {Since $\delta_3=0$ for $Q=S=JI_{\M}$ it remains $S={\rm id}$ and $S=I_{\rm M}$ in the sequel to establish {\bf (H3)}. Given  $\theta_h$ and $y_h$ in $V_h=X_h=Y_h$ of norm one, define $v_2:=Sy_h \in P_2(\cT)$ and observe $Qy_h=JI_{\rm M}y_h=JI_{\rm M}v_2$ (by $S={\rm id},I_{\rm M}$). Hence with the definition of $\widehat{b}(\bullet,\bullet)$ from \eqref{star}, Lemma~\ref{lemma:gamma_approx_new}.a shows
\begin{equation}\label{eqn.h3}
|\widehat{\nonlin}(R\theta_h,(S - Q)y_h)|=|\widehat{\nonlin}(R\theta_h,v_2-JI_{\rm M}v_2)|\le 2C(t)h_{\max}^{1-t}\trinl u\trinr \|R\theta_h\|_h\|v_2\|_h.
\end{equation}
The boundedness of $R$ and $I_{\rm M}$ and the equivalence of norms show $\|R\theta_h\|_h\|v_2\|_h \lesssim 1$ and so $\delta_{3} \lesssim h_{\rm max}^{1-t}$.} 
}

\medskip \noindent Consequently,  for the three schemes under question and for a sufficiently small mesh-size $h_{\max}$, \eqref{eqn:disinfsup} holds  with $\beta_h\ge \beta_0 \gtrsim 1$.}} 

\smallskip
\noindent 
For $u \in H^2_0(\Omega)$ and $\epsilon >0$, Remark~\ref{icimz} 
establishes \ref{h5} with $\delta_4 < \epsilon$ for all the three schemes. The existence and uniqueness of a discrete solution $u_h$ then follows from Theorem \ref{thm:existence}.

\noindent{{For the Morley/dG/$C^0$IP schemes with $F \in H^{-2}(\O)$, Lemma~\ref{lemma:gamma_approx_new}.a with $v=0$ for $S={\rm id}$  resp. 
$S=I_{\rm M}$, $\|\bullet \|_h\approx\|\bullet \|_{V_h}$ on $V_h$, and the boundedness of $I_\M$ show
\begin{align} \label{gamma:bound}
\|\widehat{\Gamma}(u,u,(S - Q)\bullet)\|_{V_h^\ast} \lesssim 
\begin{cases}
0 \text{ for } S=Q=JI_\M,\\
 h_{\rm max}^{1-t} \text{ for } S={\rm id} \text{ or } I_\M.
\end{cases}
\end{align}
The energy norm error control then follows from Theorem~\ref{thm:apriori}. 


\medskip

\noindent
		For $F \in H^{-r}(\O)$ with $r<2$, the energy norm error estimate \eqref{eqn:apriori_a} with $t=0$ 
 can be established by replacing Lemma~\ref{lemma:gamma_approx_new}.a in the above analysis for $r=2$  by Lemma~\ref{lemma:gamma_approx_new}.b.
 \qed
 }}

\subsection{ Proof of Theorem \ref{thm:apost_ns}}\label{sec:lower}
This subsection  establishes the a priori control in weaker Sobolev norms for the  Morley/dG/$C^0$IP schemes of Subsection \ref{sec:discrete_schemes}.  Given $2-\sigma \le s \le 2$,   and $G \in H^{-s}(\Omega)$ with $\|G\|_{H^{-s}(\Omega)}=1$,  the solution $z$ to the dual problem \eqref{dual_linear} belongs to 
$ V \cap H^{4-s}(\Omega)$ by elliptic regularity. This and Lemma~\ref{lemma:interpoltion_IM}.d provide
\begin{align} \label{eqn:regulardual}
\trinl  z - I_\M z \trinr_{\pw}  &\lesssim h_{\rm max}^{2-s} \|z\|_{H^{4-s}(\Omega)} \lesssim  h_{\rm max}^{2-s} \|G\|_{H^{-s}(\Omega)} 
 =  h_{\rm max}^{2-s}.  
\end{align}
The assumptions in Theorem~\ref{cor:lower} with $X_{s}:= H^{s}(\T)$ {and $z_h := I_h I_{\rm M}z$} are verified to establish Theorem~\ref{thm:apost_ns}.a-e.
The control of the linear terms  in  Theorem~\ref{cor:lower} is identical for the parts $(a)$-$(e)$ and this is discussed first. {The proof starts with a  triangle inequality 
 \begin{equation}\label{eqn.triangle}
 \|u-u_h\|_{H^s(\cT)} \le \|u-Pu_h \|_{H^s(\cT)} + \|Pu_h-u_h\|_{H^s(\cT)}  
 \end{equation}
in the norm ${H^s(\T)}=\prod_{T\in\T} H^s(T)$. 
The Sobolev-Slobodeckii semi-norm over $\Omega$ involves double integrals over $\Omega\times\Omega$ and so is larger   than or equal to the sum of the 
contributions over $T\times T$ for all the triangles $T\in\T$, i.e.,    $ \sum_{T\in\T} |\bullet  |_{H^s(T)}^2\le  |\bullet  |_{H^s(\Omega)}^2$ for any $1<s<2$.
 The definition of $\|\bullet\|_{H^{s}(\T)}$ for $1<s<2$, Lemma~\ref{hctenrich}.f with $t=1$ and $P=JI_{\rm M}$ establish
 \begin{align}\label{eqn.puh}
     \|Pu_h-u_h\|_{H^s(\cT)} & \le \|Pu_h-u_h\|_{H^1(\cT)}+|\nabla_{\pw}(Pu_h-u_h)|_{H^{s-1}(\cT)} \nonumber \\ 
     & \lesssim h_{\max}\|u-u_h\|_h+|\nabla_{\pw}(Pu_h-u_h)|_{H^{s-1}(\cT)}.
 \end{align}
The formal equivalence of the Sobolev–Slobodeckii norm and the norm by interpolation of Sobolev spaces provides for $g:=\nabla_{\pw}(Pu_h-u_h)$, $\theta:=s-1$ and $K \in \T$ that
\begin{equation}\label{eqn.g}
    |g|_{H^\theta(K)} \le C(K, \theta)\|g\|_{L^2(K)}^{1-\theta}|g|_{H^1(K)}^\theta.
\end{equation}
The point is that a scaling argument reveals 
$C(K,\theta)=C(\theta)\approx 1$ is independent of $K \in \cT$ \cite{ccnn2021}. 
The Young's inequality \big($ab\le {a^p}/{p}+{b^q}/{q}$ for $a,b \ge 0$, ${1}/{p}+{1}/{q}=1$\big) leads (for $a=h_K^{2\theta(\theta-1)}\|g\|_{L^2(K)}^{2(1-\theta)} $, $b= h_K^{2\theta(1-\theta)}|g|_{H^1(K)}^{2\theta}$, $p={1}/({1-\theta})$, and $q={1}/{\theta}$) to
\begin{align}
 \sum_{K \in \cT} \|g\|_{L^2(K)}^{2(1-\theta)}|g|_{H^1(K)}^{2\theta}
 &= \sum_{K \in \cT} h_K^{2\theta(\theta-1)}\|g\|_{L^2(K)}^{2(1-\theta)} h_K^{2\theta(1-\theta)}|g|_{H^1(K)}^{2\theta}\\
 & \le \|h_{\cT}^{-\theta}g\|^2_{L^2(\O)}+|h_{\cT}^{1-\theta}g|^2_{H^1(\cT)}.\label{eqn.youngs}
\end{align}~\noeqref{eqn.youngs}
Since $P=JI_{\M}$ and $g=\nabla_{\pw}(Pu_h-u_h)$, the estimates \eqref{eqn:ortho}-\eqref{eqn:err} with $t=\theta$ show $\|h_{\cT}^{-\theta}g\|^2_{L^2(\O)} \lesssim h_{\rm max}^{1-\theta}\|u- u_h\|_h.$ This and  Lemma~\ref{hctenrich}.f for $t=2$ provide 
\begin{equation}\label{eqn:triangle1}
\|h_{\cT}^{-\theta}g\|^2_{L^2(\O)} +|h_{\cT}^{1-\theta}g|^2_{H^1(\cT)} \lesssim h_{\rm max}^{1-\theta}\|u-u_h\|_h.
\end{equation}
The combination of \eqref{eqn.g}-\eqref{eqn:triangle1} reveals
$|\nabla_{\pw}(Pu_h-u_h)|_{H^{s-1}(\cT)} \lesssim h_{\rm max}^{2-s}\|u-u_h\|_h$ and, with \eqref{eqn.puh},
\begin{equation} \label{eqn:uh-puh}
\|Pu_h  - u_h \|_{H^s(\cT)} \lesssim h_{\rm max}^{2-s} \|u - u_h\|_{h}.
\end{equation}
}

   \noindent This leads to the assertion for one term on the right-hand side of \eqref{eqn.triangle}. To estimate the second term, $\|u-Pu_h\|_{H^s(\cT)}= G(u-Pu_h)$, we verify the assumptions in Theorem~\ref{thm:lower}.  The hypothesis~\ref{h1hat} for the Morley/dG/$C^0$IP schemes is derived in~\cite[Lemma 6.6]{ccnnlower2022} for an equivalent norm (by Lemma~\ref{lem:equivalence}) and Lemma~\ref{lem:R} for $R=JI_\M$. The conditions~\eqref{quasioptimalsmootherP}-\eqref{quasioptimalsmootherS} also follow from Lemma~\ref{lem:R} as stated in the proof of Theorem~\ref{thm:error_apriori_energy}. {Hence, Theorem~\ref{thm:lower} applies and provides	\begin{align}\label{eqn.lowerorder}
	\|u-Pu_h\|_{H^s(\cT)}=G(u-Pu_h) \lesssim   &\|u - u_h \|_{h}  ( \| z - z_h \|_{h}  + \|u - u_h \|_{h})+ {\Gamma_{\pw}}(u,u,(S-Q)z_h) \nonumber \\
& \quad +
{\Gamma}_{\pw}(Ru_h,Ru_h,Q z_h) -{\Gamma}(Pu_h,Pu_h,Qz_h).
	\end{align}}
Since  $\|\bullet\|_{\dg} \approx \trinl \bullet \trinr_{\pw}$ in $V+\M(\cT)$ (by Lemma~\ref{lem:equivalence}), \eqref{eqn:regulardual} establishes
	\begin{align} \label{eqn:approx_uh}
	\|z - z_h \|_{h} \lesssim h_{\rm max}^{2-s} 
	\end{align}
for the Morley/dG schemes with $I_h={\rm id}$. Remark~\ref{icimz} and  \eqref{eqn:regulardual} establish \eqref{eqn:approx_uh} for the $C^0$IP scheme. {The combination of \eqref{eqn.lowerorder}-\eqref{eqn:approx_uh} reads
\begin{align}\label{eqn.lowerorder1}
	\|u-Pu_h\|_{H^s(\cT)}
 \lesssim   &\|u - u_h \|_{h}  ( h_{\rm max}^{2-s} 
 + \|u - u_h \|_{h})+ {\Gamma_{\pw}}(u,u,(S-Q)z_h) \nonumber \\
& \quad +
{\Gamma}_{\pw}(Ru_h,Ru_h,Q z_h) -{\Gamma}(Pu_h,Pu_h,Qz_h).
	\end{align}
The combination of \eqref{eqn.triangle}, \eqref{eqn:uh-puh}, and \eqref{eqn.lowerorder1} verifies, for each of the Morley/dG/$C^0$IP schemes, that}
\begin{align} \label{new:inter}
	\| u - u_h \|_{H^s(\cT)} & \lesssim  \|u - u_h \|_{h}  ( h_{\rm max}^{2-s}  + \|u - u_{h} \|_{h} )  + {\Gamma}_{\pw}(u,u,(S-Q)z_h) \nonumber \\
	& \quad +
{\Gamma}_{\pw}(Ru_h,Ru_h,Q z_h) -{\Gamma}(Pu_h,Pu_h,Q z_h).
	\end{align}

\noindent{\it Proof of Theorem \ref{thm:apost_ns}.a.}
The difference $ {\Gamma}_{\pw}(Ru_h,Ru_h,Qz_h) - \Gamma(Pu_h,Pu_h, Qz_h)$ vanishes for $P=R=JI_\M$ in each of the three schemes.  The terms $ {\Gamma}_{\pw}(u,u,(S-Q)z_h)$  in  \eqref{new:inter} are estimated below for $S \in \{ {\rm id}, I_{\rm M}, \smooth \}$ {and $F \in H^{-2}(\O)$}.  Note that  $Qz_h:=J z_h=J I_\M z_h$ holds for the Morley scheme. 
{{For $S ={\rm id}$ and each of the three discretizations, Lemma~\ref{lemma:gamma_approx_new}.a with $v_2=z_h$ provides
	\begin{align*}
	\Gamma_{\pw}(u,u,(1 - \smooth)z_h ) \lesssim {h}^{1-t}_{\max} \trinl u\trinr^2 \|z- z_h \|_{h} \lesssim h_{\rm max}^{3-t-s}
	\end{align*}
with \eqref{eqn:approx_uh} in the last step.  
	For $S  = I_{\rm M}$,  Lemma  \ref{lemma:gamma_approx_new}.a with $v_2  = I_{\rm M}z_h$ and $\| \bullet \|_{\widehat{V}} \approx \|\bullet \|_{h}$ reveal
	\begin{align*}
	\Gamma_{\pw}(u,u,(1 -J  )I_{\rm M}z_h) \lesssim h_{\rm max}^{1-t} \trinl  u \trinr^2  \| z- I_{\rm M}z_h \|_{h}. 
	\end{align*}
A triangle inequality and Lemma~\ref{lem:R} for $R=I_\M$ provide
	 $\|z- I_{\rm M}z_h  \|_{h}  \le (1+\Lambda_{\rm R}) \| z - z_h \|_{h} \lesssim h_{\rm max}^{2-s}$
 with ~\eqref{eqn:approx_uh} in the last step. Altogether, we obtain $\Gamma_{\pw}(u,u,(1 -J  )I_{\rm M}z_h) \lesssim h_{\rm max}^{3-t-s} $. 
The aforementioned estimates and \eqref{new:inter}    conclude the proof. \qed}}
%

\noindent{\it Proof of Theorem \ref{thm:apost_ns}.b.} 
 All the terms except the last two  in \eqref{new:inter}  are already estimated in the proof of $(a)$.    For $P=Q=JI_\M$ and $R=I_\M$, elementary algebra reveals
\begin{align}
 {\Gamma}_{\pw}&(Ru_h,Ru_h,Qz_h) - \Gamma(Pu_h,Pu_h, Qz_h)\\& = {\Gamma}_{\pw}((R-P)u_h,Ru_h,Qz_h) +{\Gamma}_{\pw}(Pu_h,(R-P)u_h, Qz_h)\\
	&  = \Gamma_{\rm pw}( (1 - J)I_{\rm M} u_h, I_{\rm M} u_h,JI_{\rm M} z_h ) +  
\Gamma_{\rm pw}( JI_{\rm M} u_h, (1 - J)I_{\rm M} u_h ,JI_{\rm M} z_h )\label{eqn.split2}.
\end{align}
The bound $\trinl \bullet \trinr_{\pw} \le \|\bullet\|_{h}$, a triangle inequality, and Lemma~\ref{lem:R} for $R=I_\M$ result in
\begin{align} \label{eqn:key1}
\trinl u - I_\M u_h\trinr_{\pw} &\le \| u - u_h\|_{h}+\| u_h - I_\M u_h\|_{h} \le (1+ \Lambda_{\rm R}) \| u - u_h\|_{h}
\end{align}
as in Remark~\ref{rem.consequences}. This and Lemma~\ref{hctenrich}.e prove
\begin{align} \label{eqn:key2}
\trinl (1-J) I_\M u_h\trinr_{\pw}  &\lesssim \trinl u - I_\M u_h\trinr_{\pw} \lesssim \| u - u_h\|_{h}.
\end{align}
A triangle inequality and \eqref{eqn:key1}-\eqref{eqn:key2} imply
\begin{align} 
\trinl u-J I_\M u_h\trinr_{\pw}  & \le \trinl u - I_\M u_h\trinr_{\pw} + \trinl (1-J) I_\M u_h\trinr_{\pw} \label{eqn:key3}
 \lesssim   \| u - u_h\|_{h}.
\end{align}
As in Remark~\ref{rem.consequences}, analogous arguments plus \eqref{eqn:approx_uh} provide
\begin{align} \label{eqn:key4}
\trinl z - I_\M z_h\trinr_{\pw} & \le (1+ \Lambda_{\rm R}) \| z - z_h\|_{h} \text{ and } \trinl z - J I_\M z_h \trinr_{\pw}   \lesssim \| z - z_h\|_{h} \lesssim h_{\rm max}^{2-s}.
\end{align}

\medskip

\noindent {Lemma~\ref{lemma:gamma_approx_new}.c and the equivalence $\|\bullet\|_h \approx \trinl \bullet \trinr_{\pw}$ in $V+\M(\cT)$ (by Lemma~\ref{lem:equivalence}) control the first term on the right-hand side of \eqref{eqn.split2}, namely
\begin{equation}
\Gamma_{\rm pw}( (1 - J)I_{\rm M} u_h, I_{\rm M} u_h,JI_{\rm M} z_h ) \lesssim h_{\rm max}^{1-t}\trinl u-I_{\rm M}u_h\trinr_{\rm pw}\trinl I_{\rm M}u_h\trinr_{\rm pw}\trinl JI_{\rm M}z_h \trinr.
\end{equation}
The first factor is bounded in \eqref{eqn:key1}. Since the dual solution $z \in V \cap H^{4-s}(\O)$ is bounded in $V=H^2_0(\O)$ (even in $H^{4-s}(\O)$), \eqref{eqn:key4} reveals $\trinl JI_{\rm M}z_h \trinr\lesssim 1.$ Since $\trinl I_{\rm M}u_h \trinr_{\pw}\lesssim 1$ as well, we infer 
\begin{align}\label{eqn.split2.1}
    \Gamma_{\rm pw}( (1 - J)I_{\rm M} u_h, I_{\rm M} u_h,JI_{\rm M} z_h ) \lesssim h_{\rm max}^{1-t}\| u-u_h\|_h.
    \end{align}
The anti-symmetry  of $\Gamma_{\pw}(\bullet,\bullet,\bullet)$ with respect to the second and third variables allows the application of Lemma~\ref{lemma:gamma_approx_new}.a to the second term on the right-hand side of \eqref{eqn.split2}, namely
\begin{equation}\label{eqn.split2.2}
\Gamma_{\rm pw}( JI_{\rm M} u_h, (1 - J)I_{\rm M} u_h ,JI_{\rm M} z_h )\lesssim h_{\rm max}^{1-t}\trinl JI_{\rm M} u_h\trinr\trinl u-I_{\rm M}u_h\trinr_{\rm pw}\trinl JI_{\rm M} z_h\trinr {\lesssim} h_{\rm max}^{1-t} \| u - u_{h} \|_{ h}.
\end{equation}
The last step employed \eqref{eqn:key1} and the boundedness $\trinl JI_{\rm M} u_h\trinr+\trinl JI_{\rm M} z_h\trinr \lesssim 1$ as well. The combination of the previously displayed estimate with \eqref{eqn.split2.1} and  \eqref{eqn.split2} leads to
\begin{align}\label{eqn:combine}
{\Gamma}_{\pw}(I_\M u_h, I_\M u_h, JI_\M z_h) - \Gamma(JI_\M u_h, J I_\M u_h, J I_\M z_h) {\lesssim} h_{\rm max}^{1-t} \| u - u_{h} \|_{ h}.
\end{align}
The estimates of ${\Gamma}_{\pw}(u,u,(S-Q)z_h)$ from the above proof of Theorem~\ref{thm:apost_ns}.a, \eqref{eqn:combine}, and \eqref{new:inter}  conclude the proof.}

\medskip
\noindent {\it Proof of Theorem \ref{thm:apost_ns}.c}}. 
Since $u_h= u_{\rm M} = I_{\rm M}u_{\rm M}$,  and $P=Q=J$,  for the Morley FEM, 
 the difference
${\Gamma}_{\pw}( u_\M,   u_\M,J I_\M z_h) -{\Gamma}(J  u_\M, J  u_\M, J I_\M z_h)$ is controlled by~\eqref{eqn:combine}.  This, \eqref{new:inter}, and the estimates from the above proof of Theorem~\ref{thm:apost_ns}.a conclude the proof.  
\qed

\medskip
{
\noindent {\it Proof of Theorem \ref{thm:apost_ns}.d}}.
The choice $t:=s-1>0$ 
  in the estimates in   $(a)$-$(c)$
  concludes the proof. \qed }

  \medskip
  {
\noindent {\it Proof of Theorem \ref{thm:apost_ns}.e}.
\noindent For $F \in H^{-r}(\O)$ with $r<2$,  the lower-order error estimates can be established with $t=0$ by the substitution of the respective assertions of Lemma~\ref{lemma:gamma_approx_new}.a,c by Lemma~\ref{lemma:gamma_approx_new}.b,d.}
\qed
  
\begin{rem}[weaker Sobolev norm estimates with $R={\rm id}$]
For the dG/$C^0$IP schemes, \eqref{eqn.split2} involves in particular $\Gamma_{\rm pw} ((1-J I_\M)u_h,u_h,JI_{\rm M}z_h) $
and improved estimates are unknown. 
\end{rem}
\subsection{ WOPSIP scheme}\label{sec:wopsip}

Recall $a_h(\bullet,\bullet) = a_{\rm pw}(\bullet,\bullet) + {\mathsf c}_h(\bullet,\bullet)$, $P  = Q = JI_{\rm M}$ and ${\mathsf c}_h(\bullet,\bullet)$ from Table~\ref{tab:spaces}, $a_{\rm pw}(\bullet,\bullet)$ from~\eqref{eqccnerwandlast1234a}, and let $u_h  \equiv u_{\rm P}$ in this subsection. The norm $\| \bullet \|_{\rm P}$ from \eqref{eqn:norm_wopsip}  for the WOPSIP scheme is {\it not} equivalent to $\|\bullet\|_h$ from \eqref{eqn:jh_defn} and hence \ref{h1} and \ref{h1hat} do {\it not} follow. This does not prevent rather analog a priori error estimates.

\begin{thm}[a priori WOPSIP]\label{error control}
{{Given a regular root  $u \in V$ to \eqref{NS_weak} with $F \in H^{-2}(\Omega)$, $2-\sigma \le s<2$, and $0 <t<1$}},  there exist $\epsilon, \delta > 0$ such that, for any $\displaystyle\cT\in\bT(\delta)$, the unique discrete solution $u_h \in V_h$ to \eqref{eqn:DWP_Vh} with 
$ \| u-u_h\|_{\rm P} \le \epsilon$  for the WOPSIP scheme satisfies $(a)$-$(e)$. 
\begin{align}
&(a)  \; \|u - u_{h} \|_{\rm P}  \lesssim \trinl u  - I_{\rm M} u \trinr_{\rm pw} +  \trinl h_{\cT} I_{\rm M} u \trinr_{\pw} + \left\{
\begin{array}{c l}
{}0 & \text{for } S = JI_{\rm M},\\
{}{{h_{\rm max}^{1-t}}} & \text{for } S = \mathrm{id} \text{ or } I_{\rm M}.
\end{array}
\right. 
\end{align}
{
Moreover, if $u \in V \cap H^{4-r}(\Omega)$ with $F \in H^{-r}(\Omega)$ for $2-\sigma \le r, s \le 2$, then 
\begin{align}
(b) & \;  \|u - u_{h}\|_{H^s(\cT)}   \lesssim \| u - u_h \|_{\rm P} ({{h_{\rm max}^{2-s}}} + \|u - u_h \|_{\rm P}) + 
\left\{ \begin{array}{c l}
0 & \text{with } S = JI_{\rm M},\\
{{h_{\rm max}^{3-t-s}}} & \text{for } S = \mathrm{id} \text{ or } I_{\rm M} 
\end{array} 
 \right.
 \text{ for } R :=JI_\M.
\end{align} 
\begin{align}
(c) &\;  \|u - u_{h}\|_{H^s(\cT)}   \lesssim \| u - u_h \|_{\rm P} ({{h_{\rm max}^{\min\{2-s,1-t\}}}} + \|u - u_h \|_{\rm P}) + 
\left\{ \begin{array}{c l}
0 & \text{for } S = JI_{\rm M},\\
{{h_{\rm max}^{3-t-s}}} & \text{for } S = \mathrm{id} \text{ or } I_{\rm M}
\end{array}
 \right. 
 \text{ for } R := I_\M.
\end{align}
(d) For $\sigma < 1 $, whence $1<s<2$, and the WOPSIP scheme with $R \in \{ I_\M, JI_\M \}$,
  \[
  \| u -  u_h \|_{H^s(\cT)}   \lesssim  \| u - u_{h} \|_{\rm P} \left( h_{\rm max}^{2-s}+\| u - u_{h} \|_{\rm P} \right) +\begin{cases}
		0 \textrm{ for }\quad S = \smooth,\\
		{{h_{\rm max}^{4-2s}}} \textrm{ for }\quad S = {\rm id} \text{ or } I_{\rm M}.
  \end{cases}
  \]
  (e)  If $F \in H^{-r}(\O)$ for some $ r<2$, then $(a)$-$(c)$ hold with $t=0$.}
\end{thm} 

\medskip

\noindent The subsequent lemma  extends~\ref{h1} in the 
analysis of the WOPSIP scheme.

\begin{lem}[variant of~\ref{h1}] \label{lemma:lem1_wopsip}
	There exists a constant $\lamw > 0$ such that any $v \in V$ and $v_2 \in P_2(\cT)$ satisfy 
	$a_{h}(I_{\rm M}v,v_2) - a(v,Qv_2) \le  \lamw\left(\trinl (1 - I_{\rm M}) v\trinr_{\pw} +  \trinl h_{\cT} I_{\rm M} v\trinr_{\pw} \right) \| v_2 \|_{\rm P}.$
\end{lem}
\begin{proof} Note that ${\mathsf c}_{h}(I_{\rm M}v,v_2) = 0$  for $v \in V$ and $v_2 \in P_2(\cT)$ from Table~\ref{tab:spaces} and the definition of $\M(\cT)$.  Utilize this in 
	$a_h(\bullet,\bullet) = a_{\rm pw}(\bullet,\bullet) + {\mathsf c}_h(\bullet,\bullet)$ to infer
	\begin{align} \label{eqn:apwsplitmain}
	a_{h}(I_{\rm M}v,v_2) - a(v,Qv_2) = a_{\rm pw}((I_{\rm M}-1)v ,v_2) + a_{\rm pw}(v,  (1- Q)v_2).
	\end{align}
		The boundedness of $a_{\rm pw}(\bullet,\bullet)$ and $\trinl \bullet \trinr_{\pw} \le \| \bullet \|_{\rm P}$ immediately imply 
		\begin{align} \label{eqn:apwvv2}
		a_{\rm pw}((1 - I_{\rm M})v,v_2) \le \trinl (1 - I_{\rm M})v \trinr_{\rm pw} \|v_{2} \|_{\rm P}.
		\end{align}
Since $a_{\rm pw}((1 - I_{\rm M})v,(1 - I_{\rm M})v_2) = 0 = a_{\rm pw}(I_{\rm M}v,(1 - J)I_{\rm M}v_2)$ from Lemma~\ref{lemma:interpoltion_IM}.c and Remark~\ref{ortho},
	\begin{align}
	a_{\rm pw}(v, (1 - Q)v_2) ={}& a_{\rm pw}(v,(1 - I_{\rm M})v_2) + a_{\rm pw}(v, (1 - J)I_{\rm M}v_2) \\
	={}& a_{\rm pw}(I_{\rm M}v,(1 - I_{\rm M})v_2) + a_{\rm pw}((1 - I_{\rm M})v,(1 - J)I_{\rm M}v_2) \\
	\le{}& \trinl h_{\cT} I_{\rm M}v \trinr_{\pw} \trinl h_{\cT}^{-1} (1 - I_{\rm M})v_2 \trinr_{\pw} + \trinl (1 - I_{\rm M})v \trinr_{\pw} \trinl (1 - J)I_{\rm M}v_2 \trinr_{\rm pw}. \label{eqn:apwvq}
	\end{align}
Since {Lemma~\ref{hctenrich}.g} provides $\trinl h_{\cT}^{-1} (1 - I_{\rm M})v_2 \trinr_{\pw} + \trinl (1 - J)I_{\rm M}v_2 \trinr_{\rm pw} \lesssim \|v_2 \|_{\rm P}$, this proves
\begin{equation}\label{eqn.wopsip}
a_{\rm pw}(v, (1- Q)v_2) \lesssim (\trinl h_{\cT} I_{\rm M}v \trinr_{\pw}  + \trinl (1 - I_{\rm M})v \trinr_{\pw})\|v_2 \|_{\rm P}.
\end{equation}
The combination of \eqref{eqn:apwsplitmain}-\eqref{eqn.wopsip} concludes the proof. 
\end{proof}

\noindent{\it  Proof of \ref{h3}-\ref{h5} for the WOPSIP scheme}.
\noindent {{{For a regular root $u \in V$}} to~\eqref{NS_weak} and any $\theta_h \in P_2(\cT)$ with $\|\theta_{h}\|_{\rm P} = 1$, {Lemma~\ref{lemma:global_bounds_nse_new}.b, $\trinl \bullet \trinr_{\pw} \le \|\bullet \|_{\rm P}$, and Lemma~\ref{lem:equivalence} }lead to  $\widehat{b}(R\theta_h,\bullet) \in H^{-1-t}(\Omega)$ for {$R\in\{{\rm id}, I_\M, J I_\M\}$}.  Therefore,  there exists a unique $\xi\equiv \xi(\theta_h) \in V\cap H^{3-t}(\Omega)$ with $\|\xi\|_{H^{3-t}(\O)} \lesssim 1$ such that $a(\xi,\phi) = \widehat{b}(R\theta_h,\phi)$ for all $\phi \in V$.   
Since $I_h={\rm id}$ and 
$\| \bullet \|_{\rm P}=\trinl \bullet \trinr_{\pw}$ in $V+\M(\cT)$ from \eqref{eqn:norm_wopsip}, Lemma~\ref{lemma:interpoltion_IM}.d leads to~\ref{h3} with $\delta_{2} =  \sup \{ \|\xi - I_hI_{\M} \xi \|_{\rm P}: \theta_h\in P_2(\cT), \|\theta_h\|_{\rm P}=1\}\lesssim h_{\rm max}^{1-t}$.  
\medskip

\noindent The proof of {\bf (H3)} starts as in \eqref{eqn.h3} and concludes $\delta_3\lesssim h_{\rm max}^{1-t}$ from 
$\|\bullet\|_h \lesssim \|\bullet\|_{\rm P}$  by Lemma~\ref{lem:equivalence}.

\medskip \noindent {The hypothesis \ref{h5} with $\delta_4 =\|u-x_h\|_{\rm P} < \epsilon$ follows from Remark~\ref{icimz}.}} \qed

\medskip

\noindent {\emph{Proof of discrete inf-sup condition.}}
The  proof of $\beta_0 \gtrsim  1$  in \eqref{eqn:disinfsup} follows also for  the WOPSIP scheme the above lines until~\eqref{eqn:split_axh} with $\xi := A^{-1}(\widehat{\nonlin}(\rop x_h,\bullet)|_Y)\in X$. {Recall that \eqref{dis_Ah_infsup} leads to $x_h +\xi_h \in P_2(\cT)$ and then to some $\phi_h \in P_2(\cT)$ with $\|\phi_h\|_{\rm P}=1$ and $\alpha_h \| x_h  + \xi_h \|_{\rm P} = a_{h}(x_h+\xi_h,\phi_h);$ this time $\epsilon=0$ can be neglected.} An alternative split reads
\begin{align} \label{eqn:wopsip_is_1}
\alpha_h \| x_h  + \xi_h \|_{\rm P} = a_{h}(x_h,\phi_h) + a_h(\xi_h,\phi_h)  - a(\xi,Q \phi_h) + a(\xi,Q\phi_h).
\end{align}
Lemma~\ref{lemma:lem1_wopsip}, $\xi_h  = I_{\rm M}\xi$, and $\trinl (1 - I_{\rm M}) \xi \trinr_{\pw}  \lesssim \delta_2{\lesssim h_{\rm max}^{1-t}}$ from~\ref{h3} provide
\begin{align} \label{eqn:wopsip_is_2}
a_h(\xi_h,\phi_h)  - a(\xi,Q \phi_h) \lesssim \delta_2 + \trinl h_{\cT} I_{\rm M} \xi\trinr_{\pw}.
\end{align}
The arguments in ~\eqref{infsup.t3} lead to $a(\xi,Q\phi_h) \le \widehat{\nonlin}(Rx_h,S\phi_h)+\delta_3$.  The combination of this with \eqref{eqn:wopsip_is_1}-\eqref{eqn:wopsip_is_2} provides
\begin{align} \label{eqn:nstep}
\hspace{-1cm} \| x_h  + \xi_h \|_{\rm P} {}\lesssim a_{h}(x_h,\phi_h) + \widehat{\nonlin}(\rop x_h,S\phi_h) + \delta_2 + \delta_{3} +\trinl h_{\cT} I_{\rm M} \xi\trinr_{\pw} .
\end{align} 
Replace~\eqref{eqn:pstep} by~\eqref{eqn:nstep} and apply the arguments thereafter to establish the stability condition~\eqref{eqn:disinfsup} with 
$\beta_0 : = \alpha_h \widehat{\beta} - (\lamw + \alpha_h) \delta_2 - \delta_{3} - \lamw \trinl h_{\cT} I_{\rm M} \xi\trinr_{\pw}$ for some $\lamw \lesssim 1.$
\qed

\medskip
\noindent 
{\emph{Proof of 
existence and uniqueness of the discrete solution.}} The analysis follows the proof of Theorem~\ref{thm:existence} verbatim  until \eqref{nonlin_exp}.   Instead of  \ref{h1},  Lemma~\ref{lemma:lem1_wopsip} and $x_h = I_{\rm M}u$ in~\ref{h5} control the first two terms on the right-hand side of~\eqref{nonlin_exp}, namely
\begin{align}
a_h(x_h,y_h) - a(u,Qy_h) \le  \lamw (\delta_4 + \trinl h_{\cT} I_{\rm M} u \trinr_{\pw}).
\end{align}
The remaining steps follow those of the proof of Theorem~\ref{thm:existence} with \eqref{defn_delta} replaced by 
\begin{align}
\epsilon_0 &:= \beta_1^{-1} \big( (  \lamw + (1 + \lamr)(\|R\|\|S\| \trinl I_{\rm M}u \trinr_{\rm pw} + \|Q\| \|u \|_{X}) \|\widehat{\Gamma}\| )   \delta_4 \\
& \qquad + \lamw \trinl h_{\cT} I_{\rm M} u \trinr_{\pw} + \trinl I_{\rm M}u \trinr_{\rm pw} \delta_{3}/2 \big).  \qquad \qquad \qquad \qquad \qquad \qquad \qquad \qquad \qquad \qed
\end{align}

%

\medskip
\noindent{\it Proof of Theorem~\ref{error control}.a.} Recall from Lemma~\ref{lemma:quas_eq} that  $u^\ast \in X$ and $G(\bullet)=a(u^*,\bullet) \in Y^*$, $u_h^\ast \in X_h$ and $a_h(u_h^*,\bullet) =G(Q \bullet)\in Y_h^*$.  In the proof of Lemma~\ref{lemma:quas_eq}, set $x_h := I_{\rm M}u^\ast$ so that Lemma~\ref{lemma:lem1_wopsip} implies 
\begin{align}
\alpha_0 \| e_h \|_{\rm P} \le a_h(x_h,y_h) - a(u^\ast,Qy_h) \le \lamw (\trinl u^\ast - I_{\rm M} u^\ast \trinr_{\pw} +  \trinl h_{\cT} I_{\rm M} u^\ast \trinr_{\pw}).
\end{align}
Therefore,  $u^\ast$ and $u_ h^\ast$ in Lemma~\ref{lemma:quas_eq} satisfy ${\| u^\ast - u_h^\ast \|_{\rm P} \le C_{\rm qo}' \trinl u^\ast  - I_{\rm M} u^\ast \trinr_{\rm pw} +\alpha_0^{-1} \lamw \trinl h_{\cT} I_{\rm M} u^\ast \trinr_{\pw}}$ for $C_{\rm qo}' = 1 + \alpha_0^{-1} \lamw$.  

\medskip \noindent 
 The  hypotheses \eqref{quasioptimalsmootherP}-\eqref{quasioptimalsmootherS} follow from Lemma~\ref{lem:R}; \ref{h3}-\ref{h5} are already verified. The error estimate in  Lemma~\ref{lemma:quas_eq} applies to Theorem~\ref{thm:apriori} with $x_{h} = I_{\rm M}u$ and $\| \bullet \|_{\rm P} = \trinl \bullet \trinr_{\rm pw}$ in $V + \mathrm{M}(\cT)$ and establishes
\begin{align}
\|u - u_h \|_{\rm P} \lesssim \trinl u  - I_{\rm M} u \trinr_{\rm pw} +  \trinl h_{\cT} I_{\rm M} u \trinr_{\pw} + \| \widehat{\Gamma}(u,u, (S - Q)\bullet) \|_{Y_h^\ast}
\end{align}
For {$u \in V$}, the last displayed estimate, {Lemma~\ref{lemma:gamma_approx_new}.a with $v=0$ for $S={\rm id}$ (resp. with $v_2 \in \M(\cT)$ for $S=I_{\rm M}$), Lemma~\ref{lem:equivalence}}, and the boundedness of $I_{\rm M}$ conclude the proof. 

\medskip \noindent {\it Proof of Theorem~\ref{error control}.b.}
{A triangle inequality leads to\begin{equation}\label{eqn.trianglewopsip}
 \|u-u_h\|_{H^s(\cT)} \le \|u-Pu_h \|_{H^s(\cT)} + \|Pu_h-u_h\|_{H^s(\cT)}=G(u-Pu_h) + \|Pu_h-u_h\|_{H^s(\cT)}
 \end{equation}
with $G(u-Pu_h)=\|u-Pu_h\|_{H^s(\cT)}$owing to a corollary of the Hahn-Banach theorem as in the proof of Theorem~\ref{cor:lower} in the last step. }Since $z \in Y$ solves~\eqref{dual_linear}, elementary algebra with~\eqref{eqn:p}-\eqref{eqn:dp} and $z_h := I_{\rm M}z \in Y_h$ lead to an alternative identity in place of \eqref{est}, namely
\begin{align} 
G(u - Pu_h) ={}& (a+b)(u - Pu_h,z)  ={} a(u,z - Qz_h) + a_{\rm pw}(u_h - Pu_h,z) + b(u - Pu_h,z- Qz_h) \\ \label{eqn:goal_wopsip_1}
&+ b(u - Pu_h,Qz_h) + {\Gamma}_{\pw}(Ru_h,Ru_h,Sz_h) - \Gamma(u,u,Qz_h)
\end{align}
with $a_h(u_h,z_h)  = a_{\rm pw}(u_h,z)$ from Lemma~\ref{lemma:interpoltion_IM}.c in the last step. Since $a_{\rm pw}(I_{\rm M}u, z-  Qz_h) = 0$ from Lemma~\ref{lemma:interpoltion_IM}.c and Remark~\ref{ortho},
\begin{align} \label{gw_1}
a(u,z-  Qz_h)  = a_{\rm pw}(u - I_{\rm M}u,z - Qz_h) \le (1 + \Lambda_{\rm Q}) \trinl u - I_{\rm M}u \trinl_{\rm pw} \trinl z - z_h \trinr_{\pw}
\end{align}
with boundedness of $a_{\rm pw}(\bullet,\bullet)$ and~\eqref{eqn:QS} in the last step.  A triangle inequality shows that 
\begin{align} \label{gw_2}
\trinl u - I_{\rm M} u \trinr_{\rm pw} \le \trinl u - u_h \trinr_{\rm pw} + \trinl u_h - I_{\rm M}u_h \trinr_{\pw} + \trinl I_{\rm M}(u -  u_h) \trinr_{\pw} \lesssim \|u - u_h\|_{\rm P}
\end{align}
with $\trinl \bullet  \trinr_{\pw} \le \| \bullet \|_{\rm P}$, $\| (1- I_{\rm M})u_h \|_{\rm P} \le \lamr \| u - u_h \|_{\rm P}$ from Lemma~\ref{lem:R}, and  $\trinl I_{\rm M}(u - u_h) \trinr_{\pw} \le \trinl u - u_h \trinr_{\pw}$ in the last step.
Arguments analogous to~\eqref{eqn.wopsip} and Lemma~\ref{hctenrich}.g} with $v = u$ lead to~\noeqref{gw_2} 
\begin{align} \label{gw_3}
a_{\rm pw}(u_h - Pu_h,z) \lesssim ( \trinl h_{\cT} I_{\rm M}z \trinr_{\pw} + \trinl (1- I_{\rm M})z \trinr_{\pw}) \|u - u_h \|_{\rm P}.
\end{align}
The combination of~\eqref{eqn:goal_wopsip_1}-\eqref{gw_3} and the estimates for the remaining terms in the right-hand side of~\eqref{eqn:goal_wopsip_1} from the last part (after \eqref{eqn:non-linear_splitting}) of the proof of Theorem~\ref{thm:lower} result in
\begin{align}\label{g.wopsip}
G(u - Pu_h) \lesssim{}& \|u - u_{h}\|_{\rm P}(\trinl z - z_h \trinr_{\pw} + \trinl h_{\cT} z_h \trinr_{\pw} + \| u - u_{h} \|_{\rm P})+ {\Gamma}_{\pw}(u,u, (S- Q)z_h)  \nonumber \\
&+ {\Gamma}_{\pw}(Ru_h,Ru_h,Qz_h) - \Gamma(Pu_h,Pu_h,Qz_h).
\end{align}
{Since $z_h=I_{\M}z$, Lemma~\ref{lemma:interpoltion_IM}.d provides
$\trinl z-z_h\trinr_{\pw} \lesssim h_{\max}^{2-s}$ and  $\trinl h_{\cT} z_h \trinr_{\pw}  \lesssim h_{\max}$.
Lemma~\ref{hctenrich}.f and $\|\bullet\|_h \lesssim \|\bullet\|_{\rm P}$ (by Lemma~\ref{lem:equivalence}) establish
$\|Pu_h -u_h\|_{H^s(\cT)}\lesssim h_{\max}^{2-s}\|u-u_{\rm P}\|_{\rm P}.$
The combination of those estimates with \eqref{eqn.trianglewopsip} and \eqref{g.wopsip} reveals
\begin{align}\label{g.wopsip1}
\|u-u_h\|_{H^s(\cT)} \lesssim{}& \|u - u_{h}\|_{\rm P}(h_{\max}^{2-s}+\|u - u_{h}\|_{\rm P}) + {\Gamma}_{\pw}(u,u, (S- Q)z_h) \nonumber \\
&+ {\Gamma}_{\pw}(Ru_h,Ru_h,Qz_h) - \Gamma(Pu_h,Pu_h,Qz_h).
\end{align}
The last three terms in the above inequality can be estimated as in the proof of Theorem~\ref{thm:apost_ns}.a {with $\|\bullet\|_h \lesssim \|\bullet\|_{\rm P}$ (by Lemma~\ref{lem:equivalence})} and this concludes the proof.} \qed

\medskip

 {\noindent {\it Proof of Theorem~\ref{error control}.c.} The arguments in $(b)$ and Theorem~\ref{thm:apost_ns}.b establish $(c)$. \qed

 \medskip

\noindent {\it Proof of Theorem~\ref{error control}.d.}
The choice $t:=s-1$ 
  in $(b)$-$(c)$
  concludes the proof. \qed 

  \medskip
\noindent {\it Proof of Theorem~\ref{error control}.e}.
\noindent For $F \in H^{-r}(\O)$ with $r<2$, the a priori error estimates can be established with $t=0$ by a substitution of the  assertions in {Lemma~\ref{lemma:gamma_approx_new}.a,c by Lemma~\ref{lemma:gamma_approx_new}.b,d}.

\section{Application to von K\'{a}rm\'{a}n equations }\label{sec:vke}
This section verifies~\ref{h1}-\ref{h5}and~\ref{h1hat}, and establishes~{\bf (A)}-{\bf(C)} for the~\vket. Subsection~\ref{vke:model_intro} and \ref{sec:vke_disc} present the problem and four  discretizations; the a priori error control for the Morley/dG/$C^0$IP/WOPSIP schemes follows in Subsection~\ref{sec:vke_apriori}-\ref{sub:proof}.  
\subsection{Von K\'{a}rm\'{a}n equations}
\label{vke:model_intro}
The \vket~in a polygonal domain $\Omega \subset \mathbb{R}^2$ seek $(u,v) \in H^2_0(\O) \times H^2_0(\O) = V \times V =: \bV $\;
such that
\begin{align}
\Delta^2 u =[u,v]+ f \;\;\text{ and }\;\;\Delta^2 v =-\half[u,u] 
\;\;\text{ in }\; \Omega.
\label{vkedGb}
\end{align}
\noindent The von K\'{a}rm\'{a}n bracket $[\bullet
,\bullet]$ above
 is defined by $\displaystyle
[\eta,\chi]:=\eta_{xx}\chi_{yy}+\eta_{yy}\chi_{xx}-2\eta_{xy}\chi_{xy}$ for all  $\eta,\chi \in V$.
The weak formulation of~\eqref{vkedGb} seeks  $u,v\in \: V $ that satisfy for all $(\varphi_{1},\varphi_{2}) \in \bv$
\begin{align} \label{vk_weak}
a(u,\varphi_1)+ \gamma(u,v,\varphi_1) + \gamma(v,u,\varphi_1) = f(\varphi_1) \;\text{ and }\;
a(v,\varphi_2) - \gamma(u,u,\varphi_2)   = 0        
\end{align} 
with  $\displaystyle \gamma(\eta,\chi,\varphi):=-\half\integ [\eta,\chi]\varphi\dx \text{ for all } \eta,\chi, \varphi  \in V$  and $a(\bullet,\bullet)$ from~\eqref{abform}.

\medskip \noindent   For all $\Xi=(\xi_1,\xi_2),\Theta=(\theta_1,\theta_2),$ and $\Phi=(\varphi_1,\varphi_2)\in  \bv$, define the {{forms}}
\begin{gather*}
{A}(\Theta,\Phi):={} a(\theta_1,\varphi_1) + a(\theta_2,\varphi_2), \\
\Gamma(\Xi,\Theta,\Phi):={} \gamma(\xi_1,\theta_2,\varphi_1)+\gamma(\xi_2,\theta_1,\varphi_1)-\gamma(\xi_1,\theta_1,\varphi_2),\;\; \text{ and }\;\; F(\Phi) := {{f(\varphi_{1})}}.
\end{gather*}
Then the vectorised formulation of~\eqref{vk_weak} seeks $\Psi=(u,v)\in \bv$ such that
\begin{equation}\label{VKE_weak}
N(\Psi;\Phi):={A}(\Psi,\Phi)+\Gamma(\Psi,\Psi,\Phi)- {F}(\Phi)=0\fl \Phi \in \bv.
\end{equation}
The trilinear form $\Gamma(\bullet,\bullet,\bullet)$ inherits symmetry in the first two variables from $\gamma(\bullet,\bullet,\bullet)$. The following boundedness and ellipticity properties hold \cite{ng1,Brezzi, CiarletPlates}
\begin{align}
&{A}(\Theta,\Phi)\leq \trinl\Theta\trinr \: \trinl\Phi\trinr,\: \trinl\Theta\trinr^2 \le {A}(\Theta,\Theta), \text{ and } \Gamma(\Xi, \Theta, \Phi) \lesssim \trinl\Xi\trinr \: \trinl\Theta\trinr \: \trinl\Phi\trinr. \label{eqn:a_vk1}
\end{align}
\subsection{Four quadratic discretizations} \label{sec:vke_disc}
This subsection presents the Morley/dG/$C^0$IP/WOPSIP schemes for~\eqref{VKE_weak}. The spaces and operators employed in the analysis of the von K\'{a}rm\'{a}n equations  given in Table \ref{tab:spaces_vke}  are vectorised versions (denoted in boldface) of those presented in Table~\ref{tab:spaces}, e.g.,   $\boldsymbol{I}_{\rm M} = I_{\rm M} \times I_{\rm M}
$.  
Recall $a_{\rm pw}(\bullet,\bullet)$ from~\eqref{eqccnerwandlast1234a}  and 
define the  bilinear form $a_h : (\bv_h + \mathbf{M}(\cT)) \times (\bv_h + \mathbf{M}(\cT)) \rightarrow \mathbb{R}$ by
\begin{align}
 a_{h}(\Theta,\Phi)
:={}& a_{\rm pw}(\theta_1,\varphi_1)+ \mathsf{b}_h(\theta_1,\varphi_1) +  \mathsf{c}_h(\theta_1,\varphi_1)  \\
&+ a_{\rm pw}(\theta_2,\varphi_2) +  \mathsf{b}_h(\theta_2,\varphi_2) + \mathsf{c}_h(\theta_2,\varphi_2).
\;   \label{eqn:apw_vke}
\end{align} 
The definitions of $ \mathsf{b}_h$ and $\mathsf{c}_h$ for the Morley/dG/$C^0$IP/WOPSIP schemes from  Table~\ref{tab:spaces} are omitted in Table~\ref{tab:spaces_vke} for brevity.  For all  $\eta, \chi, \varphi \in H^2(\cT),$ let $\gamma_{\rm pw}(\bullet,\bullet,\bullet)$ be the piecewise trilinear form  defined by  
\[\displaystyle \gamma_{\rm pw}(\eta, \chi,\varphi)
:= -\half\sum_{K \in \cT}   \int_{K} [\eta,\chi]\varphi\dx\] and, for all 
 $\Xi=(\xi_1,\xi_2),\Theta=(\theta_1,\theta_2),  \Phi=(\varphi_1,\varphi_2) \in {\bf H}^{2}(\cT)$, let 
\begin{equation}\label{eqn:gamma_pw_vke_vec}
\widehat{\Gamma}(\Xi,\Theta,\Phi) := \Gamma_{\pw}(\Xi,\Theta,\Phi):= \gamma_{\pw}(\xi_1,\theta_2,\varphi_1)+\gamma_{\pw}(\xi_2,\theta_1,\varphi_1)-\gamma_{\pw}(\xi_1,\theta_1,\varphi_2).
\end{equation}
For all the schemes and a regular root $\Psi \in \bV$  to~\eqref{VKE_weak}, let $\widehat{b}(\bullet,\bullet) := 2\Gamma_{\pw}(\Psi,\bullet,\bullet)$ in \eqref{star}. For  $R, S \in \{\mathbf{id}, \boldsymbol{I}_{\rm M}, \boldsymbol{JI}_{\rm M} \}$, the discrete scheme seeks a root ${\boldmath\Psi}_{h} := (u_{h},v_{h})\in \bv_h$ to 
\begin{align}\label{eqn:gen_scheme_vke}
\hspace{-0.7cm}\boldsymbol{N}_h(\Psi_h;\Phi_h) := a_{h}({\boldmath\Psi}_{h},{\boldmath\Phi}_{h})
+\Gamma_{\text{pw}}(R{\boldmath\Psi}_{h},R{\boldmath\Psi}_{h},S{\boldmath\Phi}_{h})-{F}( \boldsymbol{J}\boldsymbol{I}_{\rm M}{\boldmath\Phi}_{h})=0 \fl {\boldmath\Phi}_{h} \in  \bv_h.
\end{align}

\begin{table}[H]
	\centering
	\begin{tabular}{|c|c|c|c|c|}
		\hline
		Scheme     & Morley & dG        & $C^0$IP     & WOPSIP  \\ \hline
		$X_h = Y_h = \bV_h$          &  $\mathbf{M}(\cT)$      & $\boldsymbol{P}_2(\cT)$      & $\boldsymbol{S}^2_0(\cT)$ & $\boldsymbol{P}_2(\cT)$      \\ \hline
		\begin{minipage}{2cm}
			\centering
			$\widehat{X} = \widehat{Y} = \widehat{\bV}  =\bV + \bV_h$
		\end{minipage}   &   $\bV + \mathbf{M}(\cT)$     & $ \bV + \Stb$          & $ \bV + \boldsymbol{S}^2_0(\cT)$  & $ \bV +  \Stb$     \\ \hline
		$\|\bullet \|_{\widehat{X}}$            & $\trinl \bullet \trinr_{\rm pw}$     & $\| \bullet \|_{\rm dG}$       &    $\| \bullet \|_{\rm IP}$  &  $\| \bullet \|_{\rm P}$     \\ \hline
		$P=Q$               & ${\boldsymbol{J}}$     & $\boldsymbol{J}\boldsymbol{I}_{\rm M}$ &  $\boldsymbol{J}\boldsymbol{I}_{\rm M}$ & $\boldsymbol{J}\boldsymbol{I}_{\rm M}$          \\ \hline
		$I_{h}$ & ${\mathbf {id}}$ & ${\mathbf {id}}$ &$\boldsymbol{I}_{\rm C}$  & ${\mathbf  {id}}$  \\ \hline 
			$I_{X_h}=I_{\bV_h} = I_{h}\boldsymbol{I}_{\rm M}$           &  $\boldsymbol{I}_{\rm M}$ & $\boldsymbol{I}_{\rm M}$ & $\boldsymbol{I}_{\rm C}\boldsymbol{I}_{\rm M}$ & $\boldsymbol{I}_{\rm M}$    \\ \hline
%
	\end{tabular}
	\caption{Spaces, operators, and norms in Section \ref{sec:vke}.}
	\label{tab:spaces_vke}
\end{table}
\subsection{Main results} \label{sec:vke_apriori}
The main results on a priori error control in energy and weaker Sobolev norms for the Morley/dG/$C^0$IP/ WOPSIP schemes of Subsection~\ref{sec:vke_disc} are stated in this and verified in the subsequent subsections. {Unless stated otherwise,  $R \in \{\mathbf{id}, \boldsymbol{I}_\M, \boldsymbol{JI}_\M\}$ is arbitrary.}

\begin{thm}[A priori energy norm error control] \label{thm:error_apriori_energy_vke}
Given a regular root {{$\Psi \in {\bV}$ to \eqref{VKE_weak}  with $F \in {\bf H}^{-2}(\Omega)$}},  there exist $\epsilon, \delta > 0$ such that, for any $\displaystyle\cT\in\bT(\delta)$, the unique discrete solution $\Psi_h  \in {\bV}_h$ to~\eqref{eqn:gen_scheme_vke} 
with 
$\| \Psi-\Psi_h\|_{h} \le \epsilon$ for the Morley/dG/$C^0$IP schemes 
satisfies
	\begin{align*}  
	\| \Psi - \Psi_h \|_{h} \lesssim{}&  \min_{ \Psi_h \in \bV_h} \|\Psi - \Psi_h \|_{h}+ 
	\begin{cases}
	0 \; \text{ for  }S=\boldsymbol{JI}_\M,\\
	h_{\rm max} \; \text{ for }S=\mathbf{id} \text{ or } \boldsymbol{I}_\M.
	\end{cases}
	\end{align*}
\end{thm}

\noindent The a priori estimates in Table~\ref{tab:apriori} {{hold}} for~\vket~component-wise for ${F} \in {\bf H}^{-r}(\Omega)$, {{$2 - \sigma \le r \le 2$}} and $\Psi \in \bV \cap {\bf H}^{4-r}(\Omega)$.
\begin{rem}[Comparison]
	Suppose {{$\Psi \in \bV$ is a regular root to \eqref{VKE_weak} with  ${F} \in {\bf H}^{-2}(\Omega)$}} and $S= {\boldsymbol{J}} \bI_\M$. If $h_{\rm max}$ is sufficiently small, then the respective local discrete solutions $\Psi_\M, \Psi_{\dg}, \Psi_{\rm IP} \in \bV_h$ to \eqref{eqn:gen_scheme_vke} 
	for the Morley/dG/$C^0$IP  schemes  satisfy
	\begin{align*}
		\|\Psi-\Psi_\M\|_h
		\approx
		\|\Psi-\Psi_\dg\|_h
		\approx
		\|\Psi-\Psi_\ip\|_h
		\approx \|(1-\Pi_0) D^2 \Psi\|_{\bL^2(\Omega)}.  \qquad \qed 
	\end{align*}
\end{rem}
\begin{thm}[a priori error control in weaker norms] \label{thm:aprio_weak_vke}
Given a regular root  $\Psi \in {\bV}\cap {\bf H}^{4-r}(\Omega)$ to \eqref{VKE_weak}  with ${F} \in {\bf H}^{-r}(\Omega)$  for {{$2-\sigma \le r,s \le 2$}},  there exist $\epsilon, \delta > 0$ such that, for any $\displaystyle\cT\in\bT(\delta)$, the unique discrete solution $\Psi_h  \in {\bV}_h$ to~\eqref{eqn:gen_scheme_vke} 
with 
$\| \Psi-\Psi_h\|_{h} \le \epsilon$ satisfies
	\begin{align*}  
\hspace{-0.3in} \; \| \Psi - \Psi_h \|_{{\bf H}^s(\cT)} \lesssim{}&  \|\Psi - \Psi_h \|_{h} \big(h_{\rm max}^{2-s} + \| \Psi - \Psi_h \|_{h}  \big)+ 
\begin{cases}
0 \; \text{ for  }S=\boldsymbol{JI}_\M,\\
h_{\rm max}^{3-s} \; \text{ for }S=\mathbf{id} \text{ or } \boldsymbol{I}_\M
\end{cases}
\end{align*}
\noindent 
(a) for the  Morley/dG/$C^0$IP  schemes and $R =\{{\boldsymbol{J}} \bI_\M, \bI_\M\}$ and 
(b) for the Morley scheme and $R = \mathbf{id}.$

\end{thm}

\begin{thm}[a priori WOPSIP] \label{thm:aprio_weak_vke_WOPSIP}
{{Given a regular root  $\Psi \in {\bV}$ to \eqref{VKE_weak}  with ${F} \in {\bf H}^{-2}(\Omega)$}},  there exist $\epsilon, \delta > 0$ such that, for any $\displaystyle\cT\in\bT(\delta)$, the unique discrete solution $\Psi_h  \in {\bV}_h$ to~\eqref{eqn:gen_scheme_vke} 
with 
$\| \Psi-\Psi_h\|_{\rm P} \le \epsilon$  for the WOPSIP scheme satisfies
\begin{align*}  
	(a) \; \| \Psi - \Psi_h \|_{\rm P} \lesssim{}&   \trinl \Psi -\boldsymbol{I}_{\rm M} \Psi \trinr_{\rm pw}+  \trinl  h_\cT  \boldsymbol{I}_\M \Psi  \trinr_{\rm pw} +
	\begin{cases}
	0 \; \text{ for  }S=\boldsymbol{JI}_\M,\\
	h_{\rm max}  \; \text{ for }S=\mathbf{id} \text{ or } \boldsymbol{I}_\M.
	\end{cases}
\hspace{-0.2in}\;
	\end{align*}\noindent
\hspace{-0.05in} 
{{Moreover, if ${F} \in {\bf H}^{-r}(\Omega)$ for $2-\sigma \le r,s\le 2$ and $R\in \{{\boldsymbol{J}} \bI_\M,  \bI_\M\}$, then}}
\begin{align*}
(b) \;  \| \Psi - \Psi_h \|_{{{{\bf H}^{s}(\cT)}}} \lesssim{}&  \|\Psi - \Psi_h \|_{\rm P} \big({{h_{\rm max}^{2-s}}} + \| \Psi - \Psi_h \|_{\rm P}  \big)+ 
\begin{cases}
0 \; \text{ for  }S=\boldsymbol{JI}_\M,\\
{{h_{\rm max}^{3-s}}} \; \text{ for }S=\mathbf{id} \text{ or } \boldsymbol{I}_\M.
\end{cases}
\hspace{-0.1in}
\end{align*}
\end{thm}

\subsection{ Preliminaries}
 Two lemmas on the trilinear form $\Gamma_{\rm pw}(\bullet,\bullet,\bullet)$  from~\eqref{eqn:gamma_pw_vke_vec}  are crucial for the a priori error control. 

{{\begin{lem}[{boundedness}]\label{lemma:global_bounds_vke_new} For any $0<t<1$ there exists a constant $C(t)>0$ such that any $\widehat{\Phi}, \widehat{\boldsymbol{\chi}} \in \bV + \boldsymbol{P}_2(\cT)$, $\widehat{\Xi} \in \bV+\mathbf{M}(\cT)$, and $\Xi \in \bV$ satisfy
	\begin{align*}
	&(a)\; {\Gamma}_{\rm pw}(\widehat{\Phi},\widehat{\boldsymbol{\chi}}, \widehat{\Xi}) \lesssim  \trinl \widehat{\Phi} \trinr_{\rm pw}\trinl \widehat{\boldsymbol{\chi}} \trinr_{\rm pw} \trinl \widehat{\Xi}\trinr_{\rm pw}\;and\;
	(b)   \; {\Gamma}_{\rm pw}(\widehat{\Phi},\widehat{\boldsymbol{\chi}}, \Xi) \le C(t)  \trinl \widehat{\Phi} \trinr_{\rm pw}\trinl \widehat{\boldsymbol{\chi}} \trinr_{\rm pw} \| \Xi\|_{{\bf H}^{1+t}(\O)}.
		\end{align*}
\end{lem}
%


\noindent {\em Proof of (a).} The definition of $\gamma_{\pw}(\bullet,\bullet,\bullet)$, \Holder inequalities,  and 
$\| \bullet\|_{L^{\infty}(\Omega)} \lesssim \trinl \bullet \trinr_{\pw}$ in $V+\M(\cT)$ from 
\cite[Lemma 4.7]{CCGMNN_semilinear} establish, for 
$\widehat{\phi}$, $\widehat{\chi} \in V + P_2(\cT)$, $\widehat{\xi} \in V+\M(\cT)$, that
\begin{align*}
\gamma_{\pw}(\widehat{\phi}, \widehat{\chi},\widehat{\xi}) & \le \trinl \widehat{\phi} \trinr_{\pw} \trinl \widehat{\chi} \trinr_{\pw} \| \widehat{\xi} \|_{L^{\infty}(\Omega)} \lesssim \trinl \widehat{\phi} \trinr_{\pw} \trinl \widehat{\chi} \trinr_{\pw} \trinl \widehat{\xi}\trinr_{\pw}.  \qquad \qquad  
\end{align*}

\noindent {\em Proof of (b).}	For $\widehat{\phi}$, $\widehat{\chi} \in V + P_2(\cT)$ and $\xi \in V$, the definition of $\gamma_{\rm pw}(\bullet,\bullet,\bullet)$, \Holder inequalities, and the continuous Sobolev embedding $H^{1+t}(\Omega)\hookrightarrow L^\infty(\Omega)$  \cite[Corollary 9.15]{brezis} for $t>0$ show
	\begin{align} \label{new:trilinearestimate}
	\gamma_{\rm pw}(\widehat{\phi},\widehat{\chi},\xi) \le \trinl \widehat{\phi} \trinr_{\rm pw} \trinl \widehat{\chi} \trinr_{\pw} \|  \xi \|_{L^\infty(\O)} \lesssim \trinl \widehat{\phi} \trinr_{\rm pw} \trinl \widehat{\chi} \trinr_{\pw} \|  \xi \|_{H^{1+t}(\O)}. 
	\end{align}
	This and \eqref{eqn:gamma_pw_vke_vec} conclude the proof. \qed
}}
{{\begin{lem}[approximation] \label{lemma:gamma_approximation_new}
		Any $\widehat{\boldsymbol{\chi}} \in \bv + \boldsymbol{P}_2(\cT)$, ${\Phi},{\bf v}\in \bv$, and  $({\bf v}_2, {\bf v}_{\rm M}) \in \boldsymbol{P}_2(\cT)\times \mathbf{M}(\cT)$ satisfy
	\begin{itemize}
\item[(a)] $\Gamma_{\rm pw}(\Phi,\widehat{\boldsymbol{\chi}}, (1 - \boldsymbol{J}\boldsymbol{I}_{\rm M})  {\bf v}_2) \lesssim h_{\rm max} \trinl \Phi \trinr  \trinl \widehat{\boldsymbol{\chi}}\trinr_{\rm pw} \|{\bf v}  -{\bf v}_2\|_{h}$,
		\item[(b)] 	$\Gamma_{\pw}((1 - \boldsymbol{J}) {\bf v}_{\rm M}, {\bf v}_2, \Phi) \lesssim h_{\rm max} \trinl {\bf v} - {\bf v}_{\rm M} \trinr_{\pw} \trinl {\bf v}_2\trinr_{\pw}\trinl \Phi \trinr.$
	\end{itemize}
\end{lem}
		\noindent {\em Proof of (a).} For $\phi \in V $, $\widehat{\chi} \in V + P_2(\cT)$ and $ v_2 \in P_2(\cT)$,  the definition of $\gamma_{\pw}(\bullet,\bullet,\bullet)$, \Holder inequalities, and an inverse estimate $h_T\| (1 - JI_{\rm M}) v_2 \|_{L^{\infty}(T)}\lesssim \| (1 - JI_{\rm M}) v_2 \|_{L^{2}(T)}$ lead to 
	\begin{align}
	\gamma_{\pw}(\phi,\widehat{\chi},(1 - JI_{\rm M})v_2 ) \le \trinl \phi \trinr \trinl \widehat{\chi} \trinr_{\rm pw} \| (1 - JI_{\rm M}) v_2 \|_{L^{\infty}(\Omega)} \lesssim  \trinl \phi \trinr  \trinl \widehat{\chi} \trinr_{\rm pw}\| h_{\cT}^{-1}(1 - JI_{\rm M})v_2\|.
	\end{align}
	This, Lemma~\ref{hctenrich}.f, and the definition of $\Gamma_{\rm pw}(\bullet,\bullet,\bullet)$ conclude the proof of $(a)$. 
	

\medskip
\noindent {\em Proof of (b).}	For $\phi \in V$, $ v_2 \in P_2(\cT)$, and  $v_{\rm M} \in \mathrm{M}(\cT)$, an introduction of $\Pi_0\phi$  and $ \gamma_{\pw}((1 - J)v_{\rm M},v_2,\Pi_0 \phi)=0$ from Lemma~\ref{lemma:interpoltion_IM}.c and Remark~\ref{ortho} provide
	\begin{align} \label{eqn:split_of_gamma}
	\gamma_{\pw}((1 - J)v_{\rm M},v_2,\phi) =  \gamma_{\pw}((1 - J)v_{\rm M},v_2,\phi-  \Pi_0 \phi).
	\end{align}
\Holder inequalities and the estimate $\| \phi- \Pi_0 \phi\|_{L^{\infty}(\Omega)} \lesssim h_{\rm max}\trinl \phi\trinr$ \cite[Theorem 3.1.5]{Ciarlet} provide
 \begin{align}
   \gamma_{\pw}((1 - J)v_{\rm M},v_2,\phi-  \Pi_0 \phi) \lesssim h_{\max} \trinl	(1 - J)v_{\rm M} \trinr_{\pw} \trinl v_2 \trinr_{\rm pw} \trinl \phi \trinr\lesssim h_{\rm max} \trinl v - v_{\rm M} \trinr_{\rm pw}\trinl v_2 \trinr_{\pw} \trinl \phi \trinr  
 \end{align}
with $\trinl	(1 - J)v_{\rm M} \trinr_{\pw} \lesssim \trinl v - v_{\rm M} \trinr_{\rm pw}$ from Lemma~\ref{hctenrich}.e in the last step.
Recall~\eqref{eqn:gamma_pw_vke_vec} and \eqref{eqn:split_of_gamma} to conclude the proof of $(b)$.  \qed}}

\subsection{ Proof of Theorem \ref{thm:error_apriori_energy_vke}}
The conditions in Theorem~\ref{thm:apriori} are verified to establish the energy norm estimates. 
The hypotheses~\eqref{quasioptimalsmootherP}-\eqref{quasioptimalsmootherS} follow from Lemma~\ref{lem:R} (component-wise).  The paper~\cite{ccnnlower2022} has verified hypothesis~\ref{h1} for Morley/dG/$C^0$IP in the norm $\|\bullet
 \|_h$ that is equivalent to $\trinl \bullet \trinr_{\pw}$, $\|\bullet\|_{\rm dG}$, and $\|\bullet \|_{\rm IP}$ by Lemma~\ref{lem:equivalence}.
 

 \medskip
\noindent For any $\boldsymbol{\theta}_{h} \in \bV_h$ with $\|\boldsymbol{\theta}_{h}\|_{\bV_h} = 1$,  {{Lemma~\ref{lemma:global_bounds_vke_new}.b with $\trinl \bullet \trinr_{\rm pw} \le \| \bullet \|_{h}$ implies $\widehat{b}(R\boldsymbol{\theta}_{h},\bullet) \in {\bf H}^{-1-t}(\Omega)$ for $R \in \{\mathbf{id}, \boldsymbol{I}_\M, \boldsymbol{JI}_\M\}$. Therefore, there exists a unique $\boldsymbol{\chi} \in \bV\cap {\bf H}^{3-t}(\Omega)$ with $\|\boldsymbol{\chi} \|_{{\bf H}^{3-t}(\O)} \lesssim 1$ such that $A(\boldsymbol{\chi} ,\Phi) = \widehat{b}(R\boldsymbol{\theta}_h ,\Phi)$ for all $\Phi \in \bV$. Hence, for Morley/dG schemes (resp. $C^0$IP scheme), 
the boundedness of $R$ (from Lemma~\ref{lem:R}), Lemma~\ref{lem:equivalence} (resp. Remark~\ref{icimz}), and Lemma~\ref{lemma:interpoltion_IM}.d provide~\ref{h3} with $\delta_{2} \lesssim h_{\rm max}^{1-t}$.

\medskip 

\noindent {The proof of {\bf (H3)} starts as in Subsection~\ref{proof_of_8.1} and adopts
Lemma~\ref{lemma:gamma_approximation_new}.a (in place of Lemma~\ref{lemma:gamma_approx_new}.a) to establish \eqref{eqn.h3} with $t=0$ and the slightly sharper version $\delta_3 \lesssim h_{\rm max}$.}

\noindent Since $\delta_3=0$ for $S=Q=\boldsymbol{JI}_\M$, it remains $S=\mathbf{id}$ and $=\bI_{\rm M}$ in the sequel to establish {\bf (H3)}. Given $\boldsymbol{y}_{h}$ and $\boldsymbol{\theta}_{h} \in \bV_h$ of norm one, define ${\bf v}_2:=S\boldsymbol{y}_{h} \in \boldsymbol{P}_2(\cT)$ and observe $Q\boldsymbol{y}_h =\boldsymbol{JI}_\M \boldsymbol{y}_h=\boldsymbol{JI}_\M {\bf v}_2$ (by $S=\mathbf{id},\bI_{\rm M})$. Hence with the definition of $\widehat{b}(\bullet,\bullet)$, Lemma~\ref{lemma:gamma_approximation_new}.a shows
$$|\widehat{\nonlin}(R\boldsymbol{\theta}_h,(S - Q)\boldsymbol{y}_h)|=|\widehat{\nonlin}(R\boldsymbol{\theta}_h,{\bf v}_2-\boldsymbol{JI}_{\rm M}{\bf v}_2)|\lesssim h_{\max}\trinl u\trinr \trinl R\boldsymbol{\theta}_h\trinr_{\pw}\|{\bf v}_2\|_h.$$ 
The boundedness of $R$ and $\boldsymbol{I}_{\rm M}$ and the equivalence of norms show $\trinl R\boldsymbol{\theta}_h\trinr_{\pw}\|{\bf v}_2\|_h \lesssim 1$ and hence $\delta_{3} \lesssim h_{\rm max}$.}

\smallskip
\noindent 
\noindent{{As in the application for Navier-Stokes equations, Remark~\ref{icimz} leads to hypothesis \ref{h5} with $\delta_4 < \epsilon$.}} The existence and uniqueness of a discrete solution $\Psi_h$ then follows from Theorem \ref{thm:existence}.

\smallskip
\noindent 	
\noindent Note that for ${\bf v}_h \in {\bf M}(\cT)$,  $Q {\bf v}_h=\boldsymbol{JI}_\M {\bf v}_h$. For Morley/dG/$C^0$IP, {{Lemma~\ref{lemma:gamma_approximation_new}.a with ${\bf v}=0$  for $S=\mathbf{id}$; and Lemma~\ref{lemma:gamma_approximation_new}.a with ${\bf v}_2\in {\bf M}(\cT)$ and ${\bf v}=0$ for $
S=\bI_\M$ show}} 	 $$\displaystyle
\|\widehat{\Gamma}(\Psi,\Psi,(S - Q)\bullet)\|_{\bV_h^\ast} \lesssim 
\begin{cases}
0\textrm{ for } S=\boldsymbol{JI}_\M,\\
h_{\rm max}\text{ for } S=\mathbf{id} \text{ or } \bI_\M.
\end{cases}$$
The energy norm error control then follows from Theorem~\ref{thm:apriori}. \qed
\subsection{Proof of Theorem 9.3}\label{sub:proof}
\noindent Given $2-\sigma \le s \le 2$ and $G \in {\bf H}^{-s}(\Omega)$ with $\|G\|_{{\bf H}^{-s}(\Omega)}=1$ ,  the solution $z \in \bV$ to the dual problem \eqref{dual_linear} belongs to 
$ \bV \cap {\bf H}^{4-s}(\Omega)$ by  elliptic regularity. This and  Lemma~\ref{lemma:interpoltion_IM}.d verify
\begin{align} \label{eqn:regulardual_vke}
\trinl  z - \bI_\M z \trinr_{\pw}  &\lesssim h_{\rm max}^{2-s} \|z\|_{{\bf H}^{4-s}(\Omega)} 
\lesssim  h_{\rm max}^{2-s}.  
\end{align}
\noindent{\it Proof of Theorem~\ref{thm:aprio_weak_vke}.a for  $R= \boldsymbol{J} \boldsymbol{I}_\M$}.
The assumptions in Theorem~\ref{cor:lower} with $X_{s}:= {\bf H}^s(\T)$ are verified to establish the lower-order estimates.  Hypothesis~\ref{h1hat} for Morley/dG/$C^0$IP schemes is verified in~\cite[Lemma 6.6]{ccnnlower2022} for an equivalent norm (with Lemma~\ref{lem:equivalence}) and Lemma~\ref{lem:R} for $R = JI_{\rm M}$ (applied component-wise to vector functions).  The conditions~\eqref{quasioptimalsmootherP}-\eqref{quasioptimalsmootherS} follow from Lemma~\ref{lem:R}. 
	 In Theorem~\ref{cor:lower}, set $z_h = \bI_h \bI_{\rm M}z$ with $\bI_{h} = \mathbf{id}$ for Morley/dG resp. $\bI_h = \bI_{\rm C}$ for $C^0$IP.  Notice that~\eqref{eqn:regulardual_vke} implies
	\begin{align} \label{eqn:approx_uh_vke}
	\|z - z_h \|_{h}\lesssim h_{\rm max}^{2-s}
	\end{align}
	for Morley/dG with $\|\bullet \|_{\dg} \approx \trinl \bullet \trinr_{\rm pw}$ in $\bV+\mathbf{M}(\cT)$. Remark~\ref{icimz} and  \eqref{eqn:regulardual_vke} provide \eqref{eqn:approx_uh_vke} for $C^0$IP.  For Morley/dG/$C^0$IP,  Lemma~\ref{hctenrich}.f implies 
	$\|\Psi_h- P\Psi_h   \|_{{\bf H}^{s}(\cT)} \lesssim h_{\rm max}^{2-{s}} \|\Psi - \Psi_h\|_{h}. $  
 The difference $ {\Gamma}_{\pw}(R \Psi_h,R \Psi_h,Qz_h) - \Gamma(P \Psi_h,P\Psi_h, Qz_h)$ vanishes for $R=\boldsymbol{JI}_{\rm M}=P$ (for all schemes). It remains to control  the term $ \widehat{\Gamma}(\Psi,\Psi,(S-Q)z_h)$ for  $S \in \{ \mathbf{id}, \bI_{\rm M}, \boldsymbol{JI}_{\rm M} \}$.

	\medskip 
	\noindent For $S=Q=\boldsymbol{JI}_{\rm M}$, $\Gamma_{\pw}(\Psi,\Psi,(S-Q)z_h )=0$. {{For $S =\mathbf{id}$,  Lemma~\ref{lemma:gamma_approximation_new}.a and ~\eqref{eqn:approx_uh_vke} establish
	\begin{align*} 
	\Gamma_{\pw}(\Psi,\Psi,(1 - \boldsymbol{JI}_{\rm M})z_h ) \lesssim {h}_{\max} \trinl \Psi\trinr^2 \|z - z_h \|_{h} \lesssim h_{\rm max}^{3-s}.
	\end{align*}
	For $S  = \bI_{\rm M}$,  Lemma  \ref{lemma:gamma_approximation_new}.a applies to ${\bf v}_h=\bI_{\rm M}z_h$. A triangle inequality and Lemma~\ref{lem:R} reveal $\|z-\bI_{\rm M}z_h\|_h \lesssim \| z- z_h   \|_h\lesssim h_{\rm max}^{2-s}$ with~\eqref{eqn:approx_uh_vke} in the last step. Hence,
	 \[\hspace{2cm}\Gamma_{\pw}(\Psi,\Psi,(\bI_{\rm M} -\boldsymbol{JI}_{\rm M} )z_h) \lesssim h_{\rm max} \trinl \Psi \trinr^2 \| z- z_h   \|_{h} \lesssim  h_{\rm max}^{3-s}. \qquad   \qed\]
 }}
\noindent{\it Proof of Theorem~\ref{thm:aprio_weak_vke}.a for  $R= \boldsymbol{I}_\M$}.
{Elementary algebra and the symmetry of $\Gamma_{\rm pw}(\bullet,\bullet,\bullet)$ with respect to the first and second
argument recast the last two terms on the right-hand side of Theorem~\ref{cor:lower} as
\begin{align}
 \Gamma_{\rm pw}&(\bI_\M \Psi_{h},\bI_\M \Psi_{h},\boldsymbol{J}\bI_\M z_h)-\Gamma_{\rm pw}(\boldsymbol{J}\bI_\M \Psi_{h},\boldsymbol{J}\bI_\M \Psi_{h},\boldsymbol{J}\bI_\M z_h)  \\
=& 2\Gamma_{\rm pw}( (1 - \boldsymbol{J})\bI_\M \Psi_{h},\bI_\M \Psi_{h},\boldsymbol{J}\bI_\M z_h)  
-\Gamma_{\rm pw}((1 - \boldsymbol{J})\bI_\M \Psi_{h},(1-\boldsymbol{J})   \bI_\M \Psi_{h},\boldsymbol{J}\bI_\M z_h ).
\label{eqn.split2.vke}
\end{align}
The arguments in \eqref{eqn:key1}-\eqref{eqn:key3} for $(\Psi,\Psi_h)$ replacing $(u,u_h)$  and~\eqref{eqn:approx_uh_vke} reveal 
\begin{align*}
& \trinl  \Psi-\bI_\M \Psi_{h} \trinr_{\pw}  \lesssim \|  \Psi- \Psi_{h} \|_{h}\; \mbox{ and }\;
\;\trinl z- \boldsymbol{J}\bI_\M z_{h} \trinr_{\pw}  \lesssim h_{\rm max}^{2-s}.
\end{align*}
This and Lemma~\ref{lemma:gamma_approximation_new}.b for the first term in \eqref{eqn.split2.vke}  (resp. Lemma~\ref{lemma:global_bounds_vke_new}.a and~\ref{hctenrich}.e for the second) show
\begin{align}
\Gamma_{\rm pw}( (1 - \boldsymbol{J})\bI_\M \Psi_{h},\bI_\M \Psi_{h},\boldsymbol{J}\bI_\M z_h)  &\lesssim h_{\rm max}\| \Psi - \Psi_{h} \|_h\\
\Gamma_{\rm pw}((1 - \boldsymbol{J})\bI_\M \Psi_{h},(1 - \boldsymbol{J})\bI_\M \Psi_{h},\boldsymbol{J}\bI_\M z_h )&\lesssim \trinl (1-\boldsymbol{J})   \bI_\M \Psi_{h}\trinr_{\pw}^2 
\lesssim \| \Psi - \Psi_{h} \|_h^2.
\end{align}
\noindent This leads  in~\eqref{eqn.split2.vke} to
	\begin{align}
&\Gamma_{\rm pw}(\bI_{\M} \Psi_h,\bI_{\M} \Psi_h,\boldsymbol{J} \bI_{\rm M} z_h) -\Gamma_{\rm pw}(\boldsymbol{J}\bI_{\M} \Psi_h,\boldsymbol{J}\bI_{\M} \Psi_h,\boldsymbol{J} \bI_{\M} z_h)   \\&\qquad \qquad \qquad\lesssim \| \Psi - \Psi_{h} \|_{h}(h_{\max} +\| \Psi - \Psi_{h} \|_{h}).\label{eqn:gamma_split_1_vke1}
	\end{align}
The remaining terms are controlled as in the above case $R= \boldsymbol{J}\bI_\M$. This concludes the proof. \qed
}

\medskip
\noindent{\it Proof of Theorem~\ref{thm:aprio_weak_vke}.b.} Since $ \Psi_{h} = \bI_\M \Psi_{\rm M}$, and $P=Q=\boldsymbol{J}$ for the Morley FEM, the last two terms of Theorem~\ref{cor:lower} read
${\Gamma}_{\pw}( \Psi_{\rm M},\Psi_{\rm M},\boldsymbol{J} \bI_\M z_h) -{\Gamma}( \boldsymbol{J}\Psi_{\rm M}, \boldsymbol{J}\Psi_{\rm M},\boldsymbol{J}\bI_\M z_h)$ and are controlled in~\eqref{eqn:gamma_split_1_vke1}.  This, Theorem~\ref{cor:lower}, and the above estimates from the proof for $R=\boldsymbol{J}\bI_\M$ in $(a)$ conclude the proof.   \qed

\medskip
\noindent {\it Proof of Theorem~\ref{thm:aprio_weak_vke_WOPSIP}}. The proofs at the abstract level in Section \ref{sec.stability}-\ref{sec:goal-oriented} follow as further explained for  the Navier Stokes equations. A straightforward adoption of the arguments provided in  the proofs of Theorem~\ref{thm:error_apriori_energy_vke} and \ref{thm:aprio_weak_vke}.a lead to \ref{h3}-\ref{h5} and the a priori error control.

\qed
%
\subsection*{Acknowledgements}
The finalization of this paper is supported by the SPARC project 
(id 235) {\it mathematics and computation of plates} and the SERB POWER Fellowship SPF/2020/000019.  The second author  acknowledges the support by the IFCAM project “Analysis, Control and Homogenization of Complex Systems”.
\bibliographystyle{amsplain}
\bibliography{NSEBib}
\end{document}